\documentclass[12pt]{amsart}
\usepackage{graphicx}
\usepackage{hyperref}
\usepackage{setspace}
\usepackage{enumerate}
\usepackage{amsmath,amsfonts,amssymb,amscd}

\newtheorem{thm}{Theorem}[section]
\newtheorem{cor}[thm]{Corollary}
\newtheorem{lem}[thm]{Lemma}
\newtheorem{prop}[thm]{Proposition}

\theoremstyle{definition}
\newtheorem{defn}[thm]{Definition}

\theoremstyle{remark}

\numberwithin{equation}{section}

\begin{document}

\title[Prime solutions to polynomial equations]{Prime solutions to polynomial equations \\ in many variables and differing degrees}

\author{Shuntaro Yamagishi}
\address{Department of Mathematics \& Statistics \\
Queen's University \\
Kingston, ON\\  K7L 3N6 \\
Canada}
\email{sy46@queensu.ca}
\indent

\date{Revised on \today}

\begin{abstract}
Let $\mathbf{f} = (f_1, \ldots, f_R)$ be a system of polynomials with integer coefficients in which the degrees
need not all be the same. We provide sufficient conditions for which the system of equations $f_i (x_1, \ldots, x_n) = 0 \ (1 \leq i \leq R)$
has a solution with every coordinate a prime number. 
\end{abstract}



\maketitle

Overview. Let $\mathbf{f} = (\mathbf{f}_{d}, \ldots, \mathbf{f}_{1})$ be a system of polynomials in $\mathbb{Z}[x_1, \ldots, x_n]$, where $\mathbf{f}_{\ell} = ( f_{\ell,1} , \ldots, f_{\ell, r_{\ell} })$ is the subsystem of degree $\ell$ polynomials of $\mathbf{F}$ $(1 \leq \ell \leq d)$.
Let $F_{\ell, r}$ be the degree $\ell$ homogeneous portion of $f_{\ell, r}$.  We define $V_{\mathbf{F}_{\ell}}^*$ to be the set of points in $\mathbb{C}^n$ given by
$$
\text{rank } \left( \frac{ \partial F_{\ell,r} (\mathbf{x}) }{ \partial x_j }\right)_{ \substack{ 1 \leq r \leq r_{\ell}  \\ 1 \leq j \leq n } }  < r_{\ell},
$$
which is an affine variety over $\mathbb{C}$. Let us denote $\mathcal{B}_{\ell} (\mathbf{F}_{\ell})$ to be the codimension of $V_{\mathbf{F}_{\ell}}^*$.
In this paper we prove that provided the polynomials $\mathbf{f}$ satisfy suitable local conditions and $\mathcal{B}_{\ell}(\mathbf{F}_{\ell})$ is sufficiently large with respect to $d$ and $r_d,\ldots, r_1$ for each $1 \leq \ell \leq d$,
the system of equations $f_{\ell, r} (x_1, \ldots, x_n) = 0 \ (1 \leq \ell \leq d, 1 \leq r \leq r_{\ell})$
has a solution where each coordinate is prime. In fact we obtain the asymptotic formula for number of such solutions, counted with a logarithmic weight, under these hypotheses. We prove the statement via the Hardy-Littlewood circle method. This is a generalization of the work of B. Cook and \'{A}. Magyar \cite{CM}, where they
obtained the result when the polynomials of $\mathbf{f}$ all have degree $d$.
Hitherto, results of this type for systems of polynomial equations involving different degrees have been restricted to the diagonal case.

\section{Introduction}
Let $d \geq 1$, and let $\mathbf{f} = (\mathbf{f}_{d}, \ldots, \mathbf{f}_{1})$ be a system of polynomials in $\mathbb{Z}[x_1, \ldots, x_n]$, where
$\mathbf{f}_{\ell} = ( f_{\ell,1} , \ldots, f_{\ell, r_{\ell} })$ is the subsystem of degree $\ell$ polynomials of $\mathbf{f}$ $(1 \leq \ell \leq d)$.
We are interested in finding prime solutions, which are solutions with each coordinate a prime number,
to the equations
\begin{equation}
\label{main system}
f_{\ell,r}(x_1, \ldots, x_n) = 0  \ \ ( 1 \leq \ell \leq d, 1 \leq r \leq r_{\ell}).
\end{equation}
Let us denote $V_{\mathbf{f}, \mathbf{0}}(\mathbb{C})$ to be the affine variety in $\mathbb{C}^n$ defined by the equations (\ref{main system}).

Solving diophantine equations in primes is a fundamental problem in number theory. For example, the celebrated work of B. Green and
T. Tao \cite{GT1} on arithmetic progressions in primes can be phrased as the statement that given any $n \in \mathbb{N}$ the system of linear equations
$$
x_{i+2} - {x_{i+1}} = x_{i+1} - x_{i} \ (1 \leq i \leq n)
$$
has a prime solution $(x_1, \ldots, x_{n+2}) = (p_1, \ldots, p_{n+2})$ where $p_1< p_2 < \ldots < p_{n+2}$.
The modern results on the large scale distribution of prime solutions on
$V_{\mathbf{f}, \mathbf{0}}(\mathbb{C})$ when $\mathbf{f}$ consists only of linear polynomials,
for scenarios which do not reduce to a binary problem, is mostly summed up in the work of  B. Green and T. Tao \cite{GT}.
An example of a binary problem is bounding gaps between primes, an area which
J. Maynard \cite{M}, T. Tao (see  \cite[pp. 385]{M}), and Y. Zhang \cite{Z} made a significant progress in
by building on the work of D.A. Goldston, J.Pintz, and C.Y. Y{\i}ld{\i}r{\i}m \cite{GPY}.
In particular, it was shown in \cite{M} that one of the equations
$$
x_1 - x_2 = 2j \ (1 \leq j \leq 300)
$$
has infinitely many prime solutions.
Another binary problem of significance is the Goldbach's conjecture, which states that the equation
$$
x_1 + x_2 = N
$$
has a prime solution for every even integer $N$ greater than $2$.
It was proved by I. M. Vinogradov \cite{V} that the equation
\begin{equation}
\label{ternary goldbach}
x_1 + x_2 + x_3 = N
\end{equation}
has a prime solution for all sufficiently large odd $N \in \mathbb{N}$.
The Ternary Goldbach Problem, which is the assertion that the equation (\ref{ternary goldbach}) has a prime solution for all odd $N \in \mathbb{N}$ greater than or equal to 7, was solved by H.A. Helfgott in \cite{H1, H2}.

The examples given thus far had been for systems of linear equations.
The scenario for systems involving higher degree polynomials is also complex, and
has not been well-understood yet. Indeed, even the problem of whether a system of non-linear polynomial equations has a solution over $\mathbb{Q}$ is
`one of considerable complexity' \cite{BDLW}.

For solving non-linear equations in primes, there are results due to L. K. Hua \cite{H} for certain systems of homogeneous polynomials that are additive, for example on the system 
of the shape $x_1^j + \ldots + x_n^j = N_j \ (1 \leq j \leq d)$ where $N_j \in \mathbb{N}$. Hua also has results on the Waring-Goldbach problem, which is regarding prime solutions of the equation $x_1^d + \ldots + x_n^d = N$ where $N \in \mathbb{N}$. These results were established via the Hardy-Littlewood circle method. We refer the reader to \cite{KW} for a recent progress on the Waring-Goldbach problem due to A. V. Kumchev and T. D. Wooley. There is also \cite{C} by  S. Chow regarding prime solutions of certain diagonal equations by a transference principle approach.
For the case of regular indefinite integral quadratic forms, there is a result due to J. Liu \cite{L}.

The first result regarding prime solutions of general systems of non-linear polynomials is contained in the breakthrough of B. Cook and \'{A}. Magyar \cite{CM}, which we state in Theorem \ref{CM theorem}.
Before we can state their result we need to introduce some notations.
We also note that there is a discussion in \cite{CM} on this topic from the point of view of some recent results in sieve theory,
which the list includes \cite{BGS, DRS, LS}. We refer the reader to \cite{CM} for more details on this discussion.

Let $\ell > 1$. Let $\mathbf{G} = ( G_1, \ldots, G_{r'} )$ be a system of degree $\ell$ forms in $\mathbb{Q}[x_1, \ldots, x_n]$.
Here the term `forms' refers to homogeneous polynomials. We define the \emph{singular locus} $V_{\mathbf{G}}^*$ to be the set of points in $\mathbb{C}^n$ given by
\begin{equation}
\label{def sing loc}
\text{rank } \left( \frac{ \partial G_r(\mathbf{x}) }{ \partial x_j }\right)_{ \substack{ 1 \leq r \leq r'  \\ 1 \leq j \leq n } }  < r'.
\end{equation}
Observe that this defines an affine variety over $\mathbb{C}$.
We define the \textit{Birch rank}, $\mathcal{B}_{\ell} (\mathbf{G})$, to be the codimension of $V_{\mathbf{G}}^*$.
Given $\mathbf{g} = ( g_1, \ldots, g_{r'} )$, a system of degree $\ell$ polynomials in $\mathbb{Q}[x_1,\ldots, x_n]$,
where $G_{r}$ is the degree $\ell$ portion of $g_r \ (1 \leq r \leq r')$, we extend the notion of the Birch rank to systems of degree $\ell$ polynomials by
defining 
$$
\mathcal{B}_{\ell} (\mathbf{g}) :=  \mathcal{B}_{\ell} ( \mathbf{G}  ).
$$
When $\ell = 1$, following \cite{CM} we define $\mathcal{B}_{1} (\mathbf{g})$ to be the minimum number of non-zero coefficients
in a non-trivial linear combination
$$
\lambda_1 G_1 + \ldots + \lambda_{r'} G_{r'},
$$
where $\boldsymbol{\lambda} = (\lambda_1, \ldots, \lambda_{r'}) \in \mathbb{Q}^{r'} \backslash \{ \mathbf{0} \}.$
Clearly $\mathcal{B}_{1} (\mathbf{g}) > 0$ if and only if the linear forms $G_1, \ldots, G_{r'}$ are linearly independent over $\mathbb{Q}$.
For any $\ell \geq 1$, if $r'=0$ then we let $\mathcal{B}_{\ell} (\mathbf{g}) = + \infty$.

Let $\mathbf{F} = (\mathbf{F}_{d}, \ldots, \mathbf{F}_{1})$ be the system of homogeneous polynomials,
where for each $1 \leq \ell \leq d$, $\mathbf{F}_{\ell} = ( F_{\ell, 1}, \ldots, F_{\ell, r_{\ell}} )$ and
$F_{\ell, r}$ is the degree $\ell$ portion of $f_{\ell, r}$ in (\ref{main system}).
We let $V_{\mathbf{F}, \mathbf{0}}(\mathbb{R})$ be the set of points in  $\mathbb{R}^n$ satisfying
\begin{equation}
\label{main system homog}
F_{\ell,r}(x_1, \ldots, x_n) = 0  \ \ ( 1 \leq \ell \leq d, 1 \leq r \leq r_{\ell}).
\end{equation}

Let $\Lambda$ be the von Mangoldt function, where
$\Lambda(x)$ is $\log p$ if $x$ is a power of prime $p$, and $0$ otherwise.
Given $\mathbf{x} = (x_1, \ldots , x_n)$, we let
\begin{equation}
\label{Lambda}
\Lambda(\mathbf{x}) = \Lambda(x_1) \ldots \Lambda(x_n).
\end{equation}
We define the following quantity
$$
\mathcal{M}_{\mathbf{f}}(X) := \sum_{\mathbf{x} \in [0, X]^n} \Lambda(\mathbf{x})  \  \mathbf{1}_{V_{\mathbf{f}, \mathbf{0}}(\mathbb{C})} (\mathbf{x}),
$$
where $\mathbf{1}_{V_{\mathbf{f}, \mathbf{0}}(\mathbb{C})}$ is the characteristic function of the set ${V_{\mathbf{f}, \mathbf{0}}(\mathbb{C})}$.
Thus the quantity $\mathcal{M}_{\mathbf{f}}(X)$ is the number of solutions, counted with a logarithmic weight, of the equations (\ref{main system}) in $[0,X]^n$ whose coordinates are all prime powers.

We may now phrase the main result of B. Cook and \'{A}. Magyar in \cite{CM}, which is for the case when
the polynomials of $\mathbf{f}$ in (\ref{main system}) all have the same degree.
\begin{thm} \cite[Theorem 1]{CM}
\label{CM theorem}
Let $\mathbf{f} = \mathbf{f}_{d} = ( f_{d,1}, \ldots, f_{d, r_d} )$ be a system of degree $d$ polynomials in
$\mathbb{Z}[x_1,\ldots, x_n]$. If $ \mathcal{B}_d(\mathbf{f})$ is sufficiently large with respect to $d$ and $r_d$, then
there exist $\mathcal{C}(\mathbf{f})$, a constant which depends only on $\mathbf{f}$, and $c>0$ such that
$$
\mathcal{M}_{\mathbf{f} }(X) =  \mathcal{C}(\mathbf{f}) \  X^{n-d r_d} + O\left( \frac{X^{n-d r_d}}{(\log X)^c } \right).
$$
\end{thm}
In this paper, we generalize Theorem \ref{CM theorem} to handle systems of polynomials in which the degrees need not all be the same.
The following is the main theorem of this paper.
\begin{thm}
\label{main thm}
Let $\mathbf{f} = ( \mathbf{f}_{d}, \ldots,\mathbf{f}_{1} )$ be a system of polynomials in
$\mathbb{Z}[x_1,\ldots, x_n]$, where $\mathbf{f}_{\ell} = ( f_{\ell, 1}, \ldots, f_{\ell, r_{\ell}} )$ is the subsystem of  degree $\ell$ polynomials of $\mathbf{f}$ $(1 \leq \ell \leq d)$. For each $1 \leq \ell \leq d$,
suppose $\mathcal{B}_{\ell}(\mathbf{f}_{\ell})$ is sufficiently large with respect to $d$ and $r_d,\ldots, r_1$.
Then there exist $\mathcal{C}(\mathbf{f})$, a constant which depends only on $\mathbf{f}$, and
$c>0$ such that
$$
\mathcal{M}_{\mathbf{f} }(X) = \mathcal{C}(\mathbf{f}) \ X^{n- \sum_{\ell = 1}^d \ell r_{\ell}} + O\left( \frac{
X^{n- \sum_{\ell = 1}^d \ell r_{\ell}}
}{(\log X)^c } \right).
$$
\end{thm}
Of course if $r_{d-1} = \ldots = r_1 = 0$, then our Theorem \ref{main thm} recovers Theorem \ref{CM theorem}.
We also prove in Section \ref{section major arcs} that if the equations (\ref{main system}) has a non-singular solution in $\mathbb{Z}_p^{\times}$, the units of $p$-adic integers, for each prime $p$, and $V_{\mathbf{F}, \mathbf{0}}(\mathbb{R})$ has a non-singular real point in $(0,1)^n$, then
$$
\mathcal{C}(\mathbf{f}) > 0.
$$

We also present Theorem \ref{main thm 2} in Section \ref{sec concln}, where we obtain the asymptotic formula for the
number of prime solutions, counted with a logarithmic weight, instead of
solutions whose coordinates are all prime powers as in Theorem \ref{main thm}. Hitherto, the only examples in the literature of results of this type, for systems of polynomial equations involving different degrees, have been restricted to the diagonal case similar to the aforementioned result of L. K. Hua.

Theorems \ref{CM theorem} and \ref{main thm} are both obtained via the Hardy-Littlewood circle method.
In fact the method of our paper builds upon the techniques developed in \cite{CM}.
Circle method was pioneered by Hardy and Littlewood to give an asymptotic formula for the number of solutions to Waring's problem, and it has been quite effective at producing asymptotic formulas for the number of integer points of bounded height on varieties when the number of variables is sufficiently large.
The results of this type on the distribution of integer points on varieties
are provided by B. J. Birch \cite{B} and W. M. Schmidt \cite{S}.
In \cite{BHB}, T. D. Browning and D. R. Heath-Brown succeeded in generalizing the seminal work of Birch \cite{B},
and showed ``how forms of unequal degrees can be handled in an efficient manner, so as to give the results in the spirit of Birch for arbitrary systems.''
As stated in \cite{BHB}, ``Birch's original result needed the forms all to have the same degree, and there is a significant technical problem in extending the method to the general case.''
It is required in Theorem \ref{CM theorem} that the polynomials all have the same degree. 
As in the case for integer points, there are significant challenges to be overcome in generalizing the result on prime solutions of polynomial equations of equal degree to handle arbitrary systems.

The organization of the rest of the paper is as follows. In Section \ref{Sec reg lem}, we collect some definitions and results
related to the regularization process, which is an important part of the method in \cite{CM} and also of this paper.
In an application of the circle method, we consider the so-called \textit{major arcs} and \textit{minor arcs} (defined in Section \ref{sec initial set up}).
In Section \ref{sec decomp of forms}, we prove a result on decomposing a system of forms which becomes the starting point
in obtaining our minor arc estimates after the initial set-up prepared in Section \ref{sec initial set up}. We then obtain the desired minor arc estimates in Section \ref{section minor arc}. In Section \ref{sec tech 1}, we collect technical results that are necessary in obtaining our major arc estimates in Section \ref{section major arcs}. Finally, we state our conclusions and further remarks in Section \ref{sec concln}.
We also have Appendix \ref{appendix B}, where we provide proof for the results presented in Section \ref{sec tech 1}.
The work here is based on \cite{S}, and
we chose to present these technical details at the end for an easier read of the paper.


Throughout the paper we do not distinguish between the two terms `homogeneous polynomial' and `form', and
we will be using these terms interchangeably. By `rational form' we mean it is a form with coefficients in $\mathbb{Q}$.
We use $\ll$ and $\gg$ to denote Vinogradov's well-known notation.
We also use the notation $e(x)$ to denote $e^{2\pi i x}$. For $\mathbf{x} = (x_1, \ldots, x_n)$, the notation
$$
\sum_{\mathbf{x} \in [0,X]^n}
$$
means we are summing over all $\mathbf{x} \in \mathbb{Z}^n$ with $0 \leq x_i \leq X \ (1 \leq i \leq n)$.
For $q \in \mathbb{N}$, we use the numbers from $\{0, 1, \ldots, q-1 \}$ to represent the residue classes of
$\mathbb{Z}/q\mathbb{Z}$.
Finally, given $\mathbf{x} = (x_1, \ldots, x_n)$ we abuse notation slightly and let
$|\mathbf{x}| = n$ in Sections \ref{sec decomp of forms} and \ref{section minor arc},
whereas we let $|\mathbf{x}| = \max_{1 \leq i \leq n} |x_i|$ in Section \ref{sec tech 1} and onwards.
There should be no ambiguity since we are defining these notations for $|\mathbf{x}|$ again as they come up.

\textit{Acknowledgments.} The author would like to thank James Maynard and Kannan Soundararajan for their
helpful advices. The author would also like to thank the following people for helpful conversations and/or
encouragement while working on this paper: Matthew Beckett, Arunabha Biswas, Tim Browning, Francesco Cellarosi, Robert Krone, Jamie Mingo, Abdol-Reza Mansouri, M. Ram Murty, Mike Roth, Trevor Wooley, Stanley Yao Xiao, and  Serdar Y\"{u}ksel.

\section{Regularization lemmas}
\label{Sec reg lem}

In this section, we collect results from \cite{CM} and \cite{S} related to regular systems (see Definition \ref{def regular}) and the
regularization process. Given a system of rational forms $\textbf{F}$, via the regularization process we obtain
another system of forms which has at most the expected number of integer points, its number of forms is `small', and partitions the level sets
of $\textbf{F}$. This was an important component of the method in \cite{CM} used to split the
exponential sum in a controlled manner during the minor arc estimate. 

Let $\ell > 1$. Given a form $G \in \mathbb{Q}[x_1, \ldots, x_n]$ of degree $\ell$, we define the $h$-\textit{invariant},
also known as the \textit{rational Schmidt rank}, $h_{\ell}(G)$, to be the
least positive integer $h$ such that $G$ can be written identically as
\begin{equation}
\label{defn of hinv}
G = \widetilde{U}_1 \widetilde{V}_1 + \ldots + \widetilde{U}_h \widetilde{V}_h,
\end{equation}
where $\widetilde{U}_i$ and $\widetilde{V}_i$ are rational forms of positive degree $(1 \leq i \leq h)$.
Let $\mathbf{G} = ( G_1, \ldots, G_{r'} )$ be a system of degree $\ell$ forms in $\mathbb{Q}[x_1, \ldots, x_n]$. We generalize the
definition of the $h$-invariant, and
define the $h$-invariant of $\mathbf{G}$ to be
\begin{equation}
h_{\ell}(\mathbf{G}) = \min_{\boldsymbol{\mu} \in \mathbb{Q}^{r'} \backslash \{ \boldsymbol{0} \}}  h_{\ell}( \mu_{1} G_1 + \ldots + \mu_{r'} G_{r'} ).
\end{equation}
Let $\mathbf{g} = ( g_1, \ldots, g_{r'} )$ be a system of degree $\ell$ polynomials in $\mathbb{Q}[x_1,\ldots, x_n]$.
Let $G_{r}$ be the degree $\ell$ portion of $g_r \ (1 \leq r \leq r')$. We define
\begin{equation}
h_{\ell}(\mathbf{g}) := h_{\ell}(\mathbf{G}).
\end{equation}

The $h$-invariant satisfies the following property.
\begin{lem} \cite[Lemma 2.2]{XY}
\label{h ineq 1'} Let $\ell > 1$ and let $\mathbf{G} = (G_1, \ldots, G_{r'} )$ be a system of degree $\ell$ forms in $\mathbb{Q}[x_1, \ldots, x_n]$.
Suppose $h_{\ell}(\mathbf{G})>1$. Then for any  $1 \leq i \leq n$, we have
$$
h_{\ell}( \mathbf{G} ) -1   \leq h_{\ell}( \mathbf{G} |_{x_i=0} ) \leq h_{\ell}( \mathbf{G} ),
$$
where $\mathbf{G} |_{x_i=0} = ( G_1|_{x_i=0}, \ldots , G_{r'}|_{x_i = 0} )$.
\end{lem}

We have the following relation between the $h$-invariant and the Birch rank by combining \cite[Lemma 16.1, (10.3), (10.5), (17.1)]{S}.
\begin{lem}
Let $\ell > 1$ and let $\mathbf{G} = ( G_1, \ldots, G_{r'} )$ be a system of degree $\ell$ forms in $\mathbb{Q}[x_1,\ldots, x_n]$.
We have
\begin{equation}
\label{ineq of h and B}
h_{\ell}(\mathbf{G})\geq 2^{1-\ell} \ \mathcal{B}_{\ell}(\mathbf{G}).
\end{equation}
\end{lem}

\begin{defn}
\label{def regular}
Let $d>1$. Let $\mathbf{u} = (\mathbf{u}_{d}, \ldots, \mathbf{u}_{1})$ be a system of polynomials in $\mathbb{Q}[x_1, \ldots, x_n]$,
where $\mathbf{u}_{\ell} = ( u_{\ell, 1}, \ldots, u_{\ell, r_{\ell}} )$ is the subsystem of degree $\ell$ polynomials of $\mathbf{u}$ $(1 \leq \ell \leq d)$.
Let $D_{\mathbf{u}} = \sum_{\ell=1}^d \ell r_{\ell}$ and $R_{\mathbf{u} }  = \sum_{\ell=1}^d r_{\ell}$.
We denote $V_{\mathbf{u}, \mathbf{0}} (\mathbb{Z})$ to be the set of solutions in $\mathbb{Z}^n$ to the equations
\begin{equation}
\label{first set of eqns u}
u_{\ell, r} (\mathbf{x}) = 0 \ ( 1 \leq \ell \leq d, 1 \leq r \leq r_{\ell}).
\end{equation}
Let us denote the equations (\ref{first set of eqns u}) by
$\mathbf{u}(\mathbf{x}) = \mathbf{0}$.
We say the system $\mathbf{u}$ is \textit{regular} if
$$
| V_{\mathbf{u}, \mathbf{0}} (\mathbb{Z}) \cap [-X,X]^n | \ll X^{n-D_{\mathbf{u}}  }.
$$
\end{defn}
Similarly as above we also define $V_{\mathbf{u}, \mathbf{0}} (\mathbb{R})$ to be the set of solutions in $\mathbb{R}^n$
of the equations $\mathbf{u}(\mathbf{x}) = \mathbf{0}$.
For a system of polynomials $\mathbf{u}$ as given in Definition \ref{def regular},
we let $\mathbf{U} = (\mathbf{U}_{d}, \ldots, \mathbf{U}_{1})$ be the system of forms
such that for each $1 \leq \ell \leq d$, we have $\mathbf{U}_{\ell} = \{ U_{\ell, 1}, \ldots, U_{\ell, r_{\ell}} \}$
where $U_{\ell, r}$ is the degree $\ell$ portion of $u_{\ell, r}$ $(1 \leq r \leq r_{\ell})$.
The following theorem is one of the main results of \cite{S} due to Schmidt.

\begin{thm}\cite[Theorem II]{S}
\label{Schmidt main}
Let $d>1$. Let $\mathbf{u} = (\mathbf{u}_{d}, \ldots, \mathbf{u}_{2})$ be a system of rational polynomials
with notations as in Definition \ref{def regular},
and also let $\mathbf{U}_{\ell}$ be the system of degree $\ell$ portions of $\mathbf{u}_{\ell} \ (2 \leq \ell \leq d)$.
If we have
$$
h_{\ell}(  \mathbf{U}_{\ell} )  
\geq d \ 2^{4 \ell}  (\ell !)  r_{\ell} R_{\mathbf{u} }  \ \ (2 \leq \ell \leq d),
$$
then the system $\mathbf{u}$ is regular.
\end{thm}
Even though the statement of \cite[Theorem II]{S} is regarding systems of forms, the above Theorem \ref{Schmidt main}, which
is the inhomogeneous polynomials version, also holds by the explanation given in
\cite[Section 9]{S} and `Remark on inhomogeneous polynomials' in \cite[pp. 262]{S}.

Let us denote
\begin{equation}
\label{def rho}
\rho_{d,\ell}(t) =  d \ 2^{4 \ell}  (\ell !)  t^2  \ \ (2 \leq \ell \leq d).
\end{equation}
Then for each $2 \leq \ell \leq d$, $\rho_{d,\ell}(t)$ is an increasing function, and
$$
\rho_{d, \ell}(R_{\mathbf{u} }) \geq d \ 2^{4 \ell}  (\ell !)  r_{\ell} R_{ \mathbf{u} } .
$$

Note Theorem \ref{Schmidt main} is regarding systems of polynomials which do not contain any linear polynomials.
The following Corollary \ref{cor Schmidt} is for systems that contain linear forms as well.
We refer the reader to \cite[Corollary 3]{CM} or \cite[Corollary 3.3]{XY} for its proof.
\begin{cor}
\label{cor Schmidt}
Let $d>1$. Let $\mathbf{u} = (\mathbf{u}_{d}, \ldots, \mathbf{u}_{1})$ be a system of rational polynomials with notations as in Definition \ref{def regular},
and also let $\mathbf{U}_{\ell}$ be the system of degree $\ell$ portions of $\mathbf{u}_{\ell} \ (1 \leq \ell \leq d)$.
Suppose $\mathbf{u}_1$ only contains linear forms, in other words $\mathbf{u}_1 = \mathbf{U}_1$,
and that they are linearly independent over $\mathbb{Q}$.
For each $2 \leq \ell \leq d$, let $\rho_{d,\ell}(\cdot)$ be as in ~(\ref{def rho}).
If we have
$$
h(  \mathbf{U}_{\ell} ) \geq \rho_{d,\ell} (R_{ \mathbf{u} }  - r_1) + r_1 \ \ (2 \leq \ell \leq d),
$$
then the system $\mathbf{u}$ is regular.
\end{cor}


For $\mathbf{x} = (x_1, \ldots, x_n)$, by a partition of variables $\mathbf{x} = (\mathbf{y}, \mathbf{z})$ we mean that
the set of variables of $\mathbf{y}$ and $\mathbf{z}$ partition $x_1, \ldots, x_n$.
Let $\ell > 1$. Given $\mathbf{G} = ( G_1, \ldots, G_{r'} )$, a system of degree $\ell$ forms in $\mathbb{Q}[x_1, \ldots, x_n]$, and a partition of variables $\mathbf{x} = (\mathbf{y}, \mathbf{z})$, we denote $\overline{\mathbf{G}}$ to be the system obtained by removing from $\mathbf{G}$ all forms which depend only on the $\mathbf{z}$ variables. Clearly if we have the trivial partition $\mathbf{x} = (\mathbf{y}, \mathbf{z})$, where $\mathbf{z} = \emptyset$, then $\overline{\mathbf{G}} = \mathbf{G}$.
Given a degree $\ell$ form $G(\mathbf{x})$ in $\mathbb{Q}[x_1, \ldots, x_n]$, we define the \textit{$h$-invariant with respect to $\mathbf{z}$}, $h_{\ell}(G;\mathbf{z})$, to be the smallest number $h_0$ such that $G(\mathbf{x})$
can be expressed as
$$
G(\mathbf{x}) =  G(\mathbf{y}, \mathbf{z}) = \sum_{j=1}^{h_0} \widetilde{U}_j(\mathbf{y}, \mathbf{z}) \widetilde{V}_j(\mathbf{y}, \mathbf{z}) + W_0(\mathbf{z}),
$$
where $\widetilde{U}_j$ and  $\widetilde{V}_j$ are rational forms of positive degree $(1 \leq j \leq h_0)$, and
$W_0(\mathbf{z})$ is a rational form only in the $\mathbf{z}$ variables.
We also define $h_{\ell}(\mathbf{G}; \mathbf{z})$ to be
$$
h_{\ell}(\mathbf{G}; \mathbf{z}) = \min_{\boldsymbol{\lambda} \in \mathbb{Q}^{r'} \backslash \{ \boldsymbol{0} \}}  h_{\ell}( \lambda_{1} G_1 + \ldots + \lambda_{r'} G_{r'}; \mathbf{z} ).
$$
If we have the trivial partition, then clearly we have $h_{\ell}(\mathbf{G}; \emptyset) = h_{\ell}(\mathbf{G}).$
From this definition the following lemma holds.
\begin{lem} \cite[Lemma 2]{CM}
\label{Lemma 2 in CM}
Let $\ell > 1$. Let $\mathbf{G} = (G_1, \ldots, G_{r'} )$ be a system of degree $\ell$ forms in $\mathbb{Q}[x_1, \ldots, x_n]$, and suppose we have a partition of variables $\mathbf{x} = (\mathbf{y}, \mathbf{z})$. Let $\mathbf{y}'$ be a set of variables with the same number of variables as $\mathbf{y}$.
Then we have
$$
h_{\ell}(\mathbf{G}(\mathbf{y}, \mathbf{z}), \mathbf{G}(\mathbf{y}', \mathbf{z}) ; \mathbf{z} ) = h_{\ell}( \mathbf{G} ; \mathbf{z} ),
$$
where the left hand side denotes the $h$-invariant with respect to $\mathbf{z}$ of the system
$$
( G_1(\mathbf{y}, \mathbf{z}), \ldots, G_{r'}( \mathbf{y}, \mathbf{z}), G_1(\mathbf{y}', \mathbf{z}), \ldots, G_{r'}(\mathbf{y}', \mathbf{z})).
$$
\end{lem}

In \cite{CM}, the process in the following proposition is referred to as the regularization of systems.
We will be utilizing this proposition in Section \ref{section minor arc} to obtain the minor arc estimate.

\begin{prop}\cite[Propositions 1 and 1']{CM}
\label{prop reg par}
Let $d>1$, and let $\boldsymbol{\mathcal{F}}$ be any collection of non-decreasing functions $\mathcal{F}_i: \mathbb{Z}_{\geq 0} \rightarrow \mathbb{Z}_{\geq 0} \ (2 \leq i \leq d)$. For a collection of non-negative integers $r_1, \ldots, r_d$, there exist constants
$$
C_1(r_1, \ldots, r_d, \boldsymbol{\mathcal{F}} ), \ldots , C_d(r_1, \ldots, r_d,  \boldsymbol{\mathcal{F}} )
$$
such that the following holds.

Given a system of forms $\mathbf{U} = ({\mathbf{U}}_{d}, \ldots, {\mathbf{U}}_{1})$ in $\mathbb{Z}[x_1, \ldots, x_n]$, where
$\mathbf{U}_{\ell} = (U_{\ell, 1}, \ldots, U_{\ell, r_{\ell}})$ is the subsystem of degree $\ell$ forms of $\mathbf{U}$ $(1 \leq \ell \leq d)$, and a partition of variables $\mathbf{x} = (\mathbf{y}, \mathbf{z})$, there exists a system of forms
$\mathcal{R}( \mathbf{U}) = (\mathcal{R}^{(d)}(\mathbf{U}), \ldots, \mathcal{R}^{(1)}( \mathbf{U}) )$ in $\mathbb{Q}[x_1, \ldots, x_n]$,
where $\mathcal{R}^{(\ell)}( \mathbf{U})$ is the subsystem of degree $\ell$ forms of $\mathcal{R}( \mathbf{U})$, satisfying the following.
For each $1 \leq \ell \leq d$, let $r'_{\ell}$ be the number of forms in  $\mathcal{R}^{(\ell)}( \mathbf{U})$, and let $R' = r'_1 + \ldots + r'_d$.
\newline

$(1)$ Each form of the system $\mathbf{U}$ can be written as a rational polynomial expression in the forms of the system $\mathcal{R}( \mathbf{U})$.
In particular, the level sets of $\mathcal{R}( \mathbf{U})$ partition those of $\mathbf{U}$.

$(2)$ For each $1 \leq \ell \leq d$, $r'_{\ell}$ is at most $C_{\ell}(r_1, \ldots, r_d, \boldsymbol{\mathcal{F}})$.

$(3)$ For each $2 \leq \ell \leq d$, we have $h_{\ell}( \mathcal{R}^{(\ell)}( \mathbf{U}) ) \geq \mathcal{F}_{\ell}(R')$. Moreover, the linear forms of
 $\mathcal{R}^{(1)}( \mathbf{U})$ are linearly independent over $\mathbb{Q}$.

$(4)$ Let $\overline{ \mathcal{R} }^{(\ell)} ( \mathbf{U})$ be the system obtained by removing from $\mathcal{R}^{(\ell)}( \mathbf{U})$ all forms which depend only
on the $\mathbf{z}$ variables $(1 \leq \ell \leq d)$. Then for each $2 \leq \ell \leq d$, we have $h_{\ell}(\overline{ \mathcal{R} }^{(\ell)} ( \mathbf{U}) ; \mathbf{z}) \geq \mathcal{F}_{\ell}(R')$. Furthermore, we may assume that the linear forms of $\overline{ \mathcal{R} }^{(1)} ( \mathbf{U})$ depend only on the $\mathbf{y}$ variables, and that they are linearly independent over $\mathbb{Q}$.
\end{prop}

We note that the last assertion in $(4)$ regarding the linear forms of $\overline{ \mathcal{R} }^{(1)} ( \mathbf{U})$
is not stated in \cite[Proposition 1']{CM}. However, it is easy to deduce that this is indeed the case from \cite[Proposition 1']{CM}
at the expense of possibly slightly larger constants $C_i(r_1, \ldots, r_d,  \boldsymbol{\mathcal{F}} )$ $(1 \leq i \leq d)$
than in \cite[Proposition 1']{CM}.
We also note that with this assertion, it follows that every linear form of $\mathcal{R}^{(1)} ( \mathbf{U})$ is either only in the $\mathbf{y}$ variables, or only in the $\mathbf{z}$ variables.


\section{Decomposition of forms}
\label{sec decomp of forms}
In this section, we decompose a system of forms into two parts in a way that both parts have large Birch rank.
Let $d, n  \in \mathbb{N}$, and let $\textbf{F}$ be a system of forms in $\mathbb{Q}[x_1, \ldots, x_n]$ of degrees less than or equal to $d$.
We use a slightly different notation in this section compared to the previous sections in order to make the arguments
as clear as possible. We denote $\textbf{F} = ({\textbf{F}}^{(d)}, \ldots, {\textbf{F}}^{(1)})$, where ${\textbf{F}}^{(\ell)}$
is the subsystem of degree $\ell$ forms of $\textbf{F} \  (1 \leq \ell \leq d)$. For each $1 \leq \ell \leq d$, we denote the elements of ${\textbf{F}}^{(\ell)}$ by
$$
{\mathbf{F}}^{(\ell)} = ( F^{(\ell)}_1, \ldots, F^{(\ell)}_{r_{\ell}} ),
$$
where $r_{\ell}$ is the number of forms in ${\mathbf{F}}^{(\ell)}$.
Suppose we have a partition of variables $\mathbf{x} = (\mathbf{y}, \mathbf{z})$.
We define $\mathbf{F}^{(\ell)}_{\mathbf{y}, \mathbf{z} } (\mathbf{y}, \mathbf{z} )$
to be the following system of forms
\begin{eqnarray}
\label{F_y,z}
\mathbf{F}^{(\ell)}_{\mathbf{y}, \mathbf{z} } (\mathbf{y}, \mathbf{z} )
= ( F^{(\ell)}_1(\mathbf{y}, \mathbf{z}) - F^{(\ell)}_1(\mathbf{0}, \mathbf{z}), \ldots, F^{(\ell)}_{r_{\ell}}(\mathbf{y}, \mathbf{z}) - F^{(\ell)}_{r_{\ell}}(\mathbf{0}, \mathbf{z}) ).
\end{eqnarray}
Note for each $1 \leq r \leq r_{\ell}$, we have
$$
F^{(\ell)}_r(\mathbf{y}, \mathbf{z}) - F^{(\ell)}_r(\mathbf{0}, \mathbf{z}) = F^{(\ell)}_r(\mathbf{y}, \mathbf{0})
+ (F^{(\ell)}_r(\mathbf{y}, \mathbf{z}) - F^{(\ell)}_r(\mathbf{y}, \mathbf{0}) - F^{(\ell)}_r(\mathbf{0}, \mathbf{z})),
$$
and every monomial with non-zero coefficient in $(F^{(\ell)}_r(\mathbf{y}, \mathbf{z}) - F^{(\ell)}_r(\mathbf{y}, \mathbf{0}) - F^{(\ell)}_r(\mathbf{0}, \mathbf{z}))$ involves both the $\mathbf{y}$ variables and the $\mathbf{z}$ variables, in other words it can not be in terms of only the $\mathbf{y}$ variables or only the $\mathbf{z}$ variables.
For each $1 \leq \ell \leq d,$ we also define
\begin{eqnarray}
\label{F_z}
\mathbf{F}^{(\ell)}_{\mathbf{z} } (\mathbf{z})
= ( F^{(\ell)}_1(\mathbf{0}, \mathbf{z}), \ldots, F^{(\ell)}_{r_{\ell}}(\mathbf{0}, \mathbf{z}) ).
\end{eqnarray}
It should be clear from the context which partition of variables is being used when the notations (\ref{F_y,z}) and (\ref{F_z})
come up in this section. We now give an example of how these notations may be used. Let us consider $\mathbf{F}^{(\ell)}$ with $\ell > 1$.
Suppose we have partitions of variables $\mathbf{x} = (\mathbf{v}, \mathbf{z})$ and $\mathbf{z} = (\mathbf{y}, \mathbf{z}')$, and let us denote
$\mathbf{x} = (\mathbf{v}, (\mathbf{y}, \mathbf{z}') )$.
From the first partition of variables, we have $\mathbf{F}^{(\ell)}_{\mathbf{v}, \mathbf{z} } (\mathbf{v}, \mathbf{z} )$
and $\mathbf{F}^{(\ell)}_{\mathbf{z} } (\mathbf{z})$ as above. Since $\mathbf{F}^{(\ell)}_{\mathbf{z} } (\mathbf{z})$ is in terms of the $\mathbf{z}$
variables, we can consider (\ref{F_y,z}) and (\ref{F_z}) of this system with resect to the partition  $\mathbf{z} = (\mathbf{y}, \mathbf{z}')$.
We then have
$$
{(\mathbf{F}^{(\ell)}_{\mathbf{z} })}_{\mathbf{y}, \mathbf{z}' }(\mathbf{y}, \mathbf{z}')
=
\Big{(}
F^{(\ell)}_1(\mathbf{0}, (\mathbf{y}, \mathbf{z}')) - F^{(\ell)}_1(\mathbf{0}, (\mathbf{0}, \mathbf{z}')), \ldots, F^{(\ell)}_{r_{\ell}}(\mathbf{0},(\mathbf{y}, \mathbf{z}')) - F^{(\ell)}_{r_{\ell}}(\mathbf{0}, (\mathbf{0}, \mathbf{z}'))
\Big{)}
$$
and
$$
{(\mathbf{F}^{(\ell)}_{\mathbf{z} })}_{ \mathbf{z}' } (\mathbf{z}') = ( F^{(\ell)}_1(\mathbf{0}, (\mathbf{0}, \mathbf{z}') ), \ldots, F^{(\ell)}_{r_{\ell}}(\mathbf{0},  (\mathbf{0}, \mathbf{z}')  ).
$$

Given a set of variables $\mathbf{y}$, we denote $|\mathbf{y}|$ to be the number of variables of $\mathbf{y}$.
The following proposition for a system of forms of equal degree is proved in \cite{CM}. We will be using this proposition
as a base case in induction to generalize the result in Proposition \ref{prop decomp}.
\begin{prop} \cite[Proposition 2]{CM}
\label{Prop 2 in CM}
Let $C_1$ and $C_2$ be some positive integers. Let $d \geq 1$ and $\mathbf{F} = \mathbf{F}^{(d)} =(F^{(d)}_{1}, \ldots, F^{(d)}_{r_d})$ be a system of degree $d$ forms in $\mathbb{Q}[x_1, \ldots, x_n]$, where $\mathcal{B}_d(\mathbf{F})$ is sufficiently large with respect to $C_1$, $C_2$, $r_d$ and $d$. Then there exists a
partition of variables $\mathbf{x} = (\mathbf{y}, \mathbf{z})$ such that
$$
|\mathbf{y}| \leq C_{1} r_d,
$$

$$
\mathcal{B}_d (\mathbf{F}_{\mathbf{y}, \mathbf{z} } (\mathbf{y}, \mathbf{z} )  ) \geq C_{1},
\ \ \
\text{ and }
\ \ \
\mathcal{B}_d (    \mathbf{F}_{ \mathbf{z} } (\mathbf{z})   ) \geq C_{2}.
$$
\end{prop}

The following lemma and its corollary are also proved in \cite{CM}.
\begin{lem} \cite[Lemma 3]{CM}
\label{Lemma 3 in CM}
Let $\ell \geq 1$ and let $\mathbf{G} = (G_1, \ldots, G_{r'})$ be a system of degree $\ell$ forms in $\mathbb{Q}[x_1, \ldots, x_n]$.
Given any $1 \leq j \leq n$, we have
\begin{equation}
\label{eqn 1 in lemma 3 CM}
\mathcal{B}_{\ell}(\mathbf{G}) \geq \mathcal{B}_{\ell}(\mathbf{G}|_{x_j = 0}) \geq \mathcal{B}_{\ell}(\mathbf{G}) - r' - 1,
\end{equation}
where
$\mathbf{G}|_{x_j = 0} =  (G_1 |_{x_j = 0}, \ldots, G_{r'}|_{x_j = 0}).$
When $\ell=1$, we in fact have
$$
\mathcal{B}_1(\mathbf{G}) \geq \mathcal{B}_1(\mathbf{G}|_{x_j = 0}) \geq \mathcal{B}_1(\mathbf{G}) - 1.
$$
\end{lem}
\begin{proof}
The lower bounds are provided in \cite[Lemma 3]{CM}, so we only provide the arguments for the upper bounds here.
We remark that due to a minor oversight the lower bound is stated to be $\mathcal{B}_{\ell}(\mathbf{G}) - r'$
in \cite[Lemma 3]{CM} instead of the lower bound given in (\ref{eqn 1 in lemma 3 CM}). However, by following through their argument it can be seen that in fact (\ref{eqn 1 in lemma 3 CM}) is the correct lower bound.
We begin by considering the case $\ell > 1$. Let us denote
$$
\mathbf{G}' = (G'_1, \ldots, G'_{r'}),
$$
where $G'_r = G_r |_{x_j = 0}$ for each $1 \leq r \leq r'$. It follows from the definition of the singular locus given in (\ref{def sing loc})
that
$$
V^*_{\mathbf{G} } \cap \{ \mathbf{x} \in \mathbb{C}^n : x_j = 0  \} \subseteq V^*_{\mathbf{G}' }.
$$
Since the dimension of $V^*_{\mathbf{G} } \cap \{ \mathbf{x} \in \mathbb{C}^n : x_j = 0  \}$
is either $\dim (V^*_{\mathbf{G} } ) - 1$ or $\dim (V^*_{\mathbf{G} } )$, we have
$$
\dim (V^*_{\mathbf{G} } ) - 1 \leq \dim (V^*_{\mathbf{G}' }),
$$
and equivalently,
$$
\mathcal{B}_{\ell}(\mathbf{G}) = n - \dim (V^*_{\mathbf{G} } ) \geq n - 1 -  \dim (V^*_{\mathbf{G}' }) = \mathcal{B}_{\ell}(\mathbf{G}|_{x_j = 0}).
$$
For the case $\ell = 1$, it follows immediately from the definition.
\end{proof}

\begin{cor} \cite[Corollary 4]{CM}
\label{Cor 4 in CM}
Let $\ell \geq 1$ and let $\mathbf{G} = (G_1, \ldots, G_{r'})$ be a system of degree $\ell$ forms in $\mathbb{Q}[x_1, \ldots, x_n]$.
If $\mathcal{H}$ is an affine linear space of co-dimension $m$, then the restriction of $\mathbf{G}$ to $\mathcal{H}$
has Birch rank at least $(\mathcal{B}_{\ell}(\mathbf{G}) - m (r' + 1))$. When $\ell = 1$, we in fact have that it is at least $(\mathcal{B}_{1}(\mathbf{G}) - m)$.
\end{cor}

We obtain the following technical result for a system of forms that is  more general than in Proposition \ref{Prop 2 in CM}.
\begin{prop}
\label{prop decomp}
Let $d, n \in \mathbb{N}$. Let $\textbf{F} = ({\textbf{F}}^{(d)}, \ldots, {\textbf{F}}^{(1)})$ be a system of forms in $\mathbb{Q}[x_1, \ldots, x_n]$,
where ${\textbf{F}}^{(i)} = ( F^{(i)}_1, \ldots, F^{(i)}_{r_{i}} )$ is the subsystem of degree $i$ forms of $\textbf{F}$ $(1 \leq i \leq d)$.
Let $C_{i,1}, C_{i,2} \ (1 \leq i \leq d)$ be positive integers.
For each $1 \leq i \leq d$, suppose $\mathcal{B}_i(\mathbf{F}^{(i)})$ is sufficiently large with respect to $C_{1,1}, \ldots, C_{d,1}$,  $C_{1,2}, \ldots, C_{d,2}$, $r_d, \ldots, r_1$, and $d$. Then there exists a partition of variables $\mathbf{x} = (\mathbf{y}, \mathbf{z})$ such that
$$
|\mathbf{y}| \leq \sum_{i=1}^d C_{i,1} r_{i},
$$
and for each $1 \leq i \leq d$, we have
$$
\mathcal{B}_i \left( {\mathbf{F}}^{(i)}_{\mathbf{y}, \mathbf{z} } (\mathbf{y}, \mathbf{z} )  \right) \geq
C_{i,1} -  (r_i + 1) \sum_{\ell = 1}^{i-1} C_{\ell,1} r_{\ell},
$$
and
$$
\mathcal{B}_i \left(    {\mathbf{F}}^{(i)}_{ \mathbf{z} } (\mathbf{z})   \right) \geq C_{i,2} -  (r_i + 1) \sum_{\ell = 1}^{i-1} C_{\ell,1} r_{\ell}.
$$
\end{prop}

\begin{proof}
We prove by induction the following statement: Given $2 \leq j \leq d$,
there exists a partition of variables $\mathbf{x} = (\mathbf{v}_{j}, \mathbf{z}_{j} )$, where
$\mathbf{v}_{j} = (\mathbf{y}_d, \ldots , \mathbf{y}_{j})$,
such that for each $j \leq i \leq d$ we have
$$
| \mathbf{y}_{i} | \leq C_{i,1} r_{i},
$$
$$
\mathcal{B}_i \left(  {\mathbf{F}}^{(i)}_{\mathbf{v}_{j}, \mathbf{z}_{j} } (\mathbf{v}_{j}, \mathbf{z}_{j} ) \right) \geq C_{i,1} -  (r_i + 1)  \sum_{\ell = j}^{i-1} C_{\ell,1} r_{\ell},
$$
and
$$
\mathcal{B}_i \left(  {\mathbf{F}}^{(i)}_{\mathbf{z}_{j}}(\mathbf{z}_{j})  \right) \geq C_{i,2} -  (r_i + 1) \sum_{\ell = j}^{i-1} C_{\ell,1} r_{\ell}.
$$

We begin with the base case $j=d.$ We know from Proposition \ref{Prop 2 in CM} that there exists a partition of variables
$\mathbf{x} = ( \mathbf{y}_d, \mathbf{z}_d )$ such that
$$
|\mathbf{y}_d| \leq C_{d,1} r_d,
$$

$$
\mathcal{B}_d \left( {\mathbf{F}}^{(d)}_{\mathbf{y}_d, \mathbf{z}_d } (\mathbf{y}_d, \mathbf{z}_d )  \right) \geq C_{d,1},
\ \ \
\text{ and }
\ \ \
\mathcal{B}_d  \left(    {\mathbf{F}}^{(d)}_{ \mathbf{z}_d } (\mathbf{z}_d)   \right) \geq C_{d,2}.
$$
This concludes our base case.

Suppose the statement holds for $j+1$, in other words there exists a partition of variables $\mathbf{x} = (\mathbf{v}_{j+1}, \mathbf{z}_{j+1} )$, where
$\mathbf{v}_{j+1} = (\mathbf{y}_d, \ldots , \mathbf{y}_{j+1})$,
such that for each $j+1 \leq i \leq d$ we have
$$
| \mathbf{y}_{i} | \leq C_{i,1} r_{i},
$$
\begin{equation}
\label{eqn d6}
\mathcal{B}_i \left(  {\mathbf{F}}^{(i)}_{\mathbf{v}_{j+1}, \mathbf{z}_{j+1} } (\mathbf{v}_{j+1}, \mathbf{z}_{j+1} ) \right) \geq C_{i,1} -  (r_i + 1) \sum_{\ell = j+1}^{i-1} C_{\ell,1} r_{\ell},
\end{equation}
and
\begin{equation}
\label{eqn d7}
\mathcal{B}_i \left(  {\mathbf{F}}^{(i)}_{\mathbf{z}_{j+1}}(\mathbf{z}_{j+1})  \right) \geq C_{i,2} -  (r_i + 1) \sum_{\ell = j+1}^{i-1} C_{\ell,1} r_{\ell}.
\end{equation}

First we observe that by Lemma \ref{Lemma 3 in CM} the following holds
\begin{eqnarray}
\notag
\mathcal{B}_j \left( \mathbf{F}^{(j)}_{\mathbf{z}_{j+1}} (\mathbf{z}_{j+1}) \right) = \mathcal{B}_j \left( \mathbf{F}^{(j)}(\mathbf{x}) \Big{|}_{\mathbf{v_{j+1}} = \mathbf{0}}  \right)
&\geq& \mathcal{B}_j \left( \mathbf{F}^{(j)}(\mathbf{x}) \right) - (r_j + 1)|\mathbf{v}_{j+1}|
\\
\notag
&\geq& \mathcal{B}_j \left( \mathbf{F}^{(j)}(\mathbf{x}) \right) - (r_j + 1) \sum_{i=j+1}^{d} C_{i,1} r_{i}.
\end{eqnarray}
Since $\mathcal{B}_j \left( \mathbf{F}^{(j)}(\mathbf{x}) \right)$ is sufficiently large with respect to $d$, $r_j, \ldots , r_d$, $C_{j,1}, \ldots , C_{d,1}$,
and $C_{j,2}$, we obtain from Proposition \ref{Prop 2 in CM} a partition of variables $\mathbf{z}_{j+1} = (\mathbf{y}_j, \mathbf{z}_j)$ such that
\begin{equation}
\label{eqn d8}
|\mathbf{y}_j| \leq C_{j,1} r_j,
\end{equation}
\begin{equation}
\label{eqn d4}
\mathcal{B}_{j} \left( {( {\mathbf{F}}^{(j)}_{ \mathbf{z}_{j+1} } )}_{\mathbf{y}_{j}, \mathbf{z}_{j} } (\mathbf{y}_{j}, \mathbf{z}_{j} )   \right) \geq C_{j,1},
\end{equation}
and
\begin{equation}
\label{eqn d1}
\mathcal{B}_{j} \left(    {( {\mathbf{F}}^{(j)}_{ \mathbf{z}_{j+1} } )}_{ \mathbf{z}_{j} } (\mathbf{z}_{j})   \right) \geq C_{j,2}.
\end{equation}
We denote $\mathbf{v}_j = ( \mathbf{v}_{j+1}, \mathbf{y}_j ) =  ( \mathbf{y}_{d}, \ldots , \mathbf{y}_j )$, and
consider the partition of variables $\mathbf{x} = (\mathbf{v}_{j}, \mathbf{z}_{j} )$.

Since $\mathbf{v}_{j+1} \subseteq \mathbf{v}_{j}$, we have
\begin{eqnarray}
{{F}_r}^{(j)} (\mathbf{x}) |_{ \mathbf{v}_{j} = \mathbf{0} }
&=&
\left( {{F}_r}^{(j)} (\mathbf{x}) |_{ \mathbf{v}_{j+1} = \mathbf{0} } \right) \Big{|}_{ \mathbf{v}_{j} = \mathbf{0} } \ \ \  (1 \leq r \leq r_j)
\notag
\end{eqnarray}
and consequently,
\begin{equation}
\label{Fj eqn 1}
{\mathbf{F}}^{(j)}_{ \mathbf{z}_{j} } (\mathbf{z}_{j}) = {( {\mathbf{F}}^{(j)}_{ \mathbf{z}_{j+1}} )}_{ \mathbf{z}_{j} } (\mathbf{z}_{j}).
\end{equation}
Therefore, we obtain by (\ref{eqn d1}) that
\begin{equation}
\label{eqn d10}
\mathcal{B}_{j} \left(  {\mathbf{F}}^{(j)}_{ \mathbf{z}_{j} } (\mathbf{z}_{j}) \right) \geq C_{j,2}.
\end{equation}

We have the following two decompositions for $\mathbf{F}^{(j)}(\mathbf{x})$,
\begin{eqnarray}
\notag
&&{\mathbf{F}}^{(j)}_{\mathbf{v}_{j}, \mathbf{z}_j } (\mathbf{v}_{j}, \mathbf{z}_{j} )  +     {\mathbf{F}}^{(j)}_{ \mathbf{z}_{j} } (\mathbf{z}_{j})
\\
\notag
&=&
\mathbf{F}^{(j)}_{ \mathbf{v}_{j+1}, \mathbf{z}_{j+1} } ( \mathbf{v}_{j+1}, \mathbf{z}_{j+1}) +
{( {\mathbf{F}}^{(j)}_{ \mathbf{z}_{j+1} } )}_{\mathbf{y}_{j}, \mathbf{z}_{j} } (\mathbf{y}_{j}, \mathbf{z}_{j} ) + {( {\mathbf{F}}^{(j)}_{ \mathbf{z}_{j+1}} )}_{ \mathbf{z}_{j} } (\mathbf{z}_{j}),
\end{eqnarray}
where the first decomposition is via the partition $\mathbf{x} = (\mathbf{v}_j, \mathbf{z}_j )$, and the second via the partitions
$\mathbf{x} = (\mathbf{v}_{j+1}, \mathbf{z}_{j+1} )$ and $\mathbf{z}_{j+1}  = (\mathbf{y}_{j}, \mathbf{z}_{j} )$. It follows from (\ref{Fj eqn 1})  that
\begin{eqnarray}
&& {\mathbf{F}}^{(j)}_{\mathbf{v}_{j}, \mathbf{z}_j } (\mathbf{v}_{j}, \mathbf{z}_{j} )
\label{eqn d2} =\mathbf{F}^{(j)}_{ \mathbf{v}_{j+1}, \mathbf{z}_{j+1} } ( \mathbf{v}_{j+1}, \mathbf{z}_{j+1}) +
{( {\mathbf{F}}^{(j)}_{ \mathbf{z}_{j+1} } )}_{\mathbf{y}_{j}, \mathbf{z}_{j} } (\mathbf{y}_{j}, \mathbf{z}_{j} ).
\end{eqnarray}

Since $\mathbf{v}_{j+1} \cap (\mathbf{y}_j, \mathbf{z}_j) = \emptyset$, we obtain from (\ref{eqn d2})
\begin{eqnarray}
\label{eqn d3}
\left( {\mathbf{F}}^{(j)}_{\mathbf{v}_j, \mathbf{z}_{j}} (\mathbf{v}_j, \mathbf{z}_{j})
\right)
\Big{|}_{\mathbf{v}_{j+1} = \mathbf{0}}
&=& \left(
 {( {\mathbf{F}}^{(j)}_{ \mathbf{z}_{j+1} } )}_{\mathbf{y}_{j}, \mathbf{z}_{j} } (\mathbf{y}_{j}, \mathbf{z}_{j} )
\right) \Big{|}_{\mathbf{v}_{j+1} = \mathbf{0}}
\\
\notag
&=&
{( {\mathbf{F}}^{(j)}_{ \mathbf{z}_{j+1} } )}_{\mathbf{y}_{j}, \mathbf{z}_{j} } (\mathbf{y}_{j}, \mathbf{z}_{j} ).
\end{eqnarray}
Also see the sentence after (\ref{F_y,z}) for the explanation of
$\mathbf{F}^{(j)}_{ \mathbf{v}_{j+1}, \mathbf{z}_{j+1} } ( \mathbf{v}_{j+1}, \mathbf{z}_{j+1}) \Big{|}_{\mathbf{v}_{j+1} = \mathbf{0}} = \mathbf{0}.$
Consequently, we obtain from \ref{Lemma 3 in CM}, (\ref{eqn d3}), and (\ref{eqn d4}),
\begin{eqnarray}
\label{eqn d11}
\mathcal{B}_{j} \left(  {\mathbf{F}}^{(j)}_{\mathbf{v}_j, \mathbf{z}_{j}} (\mathbf{v}_j, \mathbf{z}_{j}) \right)
&\geq&\mathcal{B}_{j} \left(
 {\mathbf{F}}^{(j)}_{\mathbf{v}_j, \mathbf{z}_{j}} (\mathbf{v}_j, \mathbf{z}_{j})  \Big{|}_{\mathbf{v}_{j+1} = \mathbf{0}} \right)
\\
\notag
&=& \mathcal{B}_{j} \left( {( {\mathbf{F}}^{(j)}_{ \mathbf{z}_{j+1} } )}_{\mathbf{y}_{j}, \mathbf{z}_{j} } (\mathbf{y}_{j}, \mathbf{z}_{j} )  \right)
\\
\notag
&\geq& C_{j,1}.
\end{eqnarray}

Let $j+1 \leq i \leq d$. 
Recall we have partitions $\mathbf{x} = (\mathbf{v}_{j+1}, \mathbf{z}_{j+1})$, $\mathbf{x} = (\mathbf{v}_{j}, \mathbf{z}_{j})$, and $\mathbf{z}_{j+1} = (\mathbf{y}_{j}, \mathbf{z}_{j})$.
Since 
$\mathbf{v}_j = (\mathbf{v}_{j+1}, \mathbf{y}_j)$,
it follows that
$$
\left( {{F}}^{(i)}_r ( \mathbf{x} ) |_{\mathbf{v}_{j+1} = \mathbf{0}} \right) \Big{|}_{\mathbf{y}_{j} = \mathbf{0}}
=
{{F}}^{(i)}_r ( \mathbf{x} ) |_{\mathbf{v}_{j} = \mathbf{0} } \ \ (1 \leq r \leq r_{i})
$$
and consequently,
\begin{equation}
\label{eqn d5}
({\mathbf{F}}^{(i)}_{ \mathbf{z}_{j+1} })_{\mathbf{z}_{j}} (\mathbf{z}_{j})
= {\mathbf{F}}^{(i)}_{ \mathbf{z}_{j+1} } (\mathbf{z}_{j+1}) \Big{|}_{\mathbf{y}_{j} = \mathbf{0}}
= {\mathbf{F}}^{(i)}_{ \mathbf{z}_{j} } (\mathbf{z}_{j}).
\end{equation}
Therefore, we obtain from (\ref{eqn d5}), Lemma \ref{Lemma 3 in CM}, (\ref{eqn d7}), and (\ref{eqn d8}),
\begin{eqnarray}
\mathcal{B}_i \left(  {\mathbf{F}}^{(i)}_{ \mathbf{z}_{j} } (\mathbf{z}_{j}) \right)
\label{eqn d12}
&\geq&
\mathcal{B}_i \left(   {\mathbf{F}}^{(i)}_{ \mathbf{z}_{j+1} } (\mathbf{z}_{j+1})  \right) - (r_i + 1) |\mathbf{y}_{j}|
\notag
\\
&\geq&
\left( C_{i,2} - (r_i + 1)\sum_{\ell = j+1}^{i-1} C_{\ell, 1} r_{\ell} \right) - (r_i + 1) |\mathbf{y}_{j}|
\notag
\\
&\geq&
 C_{i,2} - (r_i + 1) \sum_{\ell = j}^{i-1} C_{\ell, 1} r_{\ell}.
\notag
\end{eqnarray}

Also we have the following two decomposition for $\mathbf{F}^{(i)}(\mathbf{x})$,
\begin{eqnarray}
\label{eqn d9}
&&
{\mathbf{F}}^{(i)}_{\mathbf{v}_j, \mathbf{z}_j } (\mathbf{v}_j, \mathbf{z}_{j} )  +
{\mathbf{F}}^{(i)}_{ \mathbf{z}_{j} } (\mathbf{z}_{j})
\\
&=&
{\mathbf{F}}^{(i)}_{\mathbf{v}_{j+1}, \mathbf{z}_{j+1} } (\mathbf{v}_{j+1}, \mathbf{z}_{j+1} )
+({\mathbf{F}}^{(i)}_{ \mathbf{z}_{j+1} })_{\mathbf{y}_{j}, \mathbf{z}_j} (\mathbf{y}_{j}, \mathbf{z}_j)
+({\mathbf{F}}^{(i)}_{ \mathbf{z}_{j+1} })_{\mathbf{z}_{j}} (\mathbf{z}_{j}),
\notag
\end{eqnarray}
where the first decomposition is via the partition $\mathbf{x} = (\mathbf{v}_j, \mathbf{z}_j )$, and the second via the partitions
$\mathbf{x} = (\mathbf{v}_{j+1}, \mathbf{z}_{j+1} )$ and $\mathbf{z}_{j+1}  = (\mathbf{y}_{j}, \mathbf{z}_{j} )$. Therefore, 
it follows from (\ref{eqn d5}) and (\ref{eqn d9}) that
\begin{eqnarray}
\label{eqn d14}
&&  {\mathbf{F}}^{(i)}_{\mathbf{v}_j, \mathbf{z}_j } (\mathbf{v}_j, \mathbf{z}_{j} )  \Big{|}_{\mathbf{y}_{j} = \mathbf{0}}
\\
\notag
&=&
\left( {\mathbf{F}}^{(i)}_{\mathbf{v}_{j+1}, \mathbf{z}_{j+1} } (\mathbf{v}_{j+1}, \mathbf{z}_{j+1} ) +({\mathbf{F}}^{(i)}_{ \mathbf{z}_{j+1} })_{\mathbf{y}_{j}, \mathbf{z}_j} (\mathbf{y}_{j}, \mathbf{z}_j) \right) \Big{|}_{\mathbf{y}_{j} = \mathbf{0}}
\\
\notag
&=&
  {\mathbf{F}}^{(i)}_{\mathbf{v}_{j+1}, \mathbf{z}_{j+1} } (\mathbf{v}_{j+1}, \mathbf{z}_{j+1} )  \Big{|}_{\mathbf{y}_{j} = \mathbf{0} \ .}
\end{eqnarray}
Consequently, we have by Lemma \ref{Lemma 3 in CM}, (\ref{eqn d14}), (\ref{eqn d6}), and (\ref{eqn d8}),
\begin{eqnarray}
\mathcal{B}_i \left( {\mathbf{F}}^{(i)}_{\mathbf{v}_j, \mathbf{z}_j } (\mathbf{v}_j, \mathbf{z}_{j} ) \right)
&\geq& \mathcal{B}_i \left( {\mathbf{F}}^{(i)}_{\mathbf{v}_j, \mathbf{z}_j } (\mathbf{v}_j, \mathbf{z}_{j} ) \Big{|}_{\mathbf{y}_{j} = \mathbf{0}} \right)
\label{eqn d13}
\\
&\geq&
\mathcal{B}_i \left( {\mathbf{F}}^{(i)}_{\mathbf{v}_{j+1}, \mathbf{z}_{j+1} } (\mathbf{v}_{j+1}, \mathbf{z}_{j+1} ) \right) - (r_i + 1) | \mathbf{y}_j |
\notag
\\
&\geq&
\left( C_{i,1} -  (r_i + 1) \sum_{\ell = j+1}^{i-1} C_{\ell,1} r_{\ell} \right) - (r_i + 1) C_{j,1} r_j
\notag
\\
&\geq&
C_{i,1} - (r_i + 1)  \sum_{\ell = j}^{i-1} C_{\ell,1} r_{\ell}.
\notag
\end{eqnarray}
Hence, from (\ref{eqn d8}), (\ref{eqn d10}), (\ref{eqn d11}), (\ref{eqn d12}), and (\ref{eqn d13}), we see that we have completed induction.
From the $j=2$ case with $\mathbf{v}_2 = (\mathbf{y}_d, \mathbf{y}_{d-1}, \ldots, \mathbf{y}_2 )$ and $\mathbf{z}_2$,
we can proceed in the exact same manner as above to deal with the linear forms
even though the definition of the Birch rank, $\mathcal{B}_1$, is slightly different than that for the higher degrees. We can do so because Proposition \ref{Prop 2 in CM} is still applicable with $\mathcal{B}_1$ for systems of linear forms. By letting the resulting variables $\mathbf{v}_1 = (\mathbf{y}_d, \mathbf{y}_{d-1}, \ldots, \mathbf{y}_2, \mathbf{y}_1 )$
and $\mathbf{z}_1$ be $\mathbf{y}$ and $\mathbf{z}$, respectively, we
complete the proof of Proposition \ref{prop decomp}.
\end{proof}

\section{Initial set-up to prove Theorem \ref{main thm}}
\label{sec initial set up}
Let $\mathbf{f} = (\mathbf{f}_{d}, \ldots, \mathbf{f}_{1})$ be a system of polynomials in $\mathbb{Z}[x_1, \ldots, x_n]$, where
$\mathbf{f}_{\ell} = (f_{\ell, 1},$ $\ldots,$ $f_{\ell, r_{\ell}})$ is the subsystem of degree $\ell$ polynomials of $\mathbf{f}$ $(1 \leq \ell \leq d)$.
We let $\mathbf{F} = (\mathbf{F}_{d}, \ldots, \mathbf{F}_{1})$ be the system of forms
such that for each $1 \leq \ell \leq d$, we have $\mathbf{F}_{\ell} = ( F_{\ell, 1}, \ldots, F_{\ell, r_{\ell}} )$
where $F_{\ell, r}$ is the homogeneous degree $\ell$ portion of $f_{\ell, r}$ $(1 \leq r \leq r_{\ell})$.
Recall in Theorem \ref{main thm} we consider the system of equations
\begin{eqnarray}
\label{set of eqn 1}
f_{\ell,r} (\mathbf{x})= 0 \ ( 1 \leq \ell \leq d, 1 \leq r \leq r_{\ell}),
\end{eqnarray}
where for each $1 \leq \ell \leq d$, $\mathcal{B}_{\ell}( \mathbf{f}_{\ell})$ is sufficiently large with respect to $d$ and $r_d$, $\ldots$, $r_1$.
Also recall we denote the integer solutions of these equations 
by $V_{\mathbf{f}, \mathbf{0}}(\mathbb{Z})$.
In order to prove Theorem \ref{main thm}, we begin by simplifying the polynomials in (\ref{set of eqn 1}) to satisfy more properties suitable for our purposes without changing the solution set $V_{\mathbf{f}, \mathbf{0}}(\mathbb{Z})$.

By reducing the polynomials in (\ref{set of eqn 1}) without changing the solution set, we transform system (\ref{set of eqn 1}) into the following system:
\begin{eqnarray}
\label{set of eqn 2}
f_{\ell,r} (\mathbf{x})= 0 \ (1 \leq \ell \leq d, 1 \leq r \leq r_{\ell}),
\end{eqnarray}
where
for $2 \leq \ell \leq d$,
$$
f_{\ell,r} (\mathbf{x}) = c_{\ell, r} \mathbf{w}^{\mathbf{j}_{\ell, r}} + \chi_{\ell, r}(\mathbf{x}) + \widetilde{f}_{\ell,r}(\mathbf{x})
\ \ (1 \leq r \leq r_{\ell})
$$
and
$$
f_{1,r} (\mathbf{x}) = {c}_{1, r} \mathbf{w}^{\mathbf{j}_{1, r}} + \widetilde{f}_{1,r}(\mathbf{x}) \ \ (1 \leq r \leq r_1)
$$
with the following properties. Here $\mathbf{w}$ is a subset of the variables $\mathbf{x} = (x_1, \ldots, x_n)$.

$(1)$ For each $1 \leq \ell \leq d, 1 \leq r \leq r_{\ell}$, we have $c_{\ell, r} \in \mathbb{Z} \backslash \{ 0 \}$, and
$\mathbf{w}^{\mathbf{j}_{\ell, r}}$ is the leading monomial of $f_{\ell,r}(\mathbf{x})$ with respect to a graded lexicographic ordering.
We also note $\mathbf{w}^{\mathbf{j}_{\ell, r}}$ has degree $\ell$.
\newline

$(2)$ The monomials $\mathbf{w}^{\mathbf{j}_{\ell, r}}$ are distinct, and given $1 \leq \ell \leq d, 1 \leq r \leq r_{\ell}$, $\mathbf{w}^{\mathbf{j}_{\ell, r}}$ is not divisible by any one of $\mathbf{w}^{\mathbf{j}_{\ell', r'}}$ $(1 \leq \ell' < \ell, 1 \leq r' \leq r_{\ell'})$.
\newline

$(3)$ For each $2 \leq \ell \leq d, 1 \leq r \leq r_{\ell}$, the polynomial $\chi_{\ell, r}(\mathbf{x})$ has degree less than or equal to $\ell$ with coefficients in $\mathbb{Z}$. Also $\chi_{\ell, r}(\mathbf{x})$ does not contain any monomial divisible by any one of $\mathbf{w}^{\mathbf{j}_{\ell', r'}}$ $(1 \leq \ell' \leq \ell, 1 \leq r' \leq r_{\ell'})$.
\newline

$(4)$ For each $1 \leq \ell \leq d, 1 \leq r \leq r_{\ell}$, the polynomial $\widetilde{f}_{\ell,r}(\mathbf{x})$ has degree $\ell$ with coefficients in $\mathbb{Z}$. Also  $\widetilde{f}_{\ell,r}(\mathbf{x})$ does not contain any monomial divisible by any one of $\mathbf{w}^{\mathbf{j}_{\ell', r' }}$ $(1 \leq \ell' \leq \ell, 1 \leq r' \leq r_{\ell'})$.
\newline

$(5)$ For each $2 \leq \ell \leq d, 1 \leq r \leq r_{\ell}$, we have
$$
h_{\ell}(\chi_{\ell, r}) \leq C''_0,
$$
where $C''_0$ is a constant depending only on $d$ and $r_d, \ldots, r_1$.
\newline

$(6)$  For each $1 \leq \ell \leq d$, $\mathcal{B}_{  \ell}( \{ \widetilde{f}_{\ell, r} : 1 \leq r \leq r_{\ell}  \})$
is sufficiently large with respect to $d$ and $r_d, \ldots, r_1$.
\newline

$(7)$  For each $2 \leq \ell \leq d$, $h_{ \ell}( \mathbf{f}_{\ell})$
is sufficiently large with respect to $d$ and $r_d, \ldots, r_1$, and $\mathcal{B}_{1}({\mathbf{f}}_{1})$
is sufficiently large with respect to $d$ and $r_d, \ldots, r_1$.
\newline

These conditions system (\ref{set of eqn 2}) satisfies become crucial during our minor arc estimate.
Before we describe this reduction process, first we note basic properties of the Birch rank which will be utilized in this section.
Let $\ell \geq 1$ and $\mathbf{G} = \{ G_1, \ldots, G_{r''} \}$ be a system of degree $\ell$ forms in $\mathbb{Q}[x_1, \ldots, x_n]$.
Let $\kappa_1, \ldots, \kappa_{r''} \in \mathbb{Q} \backslash \{ 0 \}$. Then it follows from the definition of the Birch rank that
$$
\mathcal{B}_{\ell} ( \{  \kappa_{r} G_{r} : 1 \leq r \leq r'' \} ) = \mathcal{B}_{\ell} ( \mathbf{G} ).
$$
Let $\kappa \in \mathbb{Q}$ and $1 \leq i, j \leq r''$. Let $G'_r = G_r$ if $r \not = i$, and $G'_i = G_i + \kappa G_j$.
It also follows from the definition of the Birch rank that
$$
\mathcal{B}_{\ell} ( \{  G'_{r} : 1 \leq r \leq r'' \} ) = \mathcal{B}_{\ell} ( \mathbf{G} ).
$$

We now transform system (\ref{set of eqn 1}) into system (\ref{set of eqn 2}) beginning with $\ell =1$.
By considering the reduced row echelon form of the matrix formed by the coefficients of $ F_{1,r} , \ldots, F_{1, r_1}$, and relabeling the
variables if necessary, we reduce the linear polynomials in (\ref{set of eqn 1}) without changing the solution set to be of the shape
$$
f_{1,r} (\mathbf{x}) = x_{n - r + 1} + \widetilde{f}_{1,r}(x_{1}, \ldots, x_{n - r_1}) \  \ (1 \leq r \leq r_1),
$$
where $\widetilde{f}_{1,r}(x_{1}, \ldots, x_{n - r_1 })$ is a linear polynomial in variables $x_{1}, \ldots, x_{n - r_1}$ with rational coefficients.
Then by substituting $ x_{n - r +1} = - \widetilde{f}_{1,r}(x_{1}, \ldots, x_{n-r_1})$ into each equation in (\ref{set of eqn 1}) with $\ell > 1$, we may further reduce without changing the solution set such that for $\ell > 1$ the polynomials $f_{\ell,r}$ do not involve any of the variables $x_{n - r_1 +1}, \ldots, x_{n}$.
Let us label $w_{r} = x_{n - r +1}$ $(1 \leq r \leq r_1)$. 
By multiplying each of the resulting equation by an integer constant if necessary, we replace system (\ref{set of eqn 1}) with the following system of equations
\begin{eqnarray}
\label{set of eqn 1'}
f_{\ell,r} (\mathbf{x})= 0 \ ( 1 \leq \ell \leq d, 1 \leq r \leq r_{\ell}),
\end{eqnarray}
where
$$
f_{\ell,r} (\mathbf{x}) \in \mathbb{Z}[x_1, \ldots, x_{n-r_1} ] \ \ (  1 < \ell \leq d, 1 \leq r \leq r_{\ell}),
$$
and for each $1 \leq r \leq r_1$, we have
$$
f_{1,r} (\mathbf{x}) = {c}_{1, r} {w}_{r} + \widetilde{f}_{1,r}(x_{1}, \ldots, x_{n-r_1})
$$
with ${c}_{1, r} \in \mathbb{Z} \backslash \{ 0 \}$ and $\widetilde{f}_{1,r}(x_{1}, \ldots, x_{n-r_1}) \in \mathbb{Z}[ x_1, \ldots, x_{n - r_1}]$.
From the definition of the Birch rank, we have that $\mathcal{B}_1(\mathbf{f}_1)$ remains the same under these changes.
Therefore, it follows by Lemma \ref{Lemma 3 in CM} that
$\mathcal{B}_{1}( \{ \widetilde{f}_{1, r} : 1 \leq r \leq r_{1}  \})$
is sufficiently large with respect to $d$ and $r_d, \ldots, r_1$.
For $1 < \ell \leq d$, we can deduce from Corollary \ref{Cor 4 in CM} that we still have
$\mathcal{B}_{\ell}(\mathbf{f}_{\ell})$ sufficiently large with respect to $d$ and  $r_d$, $\ldots$, $r_1$.
Let us put a graded lexicographic ordering on the monomials formed by $x_1, \ldots, x_{n-r_1},$ $w_1, \ldots , w_{r_1}$
such that $w_r$ is the leading coefficient of $f_{1,r}$.
By denoting $\mathbf{w}^{\mathbf{j}_{1, r}} =  w_r \ (1 \leq r \leq r_{1})$, we see that the linear polynomials in (\ref{set of eqn 1'}) satisfy the conditions of system (\ref{set of eqn 2}).
We note that for us this graded lexicographic ordering is essentially on the monomials formed by $x_1, \ldots, x_{n-r_1}$ as $w_1, \ldots, w_r$
do not appear in  $f_{\ell, r}$ with $\ell > 1$.

Let us denote $B_{2} := \mathcal{B}_{2}(\mathbf{f}_{2})$ for the Birch rank of $\mathbf{f}_{2}$ in (\ref{set of eqn 1'}).
We consider $\mathbf{F}_{2} = (F_{2,1}, \ldots, F_{2, r_2})$, the system of homogenous degree $2$ portions of $\mathbf{f}_2$.
From each
$F_{2,1}, \ldots, F_{2, r_2}$,
we collect the coefficient of the monomial $x_{i_1} x_{i_2}$
and turn it into a vector in $\mathbb{Q}^{r_{2}}$. We do this for every degree $2$ monomial.
We then form a matrix by putting these vectors in columns from left to right in the decreasing order of the
degree $2$ monomials. Since $B_2 = \mathcal{B}_2(\mathbf{F}_{2})> 0$, this matrix has full rank.
We row reduce this matrix, and we denote the $r_2$ monomials where the leading $1$'s occur
to be $\mathbf{w}^{\mathbf{j}_{2, r}} \ (1 \leq r \leq r_2)$,
and label the distinct variables involved in these $r_{2}$ monomials to be
$\mathbf{w}_2 = (w_{r_1 + 1}, \ldots, w_{r_1 + K_2})$. Clearly we have $K_2 \leq 2 r_2$.

From the row reduction operations done on the coefficient matrix of $\mathbf{F}_2$, without changing the solution set we can reduce $\mathbf{f}_{2}$ to
\begin{equation}
\label{set of eqn 1''}
f_{2,r}(\mathbf{x}) = c_{2,r} \mathbf{w}^{ \mathbf{j}_{2,r} } + \widetilde{f}_{2,r}(\mathbf{x}) \ \ (1 \leq r \leq r_2),
\end{equation}
where $\mathbf{w}^{ \mathbf{j}_{2,r} }$ is the leading monomial of $f_{2,r}(\mathbf{x})$, with respect to the graded lexicographic ordering, and
none of the monomials of $\widetilde{f}_{2,r}(\mathbf{x})$ is divisible by any one of $\mathbf{w}^{\mathbf{j}_{\ell', r'}}$ $( 1 \leq \ell' \leq 2,  1 \leq r \leq r_{\ell'})$.
We have that $c_{2,r} \mathbf{w}^{ \mathbf{j}_{2,r} } + \widetilde{f}_{2,r}(\mathbf{x}) $
is a $\mathbb{Q}$-linear combination of $f_{2,1}, \ldots, f_{2, r_2}$ in (\ref{set of eqn 1'}), where the $\mathbb{Q}$-linear combination
comes from the row reduction operations applied on the coefficient matrix described above.
Thus by the basic properties of the Birch rank, it follows that
$$
\mathcal{B}_{2}( \{ c_{2,r} \mathbf{w}^{ \mathbf{j}_{2,r}   } + \widetilde{f}_{2,r} (\mathbf{x})  : 1 \leq r \leq r_{2} \} ) = B_2.
$$
It then follows from (\ref{ineq of h and B}) that the $h$-invariant of $\mathbf{f}_2$ in (\ref{set of eqn 1''}) satisfies
$$
h_2(\mathbf{f}_2) \geq 2^{1-2} B_2,
$$
and hence $h_2(\mathbf{f}_2)$ is sufficiently large with respect to $d$ and $r_d, \ldots, r_1$.
We also have by Lemma \ref{Lemma 3 in CM},
\begin{eqnarray}
\mathcal{B}_{2}( \{  \widetilde{f}_{2,r} (\mathbf{x})  : 1 \leq r \leq r_{2} \} )
&\geq& \mathcal{B}_{2}( \{  \widetilde{f}_{2,r} (\mathbf{x}) |_{ \mathbf{w}_{2} = \mathbf{0} } : 1 \leq r \leq r_{2} \} )
\notag
\\
&=& \mathcal{B}_{2}( \{ ( c_{2,r} \mathbf{w}^{ \mathbf{j}_{2,r}} +  \widetilde{f}_{2,r} (\mathbf{x})  )|_{ \mathbf{w}_{2} = \mathbf{0} }   : 1 \leq r \leq r_{2} \} )
\notag
\\
&\geq& \mathcal{B}_{2}( \{ c_{2,r} \mathbf{w}^{ \mathbf{j}_{2,r}} + \widetilde{f}_{2,r} (\mathbf{x})  : 1 \leq r \leq r_{2} \} ) - (r_2 + 1) K_2
\notag
\\
&=& B_2 - (r_2 + 1) K_2.
\notag
\end{eqnarray}
Thus $\mathcal{B}_{2}( \{  \widetilde{f}_{2,r} (\mathbf{x})  : 1 \leq r \leq r_{2} \} )$ is sufficiently large with respect to
$d$ and $r_d, \ldots, r_1$. It is also clear that $\mathbf{w}^{\mathbf{j}_{2, r}}$ is not divisible by any one of $\mathbf{w}^{\mathbf{j}_{1, r'}} = w_{r'}$ $(1 \leq r' \leq r_{1})$. Therefore, we have obtained that we can reduce the degree $2$ polynomials of system (\ref{set of eqn 1'}) without changing the solution set to satisfy the conditions of system (\ref{set of eqn 2}) with $\chi_{2,r}(\mathbf{x})$ being the zero polynomial $(1 \leq r \leq r_2)$.

Using the $\ell = 2$ case as the base case, we prove our statement by induction. Let $\ell_0 \geq 3$. Suppose we have reduced the polynomials $\mathbf{f}_{\ell}$ in (\ref{set of eqn 1'}) for each $2 \leq \ell \leq \ell_0 - 1$, without changing the solution set, to satisfy the conditions of (\ref{set of eqn 2}).
First we take the distinct variables involved in the monomials $\mathbf{w}^{ \mathbf{j}_{3,r}}$ $(1 \leq r \leq r_3)$ that have not yet appeared
in $\mathbf{w}_{2}$, and label them as $w_{ r_1 + K_2+ 1}, \ldots, w_{ r_1 + K_2+ K_3 }$. Clearly we have $K_{3} \leq 3 r_{3}$. We adjoin these variables to $\mathbf{w}_{2}$, and let $\mathbf{w}_{3} = (w_{r_1 + 1}, \ldots, w_{r_1 + K_2 + K_3 } )$.  Then we take the distinct variables involved in the monomials $\mathbf{w}^{ \mathbf{j}_{4,r}}$ $(1 \leq r \leq r_4)$ that have not yet appeared
in $\mathbf{w}_{3}$, and label them as $w_{ r_1 + K_2+ K_3 + 1}, \ldots, w_{ r_1 + K_2+ K_3 + K_4 }$. Clearly we have $K_{4} \leq 4 r_{4}$. We adjoin these variables to $\mathbf{w}_{3}$, and let
$\mathbf{w}_{4} = (w_{r_1 + 1}, \ldots, w_{r_1 + K_2 + K_3 + K_4 } )$.  We continue in this manner until we obtain
$$
\mathbf{w}_{\ell_0 - 1} = (w_{r_1 + 1}, \ldots, w_{r_1 + K_2 +  \ldots + K_{\ell_0 - 1} } ),
$$
where $K_j \leq j r_j$ $(2 \leq j \leq \ell_0 - 1)$.

For each $1 \leq r \leq r_{\ell_0 }$, we let
\begin{equation}
\label{set of eqn 6}
f_{\ell_0,r}(\mathbf{x}) = \chi''_{\ell_0,r} (\mathbf{x}) + f''_{\ell_0,r}(\mathbf{x}),
\end{equation}
where every monomial of $\chi''_{\ell_0,r} (\mathbf{x})$ is divisible by one of $\mathbf{w}^{\mathbf{j}_{\ell, r}}$ $( 2 \leq \ell < \ell_0,
1 \leq r \leq r_{\ell})$,
and none of the monomials of $f''_{\ell_0,r}(\mathbf{x})$ is divisible by any one of $\mathbf{w}^{\mathbf{j}_{\ell, r}}$ $( 2 \leq \ell < \ell_0,
1 \leq r \leq r_{\ell})$.
Since
$$
{f}_{\ell_0, r}(\mathbf{x}) |_{\mathbf{w}_{\ell_0 - 1} = \mathbf{0}} = {f}''_{\ell_0, r}(\mathbf{x})  |_{\mathbf{w}_{\ell_0 - 1} = \mathbf{0}} \ \ (1 \leq r \leq r_{\ell_0}),
$$
we have by Lemma \ref{Lemma 3 in CM} that
\begin{eqnarray}
\mathcal{B}_{\ell_0} ( \{ f''_{\ell_0, r} (\mathbf{x}) : 1 \leq r \leq r_{\ell_0} \} )
&\geq& \mathcal{B}_{\ell_0} ( \{  f''_{\ell_0,r} (\mathbf{x}) |_{ \mathbf{w}_{\ell_0-1} = \mathbf{0} } : 1 \leq r \leq r_{\ell_0} \} )
\notag
\\
&=& \mathcal{B}_{\ell_0}( \{  {f}_{\ell_0, r}(\mathbf{x}) |_{\mathbf{w}_{\ell_0 - 1}= \mathbf{0} }  : 1 \leq r \leq r_{\ell_0} \} )
\notag
\\
&\geq& \mathcal{B}_{\ell_0}(  \mathbf{f}_{\ell_0} ) - (r_{\ell_0} + 1) (K_2 + \ldots + K_{\ell_0 - 1}).
\notag
\end{eqnarray}
Consequently, we have that $\mathcal{B}_{\ell_0}( \{ {f}''_{\ell_0, r} : 1 \leq r \leq r_{\ell_0} \} )$ is sufficiently large
with respect to $d$ and $r_d, \ldots, r_1$.
Also it follows by basic facts about reduction in Gr\"{o}bner basis theory that we may write
\begin{equation}
\label{eqn for chi''}
\chi''_{\ell_0, r}(\mathbf{x}) = \chi'_{\ell_0, r}(\mathbf{x}) +
\sum_{2 \leq \ell' < \ell_0 } \sum_{1 \leq r' \leq r_{\ell'}} \zeta_{\ell_0, r : \ell', r'} (\mathbf{x}) f_{\ell', r} (\mathbf{x}),
\end{equation}
where $\chi'_{\ell_0, r}(\mathbf{x})$ is a polynomial which does not contain any monomial divisible by any one of $\mathbf{w}^{\mathbf{j}_{\ell, r}}$ $( 2 \leq \ell < \ell_0, 1 \leq r \leq r_{\ell})$. Furthermore, $\chi'_{\ell_0, r}(\mathbf{x})$ is a polynomial of degree less than or equal to $\ell_0$, and $\zeta_{\ell_0, r : \ell', r'} (\mathbf{x})$ is a polynomial of degree less than or equal to $\ell_0 - \ell'$.
We obtain by the definition of the $h$-invariant that
\begin{equation}
\label{h inv bound on chi'}
h_{\ell_0}(\chi'_{\ell_0,r}) \leq h_{\ell_0}(\chi''_{\ell_0,r} (\mathbf{x})) + h_{\ell_0} \left( \sum_{2 \leq \ell' < \ell_0 } \sum_{1 \leq r' \leq r_{\ell'}} \zeta_{\ell_0, r : \ell', r'} (\mathbf{x}) f_{\ell', r} (\mathbf{x}) \right) \leq \sum_{\ell = 2}^{\ell_0-1} r_{\ell} + \sum_{\ell = 2}^{\ell_0-1} r_{\ell}
\end{equation}
for each $1 \leq r \leq r_{\ell_0}$.
Also, via (\ref{eqn for chi''}) we can reduce $\mathbf{f}_{\ell_0}$ of (\ref{set of eqn 6}) without changing the solution set, and assume it is of the shape
\begin{equation}
\label{set of eqn 4}
f_{\ell_0,r}(\mathbf{x}) = \chi'_{\ell_0,r} (\mathbf{x}) + f''_{\ell_0,r}(\mathbf{x}) \ \ (1 \leq r \leq r_{\ell_0}),
\end{equation}
where none of the monomials of $f''_{\ell_0,r}(\mathbf{x})$ or $\chi'_{\ell_0,r} (\mathbf{x})$ is divisible by any one of
$\mathbf{w}^{\mathbf{j}_{\ell, r}}$ $( 2 \leq \ell < \ell_0, 1 \leq r \leq r_{\ell})$.

We then consider $\mathbf{F}_{\ell_0} = (F_{\ell_0,1}, \ldots, F_{\ell_0, r_{\ell_0} })$ where $F_{\ell_0,r}$ is the homogeneous degree $\ell_0$ portion of $f_{\ell_0,r}$ in (\ref{set of eqn 4}).
From each $F_{\ell_0,1}, \ldots, F_{\ell_0, r_{\ell_0}}$,
we collect the coefficient of the monomial $x_{i_1} \ldots x_{i_{\ell_0}}$
and turn it into a vector in $\mathbb{Q}^{r_{\ell_0}}$. We do this for every degree $\ell_0$ monomial.
We then form a matrix by putting these vectors in columns from left to right in the decreasing order of the
degree $\ell_0$ monomials. From the definition of the $h$-invariant, we can deduce
$$
h_{\ell_0}( \mathbf{F}_{\ell_0} ) + \sum_{r=1}^{r_{\ell_0}} h_{\ell_0}( \chi'_{\ell_0,r} (\mathbf{x}) )
\geq h_{\ell_0}( \{ f''_{\ell_0,r}(\mathbf{x}) :1 \leq r \leq r_{\ell_0} \} ).
$$
Consequently, we obtain from (\ref{ineq of h and B}) and (\ref{h inv bound on chi'}) that
\begin{eqnarray}
h_{\ell_0}( \mathbf{F}_{\ell_0} )
&\geq& 2^{1 - \ell_0} \ \mathcal{B}_{\ell_0}( \{ f''_{\ell_0,r}(\mathbf{x}) :1 \leq r \leq r_{\ell_0} \} ) -
2 r_{\ell_0} \sum_{\ell = 2}^{\ell_0 - 1} r_{\ell}.
\notag
\end{eqnarray}
Thus it follows that $h_{\ell_0}( \mathbf{F}_{\ell_0} )$ is sufficiently large with respect to
$d$ and $r_d, \ldots, r_1$. In particular, since we have $h_{\ell_0}( \mathbf{F}_{\ell_0} )> 0$, the coefficient matrix of $\mathbf{F}_{\ell_0}$ above has full rank.
We row reduce this matrix, and we denote the $r_{\ell_0}$ monomials where the leading $1$'s occur
to be $\mathbf{w}^{\mathbf{j}_{\ell_0, r}} \ (1 \leq r \leq r_{\ell_0})$.
We then take the distinct variables involved in these $r_{\ell_0}$ monomials that have not yet appeared
in $\mathbf{w}_{\ell_0 - 1}$, and label them as
$$
w_{ r_1 + K_2 + \ldots + K_{\ell_0 - 1} +  1}, \ \ldots, \  w_{r_1 + K_2 + \ldots + K_{\ell_0 - 1} +  K_{\ell_0} }.
$$
Clearly we have $K_{\ell_0} \leq \ell_0 r_{\ell_0}$.
We adjoin these variables to $\mathbf{w}_{\ell_0 - 1}$, and let
$\mathbf{w}_{\ell_0} = (w_{r_1 + 1},$ $\ldots,$ $w_{r_1 + K_2 + \ldots + K_{\ell_0} } )$.

From the row reduction operations done on the coefficient matrix, without changing the solution set we can reduce $\mathbf{f}_{\ell_0}$ to
\begin{equation}
\label{set of eqn 5}
f_{\ell_0,r}(\mathbf{x}) = c_{\ell_0,r} \mathbf{w}^{ \mathbf{j}_{\ell_0,r} } + \chi_{\ell_0,r} (\mathbf{x}) + \widetilde{f}_{\ell_0,r}(\mathbf{x}) \ \ (1 \leq r \leq r_{\ell_0}),
\end{equation}
where $\mathbf{w}^{ \mathbf{j}_{\ell_0,r} }$ is the leading monomial of $f_{\ell_0,r}$, with respect to the graded lexicographic ordering, and
none of the monomials of $\widetilde{f}_{\ell_0,r}(\mathbf{x})$ or $\chi_{\ell_0,r} (\mathbf{x})$ is divisible by any one of $\mathbf{w}^{\mathbf{j}_{\ell', r'}}$ $(1 \leq \ell' \leq \ell_0, 1 \leq r' \leq r_{\ell})$.
Also $\chi_{\ell_0,r} (\mathbf{x}) + c^{(1)}_{\ell_0,r} \mathbf{w}^{ \mathbf{j}_{\ell_0,r} }$ is a $\mathbb{Q}$-linear combination of $\chi'_{\ell_0,1}, \ldots, \chi'_{\ell_0, r_{\ell_0}}$,
and similarly $\widetilde{f}_{\ell_0,r} (\mathbf{x}) + c^{(2)}_{\ell_0,r} \mathbf{w}^{ \mathbf{j}_{\ell_0,r} }$ is a $\mathbb{Q}$-linear combination of $f''_{\ell_0,1}, \ldots, f''_{\ell_0,r_{\ell_0}}$ for some appropriate rational coefficients $c^{(1)}_{\ell_0,r}$ and $c^{(2)}_{\ell_0,r},$ where
$ c^{(1)}_{\ell_0,r} + c^{(2)}_{\ell_0,r} = c_{\ell_0,r}.$
It then follows by the definition of the $h$-invariant and (\ref{h inv bound on chi'}) that
$$
h(\chi_{\ell_0,r}) \leq 1 + \sum_{r = 1}^{ r_{\ell_0} } h_{\ell_0} ( \chi'_{\ell_0,r} )  \leq 1 + 2 r_{\ell_0}  \sum_{\ell = 2}^{\ell_0 - 1} r_{\ell} \ \ (1 \leq r \leq r_{\ell_0}).
$$
We obtained $\widetilde{f}_{\ell_0,r} (\mathbf{x}) + c^{(2)}_{\ell_0,r} \mathbf{w}^{ \mathbf{j}_{\ell_0,r} }$
as a $\mathbb{Q}$-linear combination of $f''_{\ell_0,1}, \ldots, f''_{\ell_0,r_{\ell_0}}$, where the $\mathbb{Q}$-linear combination
came from the row reduction operations applied to the coefficient matrix of $\mathbf{F}_{\ell_0}$.
Thus by the basic properties of the Birch rank, it follows that
$$
\mathcal{B}_{\ell_0}( \{  \widetilde{f}_{\ell_0,r} (\mathbf{x}) + c^{(2)}_{\ell_0,r} \mathbf{w}^{ \mathbf{j}_{\ell_0,r} } : 1 \leq r \leq r_{\ell_0} \} ) = \mathcal{B}_{\ell_0}( \{ f''_{\ell_0,r} :  1 \leq r \leq r_{\ell_0} \} ).
$$
Therefore, we obtain by Lemma \ref{Lemma 3 in CM} that
\begin{eqnarray}
&&\mathcal{B}_{\ell_0}( \{  \widetilde{f}_{\ell_0,r} (\mathbf{x})  : 1 \leq r \leq r_{\ell_0} \} )
\notag
\\
&\geq& \mathcal{B}_{\ell_0}( \{  \widetilde{f}_{\ell_0,r} (\mathbf{x}) |_{ \mathbf{w}_{\ell_0} = \mathbf{0} } : 1 \leq r \leq r_{\ell_0} \} )
\notag
\\
&=& \mathcal{B}_{\ell_0}( \{  (\widetilde{f}_{\ell_0,r} (\mathbf{x}) + c^{(2)}_{\ell_0,r} \mathbf{w}^{ \mathbf{j}_{\ell_0,r}} )|_{ \mathbf{w}_{\ell_0} = \mathbf{0} }   : 1 \leq r \leq r_{\ell_0} \} )
\notag
\\
&\geq& \mathcal{B}_{\ell_0}( \{ \widetilde{f}_{\ell_0,r} (\mathbf{x}) + c^{(2)}_{\ell_0,r} \mathbf{w}^{ \mathbf{j}_{\ell_0,r}   } : 1 \leq r \leq r_{\ell_0} \} ) - (r_{\ell_0} + 1) (K_2 + \ldots  + K_{\ell_0})
\notag
\\
&=& \mathcal{B}_{\ell_0}( \{ f''_{\ell_0,r} :  1 \leq r \leq r_{\ell_0} \} ) - (r_{\ell_0} + 1) (K_2 + \ldots  + K_{\ell_0}).
\notag
\end{eqnarray}
Thus we have that $\mathcal{B}_{\ell_0}( \{  \widetilde{f}_{\ell_0,r} (\mathbf{x})  : 1 \leq r \leq r_{\ell_0} \} )$ is sufficiently large with respect to $d$ and $r_d, \ldots, r_1$.
It then follows by a similar argument given above, to show the $h$-invariant of $\mathbf{f}_{\ell_0}$ in (\ref{set of eqn 4}) is sufficiently large, that
the $h$-invariant of $\mathbf{f}_{\ell_0}$ in (\ref{set of eqn 5}) is sufficiently large with respect to $d$ and $r_d, \ldots, r_1$.
Finally, we also have by the construction that for each $1 \leq r \leq r_{\ell_0}$, the monomial $\mathbf{w}^{\mathbf{j}_{\ell_0,r}}$ is not divisible by any one of $\mathbf{w}^{\mathbf{j}_{\ell, r}}$ $( 1 \leq \ell < \ell_0, 1 \leq r \leq r_{\ell})$.
Thus we have completed induction. Therefore, we obtain that we can transform system (\ref{set of eqn 1}) into system (\ref{set of eqn 2})
without changing the solution set. Let us adjoin $\mathbf{w}_d$ to $(w_1, \ldots, w_{r_1})$ and denote the resulting
set of variables to be $\mathbf{w} = (w_1, \ldots, w_K)$, where
\begin{equation}
K = r_1 + K_2 + \ldots + K_d \leq \sum_{\ell = 1}^{d} \ell r_{\ell}.
\end{equation}
We also add that if there are any $\ell$ with $r_{\ell} = 0$, we simply skip these cases in the above argument.

Let $\alpha_{\ell, r} \in \mathbb{R}$ $(1 \leq \ell \leq d, 1 \leq r \leq r_{\ell})$, and consider
$$
\sum_{1 \leq \ell \leq d} \sum_{1 \leq r \leq r_{\ell}} \alpha_{\ell, r} f_{\ell, r} (\mathbf{x})
$$
as a polynomial in $x_1, \ldots, x_n$ with real coefficients, where $\mathbf{f}$ is the system of polynomials in (\ref{set of eqn 2}).
Given any $1 \leq \ell \leq d$ and $1 \leq r \leq r_{\ell}$, it follows from the construction that
the coefficient of $\mathbf{w}^{\mathbf{j}_{\ell, r} }$ of the above polynomial is $c_{\ell, r}  \alpha_{\ell, r}$.
Let $\mathbf{x} = (\mathbf{w}, \mathbf{x}')$. Let us also fix $\mathbf{x}' = \mathbf{x}'_0 \in \mathbb{Z}^{n - K}$.
It is clear that if we consider
\begin{equation}
\label{discussion on the coeff}
\sum_{1 \leq \ell \leq d} \sum_{1 \leq r \leq r_{\ell}} \alpha_{\ell, r} f_{\ell, r} (\mathbf{w}, \mathbf{x}'_0)
\end{equation}
as a polynomial in $\mathbf{w}$ with real coefficients, then given $1 \leq \ell \leq d, 1 \leq r \leq r_{\ell}$ we still have that
the coefficient of $\mathbf{w}^{\mathbf{j}_{\ell, r} }$ of this polynomial is $c_{\ell, r}  \alpha_{\ell, r}$.
Furthermore, this polynomial does not contain any monomial divisible by $\mathbf{w}^{\mathbf{j}_{\ell, r} }$ other than itself.

We set $R = r_d +  \ldots + r_1$. Let $\boldsymbol{\alpha} = ( \boldsymbol{\alpha}_{d}, \ldots, \boldsymbol{\alpha}_{1}) \in \mathbb{R}^R$ where
$\boldsymbol{\alpha}_{\ell} =  ( {\alpha}_{\ell,1}, \ldots, {\alpha}_{\ell, r_{{\ell}} } ) \in \mathbb{R}^{r_{\ell}}$ $(1 \leq \ell \leq d)$. Similarly, we denote
$\mathbf{a}=(\mathbf{a}_{d}, \ldots, \mathbf{a}_{1}) \in (\mathbb{Z}/q \mathbb{Z})^{R}$, where $q \in \mathbb{N}$ and
$\mathbf{a}_{\ell}=({a}_{\ell, 1}, \ldots, {a}_{\ell, r_{\ell}}) \in (\mathbb{Z}/q \mathbb{Z})^{r_{\ell}}$ $(1 \leq \ell \leq d)$.
Let $\mathbb{T} = \mathbb{R} / \mathbb{Z}$ and $\| \beta \|$ denote the distance from $\beta \in \mathbb{R}$ to the nearest integer, which induces a metric on $\mathbb{T}$ via $d(\alpha, \beta) = \| \alpha - \beta \|$.
For a given value of $C>0$ and an integer $1 \leq q \leq (\log X)^C$, we define
$$
\mathfrak{M}_{ \mathbf{a}, q}(C) = \{  \boldsymbol{\alpha} \in \mathbb{T}^{R} : \max_{1 \leq r \leq r_{\ell}  } \| \alpha_{\ell,r} -  a_{\ell,r} / q \| \leq X^{-\ell} (\log X)^C  \ \ (1 \leq \ell \leq d) \}
$$
for each $\mathbf{a} \in (\mathbb{Z}/q \mathbb{Z})^{R}$ with $\gcd(\mathbf{a}, q) = 1$ (by which we mean that the greatest common divisor of
the numbers $a_{d,1}, \ldots , a_{1,r_1}$ and $q$ is $1$).
These arcs are disjoint for $X$ sufficiently large.

We define the \textit{major arcs} to be
$$
\mathfrak{M}(C) = \bigcup_{q \leq (\log X)^C } \bigcup_{ \substack{
\mathbf{a} \in (\mathbb{Z}/q \mathbb{Z})^{R}
\\
\gcd(\mathbf{a}, q) = 1
} }  \mathfrak{M}_{\mathbf{a}, q}(C),
$$
and define the \textit{minor arcs} to be
$$
\mathfrak{m}(C) = \mathbb{T}^{R} \backslash \mathfrak{M}(C).
$$
In other words, the major arcs is a collection of elements in $\mathbb{T}^R$ that can be simultaneously `well approximated'
by rational numbers of the same denominator $q$, where $q$ is `small'.

For a system of polynomials $\mathbf{f}$, we define
\begin{equation}
\label{def of exp sum T}
T(\mathbf{f}; \boldsymbol{\alpha}) := \sum_{\mathbf{x} \in [0, X]^n } \Lambda(\mathbf{x}) \
e \left(
\sum_{\ell = 1}^d \sum_{r=1}^{r_{\ell} } {f}_{\ell, r}(\mathbf{x}) \cdot {\alpha_{\ell,r} }
\right),
\end{equation}
where we defined $\Lambda(\mathbf{x})$ in (\ref{Lambda}).
By the orthogonality relation, we have
\begin{eqnarray}
\label{defn of MbN'}
\mathcal{M}_{\mathbf{f}}(X) &=& \sum_{\mathbf{x} \in [0, X]^n} \Lambda(\mathbf{x}) \ \mathbf{1}_{V_{\mathbf{f}, \mathbf{0}(\mathbb{C}) }}(\mathbf{x} )
\\
&=&
\int_{\mathbb{T}} \ldots \int_{\mathbb{T}} T(\mathbf{f}; \boldsymbol{\alpha})
\ \mathbf{d} \boldsymbol{\alpha}
\notag
\\
\notag
&=& \int_{\mathfrak{M}(C)} T(\mathbf{f}; \boldsymbol{\alpha} ) \ \mathbf{d} \boldsymbol{\alpha}
+ \int_{\mathfrak{m}(C)} T(\mathbf{f}; \boldsymbol{\alpha} ) \ \mathbf{d} \boldsymbol{\alpha}.
\end{eqnarray}
For the system of polynomials $\mathbf{f}$ in (\ref{set of eqn 2}), we prove the following results on the minor arcs and the major arcs.
\begin{prop}
\label{prop minor arc bound}
Let $\mathbf{f}$ be the polynomials in (\ref{set of eqn 2}).
Given any $c>0$, for sufficiently large $C>0$ we have
$$
\int_{\mathfrak{m}(C)} T(\mathbf{f}; \boldsymbol{\alpha} ) \ \mathbf{d} \boldsymbol{\alpha}
\ll
\frac{ X^{n - \sum_{\ell = 1}^d \ell r_{\ell}  } }{(\log X)^{c}}.
$$
\end{prop}

\begin{prop}
\label{prop major arc bound}
Let $\mathbf{f}$ be the polynomials in (\ref{set of eqn 2}).
Given any $c>0$, for sufficiently large $C>0$ we have
$$
\int_{\mathfrak{M}(C)} T(\mathbf{f}; \boldsymbol{\alpha} ) \ \mathbf{d} \boldsymbol{\alpha}
=  \mathcal{C}(\mathbf{f})  \  X^{n - \sum_{\ell = 1}^d \ell r_{\ell}  } + O\left( \frac{ X^{n - \sum_{\ell = 1}^d \ell r_{\ell}  } }{(\log X)^{c}} \right),
$$
where $\mathcal{C}(\mathbf{f})$ is a constant that depends only on $\mathbf{f}$.
\end{prop}
We prove Proposition \ref{prop minor arc bound} in Section \ref{section minor arc}, and Proposition \ref{prop major arc bound} in Section \ref{section major arcs}.

\section{Hardy-Littlewood Circle Method: Minor Arcs}
\label{section minor arc}

\begin{proof}[Proof of Proposition \ref{prop minor arc bound}]
We consider the system of polynomials $\mathbf{f}$ in (\ref{set of eqn 2}) constructed in the previous section, which satisfies all the conditions described below (\ref{set of eqn 2}). Recall we denote $\mathbf{w} = (w_1, \ldots, w_K)$, where $K \leq d R$ and $R = \sum_{\ell = 1}^d r_{\ell}$.
We let $\widetilde{\mathbf{F}} = (\widetilde{\mathbf{F}}_d, \ldots, \widetilde{\mathbf{F}}_1)$ be the system of forms such that for each $1 \leq \ell \leq d$,
$\widetilde{\mathbf{F}}_{\ell} = ( \widetilde{F}_{\ell,1}, \ldots,  \widetilde{F}_{\ell,r_{\ell}})$ and $\widetilde{F}_{\ell,r}$
is the homogeneous degree $\ell$ portion of $\widetilde{f}_{\ell, r}$ $(1 \leq r \leq r_{\ell})$.
For each $1 \leq \ell \leq d$, we know that $\mathcal{B}_{\ell}(\widetilde{\mathbf{F}}_{\ell})$ is sufficiently large with respect to $d$ and $r_d, \ldots, r_1$.
Thus we apply Proposition \ref{prop decomp} to the system
$$
(\widetilde{\mathbf{F}}_d |_{\mathbf{w} = \mathbf{0}}, \ldots, \widetilde{ \mathbf{F} }_1 |_{\mathbf{w} = \mathbf{0}}),
$$
and denote the partition of variables of $\mathbf{x} \backslash \mathbf{w}$ we obtain by  $(\mathbf{y}, \mathbf{z})$ so that
$\mathbf{x} = (\mathbf{w}, \mathbf{y}, \mathbf{z})$.
Let
$$
\widetilde{Q}_{\ell,r} (\mathbf{y}, \mathbf{z}) =  \widetilde{F}_{\ell,r} (\mathbf{0}, \mathbf{y}, \mathbf{z}) - \widetilde{F}_{\ell,r} (\mathbf{0}, \mathbf{0}, \mathbf{z})
\ \ \ ( 2 \leq \ell \leq d, 1 \leq r \leq r_{\ell} ).
$$
Then the partition of variables $\mathbf{x} = (\mathbf{w}, \mathbf{y}, \mathbf{z})$ satisfies
\begin{equation}
\label{bound on M}
|\mathbf{y}| = M \leq \sum_{\ell=1}^d r_{\ell} C^{\bullet}_{\ell,1},
\end{equation}
and also
\begin{equation}
\label{ineq B1 1 -1}
\mathcal{B}_{\ell} \left(  \{ \widetilde{Q}_{\ell,r} ( \mathbf{y}, \mathbf{z}) : 1 \leq r \leq r_{\ell} \}  \right)
\geq C^{\bullet}_{\ell,1} - r_{\ell} \sum_{j=1}^{\ell - 1} C^{\bullet}_{j,1} r_{j}
\ \ (2 \leq \ell \leq d),
\end{equation}
\begin{equation}
\label{ineq B1 1}
\mathcal{B}_{1} \left( \{ \widetilde{F}_{1,r} ( \mathbf{0},  \mathbf{y}, \mathbf{0}) : 1 \leq r \leq r_{1}  \}  \right) \geq C^{\bullet}_{1,1},
\end{equation}
and
\begin{equation}
\label{ineq B1 1 +1}
\mathcal{B}_{\ell} \left(  \{ \widetilde{F}_{\ell,r} (\mathbf{0}, \mathbf{0}, \mathbf{z}) : 1 \leq r \leq r_{\ell} \}  \right)
\geq C^{\bullet}_{\ell,2} - r_{\ell} \sum_{j=1}^{\ell - 1} C^{\bullet}_{j,1} r_{j}
\ \ (1 \leq \ell \leq d),
\end{equation}
where $C^{\bullet}_{\ell, 1}$ and $C^{\bullet}_{\ell, 2}$ $(1 \leq \ell \leq d)$ are positive integer constants
depending only on $d$ and $r_d, \ldots, r_1$ to be chosen later.
In particular, we will make sure that the right hand side of (\ref{ineq B1 1 +1}) for $2 \leq \ell \leq d$ is
sufficiently large with respect to $d$ and $r_d, \ldots, r_1$.
For notational convenience, we label $\mathbf{y} = (y_1, \ldots, y_M)$ and $\mathbf{z} = (z_1, \ldots, z_{n - M - K})$.

We then apply Proposition \ref{prop decomp} (with $C_{1,1} = C_{1,2} = C_{2,2} =  \ldots = C_{d,2}  = 1$) to the system of forms $
( \widetilde{\mathbf{F}}_{d}(\mathbf{0}, \mathbf{0}, \mathbf{z}), \ldots,$ $\widetilde{\mathbf{F}}_{2}(\mathbf{0}, \mathbf{0}, \mathbf{z})  )$, where
$\widetilde{\mathbf{F}}_{\ell}(\mathbf{0}, \mathbf{0}, \mathbf{z}) = ( \widetilde{F}_{\ell,1} (\mathbf{0}, \mathbf{0}, \mathbf{z}), \ldots, \widetilde{F}_{\ell,r_{\ell}} (\mathbf{0}, \mathbf{0}, \mathbf{z}) )$ for each
$2 \leq \ell \leq d$. Let the partition of variables we obtain to be ${\mathbf{z}} = ( \mathbf{a}, \mathbf{b})$,
which satisfies
\begin{equation}
\label{bound on M'}
|\mathbf{a}| = M'  \leq \sum_{\ell = 2}^d r_{\ell} C^{\star}_{\ell, 1},
\end{equation}
and
\begin{eqnarray}
\label{ineq of Ptilde}
\mathcal{B}_{\ell} \left( \{  \widetilde{P}_{\ell,r} (\mathbf{a}, \mathbf{b} ) : 1 \leq r \leq r_{\ell}   \} \right)
\geq C^{\star}_{\ell, 1} - r_{\ell} \sum_{j=2}^{\ell - 1} C^{\star}_{j,1} r_{j}\ \ (2 \leq \ell < d, 1 \leq r \leq r_{\ell}),
\end{eqnarray}
where
$$
\widetilde{P}_{\ell, r} (\mathbf{a}, \mathbf{b}) = \widetilde{F}_{\ell,r} (\mathbf{0}, \mathbf{0}, \mathbf{z})
- \widetilde{F}_{\ell,r}(\mathbf{0}, \mathbf{0}, (\mathbf{0}, \mathbf{b})) \ \ (2 \leq \ell < d, 1 \leq r \leq r_{\ell}).
$$
Note we are only mentioning the statement (\ref{ineq of Ptilde}) for $2 \leq \ell < d$, because we will not be needing it for the case $\ell = d$.
Recall from (\ref{set of eqn 2}) we have for $2 \leq \ell \leq d, 1 \leq r \leq r_{\ell}$,
$$
f_{\ell,r} (\mathbf{x}) = c_{\ell, r} \mathbf{w}^{\mathbf{j}_{\ell, r}} + \chi_{\ell, r}(\mathbf{x}) + \widetilde{f}_{\ell,r}(\mathbf{x}),
$$
where
$$
h_{\ell}(\chi_{\ell, r}) \leq C''_0
$$
for some constant $C''_0$ dependent only on $d$ and $r_d, \ldots, r_1$.
Let $\chi_{\ell, r}^{(\ell)}(\mathbf{x})$ denote the degree $\ell$ portion of $\chi_{\ell, r}(\mathbf{x})$.
Then it is easy to deduce from the definition of the $h$-invariant that the quantities
$$
h_{\ell}( \chi_{\ell, r}^{(\ell)}(\mathbf{0}, \mathbf{y}, \mathbf{z}) - \chi_{\ell, r}^{(\ell)}(\mathbf{0}, \mathbf{0}, \mathbf{z})  ), \
h_{\ell}( \chi_{\ell, r}^{(\ell)}(\mathbf{0}, \mathbf{0}, \mathbf{z}) - \chi_{\ell, r}^{(\ell)}(\mathbf{0}, \mathbf{0}, (\mathbf{0}, \mathbf{b})) ), \
h_{\ell}( \chi_{\ell, r}^{(\ell)}(\mathbf{0}, \mathbf{0}, \mathbf{z}))
$$
are all bounded by $2 C''_0$.
We then let
\begin{equation}
\label{def Q}
Q_{\ell, r} (\mathbf{y}, \mathbf{z}) = \widetilde{Q}_{\ell, r} (\mathbf{y}, \mathbf{z}) + \chi_{\ell, r}^{(\ell)}(\mathbf{0}, \mathbf{y}, \mathbf{z}) - \chi_{\ell, r}^{(\ell)}(\mathbf{0}, \mathbf{0}, \mathbf{z})
\end{equation}
and
\begin{equation}
\label{def P}
P_{\ell, r} (\mathbf{a}, \mathbf{b}) = \widetilde{P}_{\ell, r} (\mathbf{a}, \mathbf{b}) + \chi_{\ell, r}^{(\ell)}(\mathbf{0}, \mathbf{0}, \mathbf{z}) - \chi_{\ell, r}^{(\ell)}(\mathbf{0}, \mathbf{0}, (\mathbf{0}, \mathbf{b})).
\end{equation}
We remark that from the definition it follows that $Q_{\ell, r} (\mathbf{y}, \mathbf{z})$
is precisely the degree $\ell$ portion of the polynomial
$f_{\ell, r}(\mathbf{0}, \mathbf{y},\mathbf{z}) - f_{\ell, r}(\mathbf{0}, \mathbf{0},\mathbf{z})$.
Clearly every monomial of $Q_{\ell, r} (\mathbf{y}, \mathbf{z})$ with non-zero coefficient contains at least one of the $\mathbf{y}$ variables, and hence
$h_{\ell}(Q_{\ell, r} (\mathbf{y}, \mathbf{z})) \leq |\mathbf{y}|$.
Similarly, $P_{\ell, r} (\mathbf{a}, \mathbf{b})$
is precisely the degree $\ell$ portion of the polynomial
$f_{\ell, r}(\mathbf{0}, \mathbf{0},\mathbf{z}) - f_{\ell, r}(\mathbf{0}, \mathbf{0},(\mathbf{0}, \mathbf{b}) )$.
Clearly every monomial of $P_{\ell, r} (\mathbf{a}, \mathbf{b})$ with non-zero coefficient contains at least one of the $\mathbf{a}$ variables, and hence
$h_{\ell}(P_{\ell, r} (\mathbf{a}, \mathbf{b})) \leq |\mathbf{a}|$.

We obtain the following three inequalities from (\ref{ineq B1 1 -1}), (\ref{ineq of Ptilde}), and
(\ref{ineq B1 1 +1}), respectively, by applying (\ref{ineq of h and B}) and the definition of the $h$-invariant with
the comment before (\ref{def Q}),
\begin{eqnarray}
\label{hinv bound1}
\\
\notag
h_{\ell} \left( \{ Q_{\ell,r} (\mathbf{y}, \mathbf{z}) : 1 \leq r \leq r_{\ell}  \} \right)
\geq 2^{1 - \ell} \left( C^{\bullet}_{\ell, 1} - r_{\ell} \sum_{j=1}^{\ell - 1} C^{\bullet}_{j,1} r_{j} \right) - 2 r_{\ell} C''_0
\ \
(2 \leq \ell \leq d),
\end{eqnarray}
\begin{eqnarray}
\label{hinv bound2}
\\
\notag
h_{\ell} \left( \{  P_{\ell,r} (\mathbf{a}, \mathbf{b} ) : 1 \leq r \leq r_{\ell}  \} \right)
\geq 2^{1 - \ell} \left( C^{\star}_{\ell, 1} - r_{\ell} \sum_{j=2}^{\ell - 1} C^{\star}_{j,1} r_{j} \right) - 2 r_{\ell} C''_0
\ \ (2 \leq \ell < d),
\end{eqnarray}
and
\begin{eqnarray}
\label{hinv bound3}
\\
\notag
h_{d} \left( \{ \widetilde{F}_{d,r} (\mathbf{0}, \mathbf{0}, \mathbf{z}) + \chi_{d, r}^{(d)}(\mathbf{0}, \mathbf{0}, \mathbf{z}) :  1 \leq r \leq r_{d}  \} \right)
\geq  2^{1 - d} \left(C^{\bullet}_{d,2} - r_{d} \sum_{j=1}^{d - 1} C^{\bullet}_{j,1} r_{j} \right) - 2 r_{d} C''_0.
\end{eqnarray}
It is clear from the definition that the degree $d$ portion of $f_{d,r}(\mathbf{0}, \mathbf{0}, \mathbf{z})$ is precisely
$\widetilde{F}_{d,r} (\mathbf{0}, \mathbf{0}, \mathbf{z}) + \chi_{d, r}^{(d)}(\mathbf{0}, \mathbf{0}, \mathbf{z})$ for each $1 \leq r \leq r_{d}$.
Thus we have
$$
F_{d,r} (\mathbf{0}, \mathbf{0}, \mathbf{z}) = \widetilde{F}_{d,r} (\mathbf{0}, \mathbf{0}, \mathbf{z}) + \chi_{d, r}^{(d)}(\mathbf{0}, \mathbf{0}, \mathbf{z})
\ \ (1 \leq r \leq r_d),
$$
and we also let
\begin{equation}
\label{def Fd z}
\mathbf{F}_d(\mathbf{0}, \mathbf{0}, \mathbf{z}) = ( F_{d,1} (\mathbf{0}, \mathbf{0}, \mathbf{z}), \ldots, F_{d,r_d} (\mathbf{0}, \mathbf{0}, \mathbf{z}) ).
\end{equation}

With the notations we have defined so far, we have
$$
F_{d, r}(\mathbf{0}, \mathbf{y}, \mathbf{z}) =  Q_{d,r}(\mathbf{y}, \mathbf{z}) + F_{d, r}(\mathbf{0}, \mathbf{0}, \mathbf{z})
\ \ (1 \leq r \leq r_d),
$$
and for each $2 \leq \ell \leq d$,
$$
F_{\ell, r}(\mathbf{0}, \mathbf{y}, \mathbf{z}) =  Q_{\ell,r}(\mathbf{y}, \mathbf{z}) +
P_{\ell,r}(\mathbf{a}, \mathbf{b}) + F_{\ell, r}(\mathbf{0}, \mathbf{0}, (\mathbf{0}, \mathbf{b}) ) \ \ (1 \leq r \leq r_{\ell}),
$$
where
$$
F_{\ell, r}(\mathbf{0}, \mathbf{0}, \mathbf{z})  = P_{\ell,r}(\mathbf{a}, \mathbf{b}) + F_{\ell, r}(\mathbf{0}, \mathbf{0}, (\mathbf{0}, \mathbf{b}) ) \ \ (1 \leq r \leq r_{\ell}).
$$

Let $2 \leq \ell \leq d$. For each $1 \leq r \leq r_{\ell}$, the partition of variables $\mathbf{x} = ( \mathbf{w}, \mathbf{y}, \mathbf{z} )$ gives the decomposition of the following shape
\begin{eqnarray}
\label{first decomp of f}
&&f_{\ell,r}(\mathbf{w}, \mathbf{y}, \mathbf{z})
\\
&=&
f_{\ell, r}(\mathbf{w}, \mathbf{0}, \mathbf{0})
+ \sum_{j=1}^{ \ell -1}  \ \sum_{1 \leq i_1 \leq \ldots \leq i_j \leq K}
\left( \sum_{k=1}^{\ell -j} \Phi^{(k)}_{\ell, r: i_1, \ldots, i_{j}} ( \mathbf{y}, \mathbf{z} )  \right) w_{i_1} \ldots w_{i_j}
\notag
\\
&+&
\sum_{j=1}^{\ell -1} \ \sum_{1 \leq t_1 \leq \ldots \leq t_{j} \leq M}
\left( \sum_{k=0}^{\ell -j} \Psi^{(k)}_{\ell, r: t_1, \ldots, t_{j}} ( \mathbf{z} )  \right) y_{t_1} \ldots y_{t_{j}} + F_{\ell, r}(\mathbf{0}, \mathbf{y}, \mathbf{0})
\notag
\\
&+& F_{\ell, r}(\mathbf{0},\mathbf{0}, \mathbf{z}) + \sum_{k=1}^{\ell - 1} G^{(k)}_{\ell, r} (\mathbf{z}),
\notag
\end{eqnarray}
which we describe below.
We note that $\Phi^{(k)}_{\ell, r:i_1, \ldots, i_{j}} ( \mathbf{y}, \mathbf{z} )$ and  $\Psi^{(k)}_{\ell,r:t_1, \ldots, t_{j}} ( \mathbf{z} ) $
are forms of degree $k$.
The above decomposition establishes the following. The term
$$
f_{\ell,r}(\mathbf{w}, \mathbf{0}, \mathbf{0})
+ \sum_{j=1}^{ \ell -1}  \ \sum_{1 \leq i_1 \leq \ldots \leq i_j \leq K}
\left( \sum_{k=1}^{\ell -j} \Phi^{(k)}_{\ell, r: i_1, \ldots, i_{j}} ( \mathbf{y}, \mathbf{z} )  \right) w_{i_1} \ldots w_{i_j}
$$
consists of all the monomials of $f_{\ell, r}(\mathbf{x})$ which involve any variables of $\mathbf{w}$, and also the constant term.
The term
\begin{equation}
\label{eqn hyaaa1}
\sum_{j=1}^{\ell -1} \ \sum_{1 \leq t_1 \leq \ldots \leq t_{j} \leq M}
\left( \sum_{k=0}^{\ell -j} \Psi^{(k)}_{\ell, r: t_1, \ldots, t_{j}} ( \mathbf{z} )  \right) y_{t_1} \ldots y_{t_{j}} + F_{\ell, r}(\mathbf{0}, \mathbf{y}, \mathbf{0})
\end{equation}
consists of all the monomials of $f_{\ell,r}(\mathbf{x})$ which involve any variables of $\mathbf{y}$
and do not involve any of the $\mathbf{w}$ variables. In other words, it is precisely $f_{\ell,r}(\mathbf{0}, \mathbf{y}, \mathbf{z}) - f_{\ell,r}(\mathbf{0}, \mathbf{0}, \mathbf{z})$.
Finally, we have the terms which only involve the $\mathbf{z}$ variables
$$
F_{\ell,r}(\mathbf{0},\mathbf{0}, \mathbf{z}) + \sum_{k=1}^{\ell - 1} G^{(k)}_{\ell,r} (\mathbf{z}),
$$
where
$G^{(k)}_{\ell,r} (\mathbf{z})$ is the homogeneous degree $k$ portion of  $f_{\ell,r}(\mathbf{0}, \mathbf{0}, \mathbf{z})$.

We denote by
\begin{eqnarray}
\Phi^{} &=& \{ \Phi^{(k)}_{\ell, r: i_1, \ldots, i_{j}} : 2 \leq \ell \leq d,  1 \leq r \leq r_{\ell}, 1 \leq j \leq \ell -1,
\notag
\\
&& \phantom{01234567890} 1 \leq i_1 \leq \ldots \leq i_j \leq K, 1 \leq k \leq \ell - j \}.
\notag
\end{eqnarray}
Note every form of $\Phi$ has degree strictly less than $d$, and involves only the $\mathbf{y}$ variables and the $\mathbf{z}$ variables.
We shall use the notation $|\Phi|$ to denote the number of forms in $\Phi$, and other instances of notation of this type should be interpreted in a similar manner. Clearly we have
$$
|\Phi| \leq  \sum_{\ell = 2}^{d} r_{\ell} \ell^2 K^{\ell} \leq \sum_{\ell = 2}^{d} R \ell^{2} (dR)^{\ell} \leq d^{d+3}R^{d+1}.
$$
Recall the function $\rho_{d, \ell}$ defined in (\ref{def rho}).
We apply Proposition \ref{prop reg par} to the system $\Phi$ with respect
to the partition of variables $(\mathbf{y}, \mathbf{z})$ and the functions $\boldsymbol{\mathcal{F}}^{} = \{\mathcal{F}^{}_2, \ldots, \mathcal{F}^{}_{d -1} \}$, where
\begin{eqnarray}
\mathcal{F}^{}_i(t) &=&  \rho_{d,  d }( 2R + 2 t ) + 2 t + 4 r_1
\notag
\\
&+&
2 R \left( d R ( R^{2} + 1 )^{d-2} 2^d \Big{(} \rho_{d, d} (2 R + 2 t) + 2 t + 4 r_1 +  2 R C''_0 \Big{)} +
d  R^3 ( R^{2} + 1 )^{d - 2}  (2 t + 1) \right),
\notag
\end{eqnarray}
for each $2 \leq i \leq d-1$,
and obtain
$\mathcal{R}(\Phi) = ( \mathcal{R}^{(d-1)}(\Phi), \ldots,\mathcal{R}^{(1)}(\Phi) )$. For each $1 \leq s \leq d-1$,
$$
\mathcal{R}^{(s)}(\Phi) = \{ A^{(s)}_i : 1 \leq i \leq | \mathcal{R}^{(s)}(\Phi)  |  \}
$$
is precisely all the degree $s$ forms of $\mathcal{R}(\Phi)$.
For each form $A^{({s})}_i \in \mathcal{R}^{(s)}(\Phi) \ (1 \leq s \leq d-1, 1 \leq i \leq | \mathcal{R}^{(s)}(\Phi) |)$, we write
\begin{equation}
\label{defn of a's}
A^{({s})}_i(\mathbf{y}, \mathbf{z}) = \sum_{k=0}^{s} \sum_{1 \leq i_1 \leq \ldots \leq i_k \leq M} \widetilde{\Psi}^{(s - k)}_{{s}, i: i_1, \ldots, i_k}(\mathbf{z})  y_{i_1} \ldots y_{i_k},
\end{equation}
where each
$\widetilde{\Psi}^{( s - k)}_{{s}, i: i_1, \ldots, i_k}(\mathbf{z})$ is a form of degree $s - k$.
Thus each form $A^{({s})}_i$ introduces at most $({s}+1) M^{s} \leq d M^{d}$ forms in $\mathbf{z}$.
Also for each $1 \leq s \leq d-1$, we denote $\overline{\mathcal{R}}^{(s)}(\Phi^{})$
to be the system obtained by removing from $\mathcal{R}^{(s)}(\Phi)$ all forms which depend only on the $\mathbf{z}$ variables.
Let $\overline{\mathcal{R}}(\Phi)= (\overline{\mathcal{R}}^{(d-1)}(\Phi),  \ldots, \overline{\mathcal{R}}^{(1)}(\Phi) )$, $R^{}_2 = \sum_{s = 1}^{d-1} \ | \overline{\mathcal{R}}^{(s)}(\Phi)|$, and $D^{}_2 = \sum_{s = 1}^{d-1} s \ | \overline{\mathcal{R}}^{(s)}(\Phi) |$. By relabeling if necessary, for each $1 \leq s \leq d-1$ we denote the
elements of $\overline{\mathcal{R}}^{(s)}(\Phi)$ by
\begin{equation}
\label{label of A bar}
\overline{\mathcal{R}}^{(s)}(\Phi) = \{ A^{(s)}_i : 1 \leq i \leq | \overline{\mathcal{R}}^{(s)}(\Phi) | \}.
\end{equation}

Let
\begin{eqnarray}
\Psi^{}
&=&
\{  F_{\ell,r} (\mathbf{0}, \mathbf{0}, (\mathbf{0}, \mathbf{b})) : 2 \leq \ell < d,  1 \leq r \leq r_{\ell}  \}
\notag
\\
&\cup& \{ G^{(k)}_{\ell,r} (\mathbf{z}) : 2 \leq \ell \leq d,  1 \leq r \leq r_{\ell}, 1 \leq k \leq \ell - 1 \}
\notag
\\
\notag
&\cup&
\{  \Psi^{(k)}_{\ell, r: t_1, \ldots, t_{j}} (\mathbf{z}): 2 \leq \ell \leq d,  1 \leq r \leq r_{\ell}, 1 \leq j \leq \ell -1,
\notag
\\
&&\phantom{1234567890123456789012345678901234} 1 \leq t_1 \leq \ldots \leq t_{j} \leq M ,   1 \leq k \leq \ell -j  \}
\notag
\\
&\cup&  \  \{ \widetilde{\Psi}^{(s - k)}_{{s}, i: i_1, \ldots, i_k}(\mathbf{z}) :1 \leq {s} \leq d -1, 1 \leq i \leq | \mathcal{R}^{(s)}(\Phi) |, 0 \leq k < {s}, 1 \leq i_1 \leq \ldots \leq i_k \leq M  \}.
\notag
\end{eqnarray}
In other words, $\Psi^{}$ is the collection of all $G^{(k)}_{\ell,r} (\mathbf{z})$, $ \Psi^{(k)}_{\ell, r : t_1, \ldots, t_{j}} (\mathbf{z})$, $\widetilde{\Psi}^{( s - k)}_{{s}, i: i_1, \ldots, i_k}(\mathbf{z})$ except the constants, and all $F_{\ell,r} (\mathbf{0}, \mathbf{0}, (\mathbf{0}, \mathbf{b}))$ but not $F_{d,r} (\mathbf{0}, \mathbf{0}, (\mathbf{0}, \mathbf{b}) )$. In particular, every form of $\Psi$ has degree strictly less than $d$.
We can see that
$$
|\Psi^{}| \leq   R + d R +  R d^2 M^{d} + |\mathcal{R}(\Phi) | d M^{d}.
$$
Furthermore, every form of $\Psi$ is only in terms of the $\mathbf{z}$ variables.

We let $\mathcal{R}(\Psi^{}) = (\mathcal{R}^{(d-1)}(\Psi), \ldots, \mathcal{R}^{(1)}(\Psi) )$ be a regularization of
$\Psi^{}$ with respect
to the functions
$\boldsymbol{\mathcal{F}}' = \{\mathcal{F}'_2, \ldots, \mathcal{F}'_{d - 1}\},$
where
\begin{eqnarray}
\mathcal{F}'_i(t) = \rho_{d,  d }( 2R + 2 t ) + 2 t + 4 r_1
+ 2 R \Big{(} d R ( R^{2} + 1 )^{d-2} 2^d \Big{(} \rho_{d, d} (2 R + 2 t) + 2 t + 4 r_1 +  2 R C''_0 \Big{)} \Big{)}
\notag
\end{eqnarray}
for each $2 \leq i \leq d -1$.
For each $1 \leq s \leq d-1$,
$$
\mathcal{R}^{(s)}(\Psi) = \{ V^{(s)}_i : 1 \leq i \leq | \mathcal{R}^{(s)}(\Psi) |   \}
$$
is precisely all the degree $s$ forms of $\mathcal{R}(\Psi)$. Let
$R_1 = \sum_{s = 1}^{d-1} \ | \mathcal{R}^{(s)}(\Psi) |$ and $D_1 = \sum_{s = 1}^{d -1} s \ | \mathcal{R}^{(s)}(\Psi) |$.

Let $\Phi^{(j)}$ denote the degree $j$ forms of $\Phi$.
It follows from Proposition \ref{prop reg par} that each $|\mathcal{R}^{(i)}(\Phi)|$ is bounded by some constant dependent only on $\boldsymbol{\mathcal{F}}$ and $|\Phi^{(d-1)}|$, $\ldots$, $|\Phi^{(1)}|$. Thus we see that $| \mathcal{R}(\Phi) |$ and $R_2$ are bounded by a constant dependent only on $d$ and $r_d, ..., r_1$. It also follows from Proposition \ref{prop reg par} that each $| \mathcal{R}^{(i)}(\Psi) |$ is bounded by some constant dependent only
on $\boldsymbol{\mathcal{F}}'$, $d$, $R$, $M$, and $| \mathcal{R}(\Phi) |$.
Thus we obtain that $R^{}_1$ is bounded by a constant dependent only on $M$, $d$ and  $r_d, \ldots, r_1$.

We first set $C^{\bullet}_{1, 1} = 2 R_2 + 1$.
We now set the values of $C^{\bullet}_{\ell, 1}$ $(2 \leq \ell \leq d)$ such that they satisfy
\begin{equation}
\label{hinv bound1'}
2^{1 - \ell} \left( C^{\bullet}_{\ell, 1} - r_{\ell }\sum_{j=1}^{\ell - 1} C^{\bullet}_{j,1} r_{j} \right) - 2 r_{\ell} C''_0
\geq \rho_{d, \ell} (2 R + 2 R_2) + 2 R_2 + 4 r_1.
\end{equation}
Since (\ref{hinv bound1'}) is equivalent to
$$
C^{\bullet}_{\ell, 1}
\geq  2^{\ell-1} \Big{(} \rho_{d, \ell} (2 R + 2 R_2) + 2 R_2 + 4 r_1 +  2 r_{\ell} C''_0 \Big{)} +  r_{\ell }\sum_{j=1}^{\ell - 1} C^{\bullet}_{j,1} r_{j},
$$
we can also make sure $C^{\bullet}_{\ell, 1}$ satisfies the additional constraint
$$
C^{\bullet}_{\ell, 1}
\leq  2^d \Big{(} \rho_{d, d} (2 R + 2 R_2) + 2 R_2 + 4 r_1 +  2 R C''_0 \Big{)} +  R^2 \sum_{j=1}^{\ell - 1} C^{\bullet}_{j,1}.
$$
It is then not difficult to show by induction that
\begin{eqnarray}
C^{\bullet}_{\ell, 1}
&\leq&  ( R^{2} + 1 )^{\ell - 2} 2^d \Big{(} \rho_{d, d} (2 R + 2 R_2) + 2 R_2 + 4 r_1 +  2 R C''_0 \Big{)} +
 R^2 ( R^{2} + 1 )^{\ell - 2} C^{\bullet}_{1, 1}.
\notag
\end{eqnarray}
In particular, $C^{\bullet}_{\ell, 1}$ is bounded by a constant dependent only on $d$ and $r_d, \ldots, r_1$.
Therefore, it follows from (\ref{bound on M}) that
\begin{eqnarray}
\label{bound on M1}
&&
\\
M &\leq& R \sum_{\ell = 1}^d C^{\bullet}_{\ell, 1}
\notag
\\
&\leq& d R ( R^{2} + 1 )^{d-2} 2^d \Big{(} \rho_{d, d} (2 R + 2 R_2) + 2 R_2 + 4 r_1 +  2 R C''_0 \Big{)} +
d R^3 ( R^{2} + 1 )^{d - 2} (2 R_2 + 1).
\notag
\end{eqnarray}
Thus it follows that $R_1$ is bounded by a constant dependent only on $d$ and $r_d, \ldots, r_1$.

We then set the values for $C^{\star}_{\ell, 1}$ $(2 \leq \ell \leq d)$ to satisfy
\begin{equation}
\label{hinv bound2'}
2^{1 - \ell} \left( C^{\star}_{\ell, 1} - r_{\ell} \sum_{j=2}^{\ell - 1} C^{\star}_{j,1} r_{j} \right) - 2 r_{\ell} C''_0 \geq
 \rho_{d, \ell} (2 R + 2 R_1) + 2 R_1 + 4 r_1.
\end{equation}
By a similar argument as for the $C^{\bullet}_{\ell, 1}$ above, we can also make sure that $C^{\star}_{\ell, 1}$ satisfies
\begin{eqnarray}
C^{\star}_{\ell, 1}
&\leq&  ( R^{2} + 1 )^{\ell - 2} 2^d \Big{(} \rho_{d, d} (2 R + 2 R_1) + 2 R_1 + 4 r_1 +  2 R C''_0 \Big{)},
\notag
\end{eqnarray}
and it follows from (\ref{bound on M'}) that
\begin{eqnarray}
\label{bound on M'1}
M' \leq d R ( R^{2} + 1 )^{d-2} 2^d \Big{(} \rho_{d, d} (2 R + 2 R_1) + 2 R_1 + 4 r_1 +  2 R C''_0 \Big{)}.
\end{eqnarray}
In particular, each $C^{\star}_{\ell, 1}$ and $M'$ are bounded by constants dependent only on $d$ and $r_d, \ldots, r_1$.
Let us set $C^{\bullet}_{1, 2} = 2 R_1 + 1$.
We then make sure that for each $2 \leq \ell \leq d$, $C^{\bullet}_{\ell, 2}$ is sufficiently large with respect to
$C^{\star}_{2, 1}, \ldots, C^{\star}_{d, 1}$, $C^{\bullet}_{1, 1}, \ldots, C^{\bullet}_{d-1, 1}$,
$r_d, \ldots, r_1$, and $d$,
and also that $C^{\bullet}_{d, 2}$ satisfies
\begin{equation}
\label{hinv bound3'}
2^{1 - d} \left(C^{\bullet}_{d,2} - r_{d} \sum_{j=1}^{d - 1} C^{\bullet}_{j,1} r_{j} \right) - 2 r_{d} C''_0
\geq  \rho_{d, d} (2 R + 2 R_1) + 2 R_1 + 4 r_1.
\end{equation}
We note that the three inequalities (\ref{hinv bound1'}), (\ref{hinv bound2'}) and (\ref{hinv bound3'}) provide lower bounds for the $h$-invariants in (\ref{hinv bound1}), (\ref{hinv bound2}), and (\ref{hinv bound3}), respectively.

We now decompose the linear polynomials.
From Proposition \ref{prop reg par}, we know that every linear form of $\mathcal{R}^{(1)}(\Phi)$ is either only in the $\mathbf{y}$ variables, or only in the $\mathbf{z}$ variables. First we consider the linear forms of
$\overline{\mathcal{R}}^{(1)}(\Phi) = \{ A^{(1)}_i(\mathbf{y}) : 1 \leq i \leq | \overline{\mathcal{R}}^{(1)}(\Phi) | \}$, which we know to be linearly independent over $\mathbb{Q}$ and involve only the $\mathbf{y}$ variables.
By considering
their linear combinations, we may assume without loss of generality that  these linear forms are of the shape
$$
A^{(1)}_i(\mathbf{y}) = y_i + A'_i(y_{| \overline{\mathcal{R}}^{(1)}(\Phi) | + 1}, \ldots, y_M) \  \ (1 \leq i \leq | \overline{\mathcal{R}}^{(1)}(\Phi) |),
$$
where  $A'_i(y_{| \overline{\mathcal{R}}^{(1)}(\Phi) | + 1}, \ldots, y_M)$ is a linear form in the variables $y_{| \overline{\mathcal{R}}^{(1)}(\Phi) | + 1}, \ldots, y_M$
with coefficients in $\mathbb{Q}$. By (\ref{ineq B1 1}) and Lemma \ref{Lemma 3 in CM}, we have
$$
\mathcal{B}_{1}( \
\Big{ \{ } \widetilde{F}_{1,r}(\mathbf{0}, \mathbf{y}, \mathbf{0}) \Big{|}_{ y_i = 0 \ (1 \leq i \leq | \overline{\mathcal{R}}^{(1)}(\Phi) |)  }  : 1 \leq r \leq r_1 \Big{ \}}
\ )
\geq C^{\bullet}_{1,1} - | \overline{\mathcal{R}}^{(1)}(\Phi) | \geq R_2 + 1 > 0.
$$
Therefore, we can find $r_1$ variables from $y_{| \overline{\mathcal{R}}^{(1)}(\Phi) | + 1}, \ldots, y_M$
such that the $r_1 \times r_1$ matrix, where the $r$-th row consists of the coefficients of $\widetilde{F}_{1,r}(\mathbf{0}, \mathbf{y}, \mathbf{0})$
of these $r_1$ variables, is invertible. Let us denote these variables by $\widetilde{y}_1, \ldots, \widetilde{y}_{r_1}$, and let $\widetilde{\mathbf{y}} = (\widetilde{y}_1, \ldots, \widetilde{y}_{r_1})$.
We can then write
$$
\widetilde{F}_{1,r}(\mathbf{0}, \mathbf{y}, \mathbf{0}) =  {g}_{r,1} \widetilde{y}_1 + \ldots  +  {g}_{r,r_1} \widetilde{y}_{r_1}
+ \widetilde{F}_{1,r}(\mathbf{0}, \mathbf{y}, \mathbf{0})|_{\mathbf{\widetilde{y}} = \mathbf{0}},
$$
where ${g}_{r,1}, \ldots, {g}_{r,r_1} \in \mathbb{Z}.$
Let $\mathcal{R}_+^{(1)}(\Phi)$ be a maximal linearly independent (over $\mathbb{Q}$) subset of
$$
{\mathcal{R}}^{(1)}(\Phi) \cup \{\widetilde{F}_{1,1}(\mathbf{0}, \mathbf{y}, \mathbf{0})|_{\mathbf{\widetilde{y}} = \mathbf{0}}, \ldots, \widetilde{F}_{1,r_1}(\mathbf{0}, \mathbf{y}, \mathbf{0})|_{\mathbf{\widetilde{y}} = \mathbf{0}}   \}.
$$
The important thing to note is that by our construction, we have that
the set of linear forms
$$
\{ g_{r,1} \widetilde{y}_1 + \ldots  + g_{r,r_1} \widetilde{y}_{r_1} : 1 \leq r \leq r_1 \} \cup \overline{\mathcal{R}}_+^{(1)}(\Phi)
$$
is linearly independent over $\mathbb{Q}$. Here $\overline{\mathcal{R}}_+^{(1)}(\Phi)$ is the set of forms obtained by
removing from $\mathcal{R}_+^{(1)}(\Phi) $ all forms that depend only on the $\mathbf{z}$ variables.
Note we also have $| \overline{\mathcal{R}}_+^{(1)}(\Phi) |
\leq | \overline{\mathcal{R}}^{(1)}(\Phi) | + r_1$.

We also decompose the linear forms $\widetilde{F}_{1,r}(\mathbf{0}, \mathbf{0}, \mathbf{z})$ in a similar manner.
First we consider the linear forms of ${\mathcal{R}}^{(1)}(\Psi) = \{ V^{(1)}_i(\mathbf{z}) : 1 \leq i \leq | {\mathcal{R}}^{(1)}(\Psi) | \}$, which we know to be linearly independent over $\mathbb{Q}$ and involve only the $\mathbf{z}$ variables. By considering their linear combinations, we may assume without loss of generality that these linear forms are of the shape
$$
V^{(1)}_i (\mathbf{z}) = z_i + V'_i(z_{| {\mathcal{R}}^{(1)}(\Psi) | + 1}, \ldots, z_{n-M-K})  \  \ (1 \leq i \leq | {\mathcal{R}}^{(1)}(\Psi) |),
$$
where  $V'_i(z_{| {\mathcal{R}}^{(1)}(\Psi) | + 1}, \ldots, z_{n-M-K})$ is a linear form in the variables $z_{| {\mathcal{R}}^{(1)}(\Psi) | + 1}, \ldots, z_{n-M-K}$
with coefficients in $\mathbb{Q}$.
By (\ref{ineq B1 1 +1}) and Lemma \ref{Lemma 3 in CM}, we have
$$
\mathcal{B}_{1}( \
\Big{ \{ } \widetilde{F}_{1,r}(\mathbf{0}, \mathbf{0}, \mathbf{z}) \Big{|}_{ z_i = 0 \ (1 \leq i \leq |{\mathcal{R}}^{(1)}(\Psi)| )  } : 1 \leq r \leq r_1  \Big{ \}} \ )
\geq C^{\bullet}_{1,2} - |{\mathcal{R}}^{(1)}(\Psi)| \geq R_1 + 1 > 0.
$$
Therefore, we can find $r_1$ variables from $z_{| {\mathcal{R}}^{(1)}(\Psi) | + 1}, \ldots, z_{n-M-K}$
such that the $r_1 \times r_1$ matrix, where the $r$-th row consists of the coefficients of $\widetilde{F}_{1,r}(\mathbf{0}, \mathbf{0}, \mathbf{z})$
of these $r_1$ variables, is invertible. Let us denote these variables by $\widetilde{z}_1, \ldots, \widetilde{z}_{r_1}$, and let
$\mathbf{\widetilde{z}} = (\widetilde{z}_1, \ldots, \widetilde{z}_{r_1})$.
We can then write
$$
\widetilde{F}_{1,r}(\mathbf{0}, \mathbf{0}, \mathbf{z}) = {g}'_{r,1} \widetilde{z}_1 + \ldots  + {g}'_{r,r_1} \widetilde{z}_{r_1}
+ \widetilde{F}_{1,r}(\mathbf{0}, \mathbf{0}, \mathbf{z})|_{\mathbf{\widetilde{z}} = \mathbf{0}},
$$
where ${g}'_{r,1}, \ldots ,  {g}'_{r,r_1} \in \mathbb{Z}$.
Let $\mathcal{R}_+^{(1)}(\Psi)$ be a maximal linearly independent (over $\mathbb{Q}$) subset of
$$
{\mathcal{R}}^{(1)}(\Psi) \cup \{\widetilde{F}_{1,1}(\mathbf{0}, \mathbf{0}, \mathbf{z})|_{\mathbf{\widetilde{z}} = \mathbf{0}}, \ldots, \widetilde{F}_{1,r_1}(\mathbf{0}, \mathbf{0}, \mathbf{z})|_{\mathbf{\widetilde{z}} = \mathbf{0}}   \}.
$$
The important thing to note is that by our construction, we have that
the set of linear forms
$$
\{ g'_{r,1} \widetilde{z}_1 + \ldots  + g'_{r,r_1} \widetilde{z}_{r_1} : 1 \leq r \leq r_1 \} \cup \mathcal{R}_+^{(1)}(\Psi)
$$
is linearly independent over $\mathbb{Q}$.  We also have that
$| \mathcal{R}_+^{(1)}(\Psi) | \leq | {\mathcal{R}}^{(1)}(\Psi) | + r_1$.

We replace $\mathcal{R}^{(1)}(\Phi)$ of $\mathcal{R}(\Phi)$ with $\mathcal{R}_+^{(1)}(\Phi)$
and refer to the resulting set of forms as $\mathcal{R}_+(\Phi)$.
It follows easily from the construction that the linear forms of $\mathcal{R}^{(1)}_+(\Phi)$ are either only in the  $\mathbf{y}$ variables, or
only in the $\mathbf{z}$ variables.
We denote
$$
\mathcal{R}_+(\Phi) = (\mathcal{R}^{(d-1)}(\Phi), \ldots, \mathcal{R}^{(2)}(\Phi), \mathcal{R}^{(1)}_+(\Phi) ),
$$
and by abusing notation slightly let
$$
\mathcal{R}^{(1)}_+(\Phi) = \{ A^{(1)}_i : 1 \leq i \leq | \mathcal{R}^{(1)}_+(\Phi) | \}   \text{ and  } \
\overline{\mathcal{R}}^{(1)}_+(\Phi) = \{ A^{(1)}_i(\mathbf{y}) : 1 \leq i \leq | \overline{\mathcal{R}}^{(1)}_+(\Phi) | \}.
$$
We then define $\overline{\mathcal{R}}_+(\Phi)$, $R'_{2}$, and $D'_{2}$ for $\mathcal{R}_+(\Phi)$ in an analogous manner as
$\overline{\mathcal{R}}(\Phi)$, $R_{2}$, and $D_{2}$ for  $\mathcal{R}(\Phi)$, respectively.
Similarly, we replace $\mathcal{R}^{(1)}(\Psi)$ of $\mathcal{R}(\Psi)$ with $\mathcal{R}_+^{(1)}(\Psi)$ and refer to the resulting set of forms as $\mathcal{R}_+(\Psi)$.
We denote
$$
\mathcal{R}_+(\Psi) = (\mathcal{R}^{(d-1)}(\Psi), \ldots, \mathcal{R}^{(2)}(\Psi), \mathcal{R}^{(1)}_+(\Psi) ),
$$
and by abusing notation slightly let
$$
\mathcal{R}^{(1)}_+(\Psi) = \{ V^{(1)}_i(\mathbf{z}) : 1 \leq i \leq | \mathcal{R}^{(1)}_+(\Psi) | \}.
$$
We also define $R'_{1}$ and $D'_{1}$ for $\mathcal{R}_+(\Psi)$ in an analogous manner as
$R_{1}$ and $D_{1}$ for $\mathcal{R}(\Psi)$, respectively. It then follows that we have $R'_2 \leq R_2 + r_1$ and
$R'_1 \leq R_1 + r_1$.

For each $\mathbf{H} \in \mathbb{Z}^{ R'_1 }$, we define the following set
$$
Z(\mathbf{H}) = \{ \mathbf{z} \in [0, X]^{n-M-K} \cap \mathbb{Z}^{n -M -K} : \mathcal{R}_+(\Psi) (\mathbf{z}) = \mathbf{H}^{} \}.
$$
By ${\mathcal{R}_+}(\Psi) ( \mathbf{z}) = \mathbf{H}$, we mean that
$V^{(s)}_i(\mathbf{z}) = H_{s,i}$, where $H_{s,i}$ is the corresponding term of $\mathbf{H}$,
for every $V^{(s)}_i \in {\mathcal{R}_+}(\Psi)$. Other instances of notation of this type should be interpreted in a similar manner.
By Proposition \ref{prop reg par}, we know that each of the polynomials $F_{\ell, r} (\mathbf{0}, \mathbf{0}, (\mathbf{0}, \mathbf{b}))$
$(2 \leq \ell < d, 1 \leq r \leq r_{\ell})$ and $G^{(k)}_{\ell,r} (\mathbf{z})$
in ~(\ref{first decomp of f}) can be expressed as a rational polynomial in the forms of $\mathcal{R}_+(\Psi^{})$.
Let us denote
$$
F_{\ell,r} (\mathbf{0}, \mathbf{0}, (\mathbf{0}, \mathbf{b})) = {c}^{\sharp}_{\ell, r}( \mathcal{R}_+(\Psi^{}) ),
$$
and
$$
G^{(k)}_{\ell,r} ( \mathbf{z}) = {c}^{\flat}_{\ell, r : k}( \mathcal{R}_+(\Psi^{}) ),
$$
where ${c}^{\sharp}_{\ell, r}$ and ${c}^{\flat}_{\ell, r: k}$ are rational polynomials in $R'_1$ variables.
Therefore, for any $\mathbf{z}_0 = (\mathbf{a}_0, \mathbf{b}_0) \in Z(\mathbf{H})$ we have
$$
F_{\ell,r}(\mathbf{0}, \mathbf{0}, (\mathbf{0}, \mathbf{b}_0)) = {c}^{\sharp}_{\ell, r} ( \mathbf{H}^{} )
\ \ \text{ and } \ \
G^{(k)}_{\ell, r} ( \mathbf{z}_0) = {c}^{\flat}_{\ell, r: k}( \mathbf{H} ).
$$
We also know that $\widetilde{F}_{1,r}(\mathbf{0}, \mathbf{0}, \mathbf{z})|_{\mathbf{\widetilde{z}} = \mathbf{0}}$ is constant on
$Z(\mathbf{H})$, and we denote this constant value by ${c}^{\sharp}_{1, r} ( \mathbf{H}^{} )$.

Similarly, we know that each of the polynomials $\Psi^{(k)}_{\ell,r: t_1, \ldots, t_{j}} (\mathbf{z})$ in ~(\ref{first decomp of f}) can be expressed as a rational polynomial in the forms of $\mathcal{R}_+(\Psi^{})$. Let us denote
\begin{equation}
\label{eqn hyaaa2}
\Psi^{(k)}_{\ell, r: t_1, \ldots, t_{j}} (\mathbf{z}) = {\hat{c} }^{(k)}_{\ell, r :  t_1, \ldots, t_{j}} ( \mathcal{R}_+(\Psi^{}) ),
\end{equation}
where ${\hat{c} }^{(k)}_{ \ell, r : t_1, \ldots, t_{j}} $ is a rational polynomial in $R'_1$ variables.
Therefore, for any $\mathbf{z}_0 \in Z(\mathbf{H})$ we have
$$
\Psi^{(k)}_{\ell, r : t_1, \ldots, t_{j}}(\mathbf{z}_0) = {\hat{c} }^{(k)}_{\ell, r :  t_1, \ldots, t_{j}} ( \mathbf{H}^{} ).
$$

Since each of the forms $\widetilde{\Psi}^{( s - k)}_{s,i: i_1, \ldots, i_k}(\mathbf{z})$ in (\ref{defn of a's}) can be expressed as a rational polynomial
in the forms of $\mathcal{R}_+(\Psi^{})$, let us denote
$$
\widetilde{\Psi}^{( s - k)}_{s,i: i_1, \ldots, i_k}(\mathbf{z}) = \widetilde{c}^{(s - k)}_{{s}, i: i_1, \ldots, i_k}(\mathcal{R}_+(\Psi^{})),
$$
where each $\widetilde{c}^{( s - k)}_{s, i: i_1, \ldots, i_k}$ is a rational polynomial in $R'_1$ variables.
Therefore, for each $A^{(s)}_i$ with  
$1 < s \leq d-1$ and $1 \leq i \leq | \mathcal{R}^{(s)}(\Phi) |$,  
we can write
\begin{equation}
\label{def of a's 2}
A^{({ s})}_i(\mathbf{y}, \mathbf{z}) = \sum_{k=0}^{s} \sum_{1 \leq i_1 \leq \ldots \leq i_k \leq M} \widetilde{c}^{( s - k)}_{s, i: i_1, \ldots, i_k}(\mathcal{R}_+(\Psi^{})) y_{i_1} \ldots y_{i_k}.
\end{equation}
Consequently, we can define the following polynomial for each $1 < s \leq d-1$ and $1 \leq i \leq |\mathcal{R}^{(s)}(\Phi) |$,
\begin{equation}
\label{def of a's 3}
A^{({s})}_i(\mathbf{y}, Z(\mathbf{H})  ) = \sum_{k=0}^{s} \sum_{1 \leq i_1 \leq \ldots \leq i_k \leq M} \widetilde{c}^{( s - k)}_{s,i: i_1, \ldots, i_k}(\mathbf{H}^{} ) y_{i_1} \ldots y_{i_k},
\end{equation}
so that given any $\mathbf{z}_0 \in Z(H)$ we have
$$
A^{({s})}_i(\mathbf{y}, \mathbf{z}_0  ) = A^{({ s})}_i(\mathbf{y}, Z(\mathbf{H}) ).
$$
We then define
\begin{eqnarray}
\overline{\mathcal{R}}_+(\Phi)(\mathbf{y}, Z(\mathbf{H}) ) &=& \{ A^{({s})}_i(\mathbf{y}, Z(\mathbf{H})  ) :
2 \leq s \leq d - 1, 1 \leq i \leq | \overline{\mathcal{R}}^{(s)}(\Phi)  | \}
\cup 
\overline{\mathcal{R}}_+^{(1)}(\Phi),
\notag
\end{eqnarray}
which is a system consisting of $R'_2$ polynomials (with possible repetitions).

For each $\mathbf{G} \in \mathbb{Z}^{R'_2}$, we let
$$
Y(\mathbf{G};\mathbf{H}) = \{ \mathbf{y} \in [0, X]^{ M } \cap \mathbb{Z}^{M} : \overline{\mathcal{R}}_+(\Phi) (\mathbf{y}, Z(\mathbf{H}) ) = \mathbf{G} \}.
$$
It follows from the definition of $\overline{\mathcal{R}}^{(1)}_+(\Phi)$ that for each $1 \leq r \leq r_1$ the polynomial $\widetilde{F}_{1,r}(\mathbf{0}, \mathbf{y}, \mathbf{0})|_{\mathbf{\widetilde{y}} = \mathbf{0}}$
is constant on $Y(\mathbf{G};\mathbf{H})$, and we denote this constant value by $c'_{1,r}(\mathbf{G}, \mathbf{H})$.

Recall $\Phi$ is the collection of all $\Phi^{(k)}_{\ell, r : i_1, \ldots, i_{j}} ( \mathbf{y}, \mathbf{z} )$ in ~(\ref{first decomp of f}), and
that each $\Phi^{(k)}_{\ell, r : i_1, \ldots, i_{j}} ( \mathbf{y}, \mathbf{z} )$ can be expressed as
a rational polynomial in the forms of $\mathcal{R}_+(\Phi)$.
It follows from our definition that the forms of $\mathcal{R}_+(\Phi)$ which depend only on the $\mathbf{z}$ variables are constant on $Z(\mathbf{H})$,
and the remaining forms, which are precisely the forms of $\overline{\mathcal{R}}_+(\Phi)$, are constant on $Y(\mathbf{G};\mathbf{H})\times Z(\mathbf{H})$.
Thus 
each $\Phi^{(k)}_{\ell, r: i_1, \ldots, i_{j}} ( \mathbf{y}, \mathbf{z} )$
is constant on $
Y(\mathbf{G};\mathbf{H}) \times Z(\mathbf{H})$, and we denote this constant value by
${c}^{(k)}_{\ell, r: i_1, \ldots, i_{j}} ( \mathbf{G}, \mathbf{H} )$. Let $2 \leq \ell < d$ and $1 \leq r \leq r_{\ell}$.
Therefore, for any choice of $\mathbf{z} = (\mathbf{a}, \mathbf{b} )\in Z(\mathbf{H})$ and $ \mathbf{y} \in  Y(\mathbf{G};\mathbf{H})$, the polynomial ${f}_{\ell,r}(\mathbf{x})$ takes the following
shape
\begin{eqnarray}
&&{f}_{\ell,r}( \mathbf{w}, \mathbf{y}, \mathbf{z} )
\label{decom of b with coeff}
\\
&=&
{f}_{\ell,r}( \mathbf{w}, \mathbf{0}, \mathbf{0})
+ \sum_{j=1}^{\ell -1}  \ \sum_{1 \leq i_1 \leq \ldots \leq i_j \leq K}
\left( \sum_{k=1}^{\ell -j} {c}^{(k)}_{\ell, r: i_1, \ldots, i_{j}} ( \mathbf{G}, \mathbf{H} )  \right) w_{i_1} \ldots w_{i_j}
\notag
\\
&+&
\sum_{j=1}^{\ell -1} \ \sum_{1 \leq t_1 \leq \ldots \leq t_{j} \leq M}
\left( \sum_{k=0}^{\ell - j} {\hat{c} }^{(k)}_{\ell, r:  t_1, \ldots, t_{j}} ( \mathbf{H} )  \right) y_{t_1} \ldots y_{t_{j}}
+ {F}_{\ell,r}( \mathbf{0}, \mathbf{y}, \mathbf{0} )
\notag
\\
&+& {P}_{\ell, r}(\mathbf{a}, \mathbf{b}) +
{c}^{\sharp}_{\ell, r} ( \mathbf{H}^{} )
+
\sum_{k=1}^{\ell - 1} {c}^{\flat}_{\ell, r: k}( \mathbf{H} ).
\notag
\end{eqnarray}
When $\ell = d$, we replace the term ${P}_{ \ell, r }(\mathbf{a}, \mathbf{b}) +
{c}^{\sharp}_{\ell, r} ( \mathbf{H}^{} )$ in (\ref{decom of b with coeff}) with $F_{d, r}(\mathbf{0}, \mathbf{0}, \mathbf{z})$.
Similarly, when $\ell = 1$ and $1 \leq r \leq r_{1}$, for any choice of $\mathbf{z} = (\mathbf{a}, \mathbf{b} )\in Z(\mathbf{H})$ and $ \mathbf{y} \in  Y(\mathbf{G};\mathbf{H})$, the polynomial ${f}_{1,r}(\mathbf{x})$ takes the following shape
\begin{eqnarray}
{f}_{1,r}(\mathbf{x}) &=& c_{1,r} w_r + \widetilde{f}_{1,r} (\mathbf{w}, \mathbf{0}, \mathbf{0})
+ ({g}_{r,1} \widetilde{y}_1 + \ldots  +  {g}_{r,r_1} \widetilde{y}_{r_1})
\notag
\\
&+& c'_{1,r}(\mathbf{G}, \mathbf{H})
+ ({g}'_{r,1} \widetilde{z}_1 + \ldots  +  {g}'_{r,r_1} \widetilde{z}_{r_1})  + {c}^{\sharp}_{1, r} ( \mathbf{H}^{} )
\notag,
\end{eqnarray}
where $\widetilde{f}_{1,r}$ is defined in (\ref{set of eqn 2}).

For each $2 \leq \ell \leq d, 1 \leq r \leq r_{\ell}$, we label
\begin{eqnarray}
\mathfrak{C}_{\ell,r} (\mathbf{w}, \mathbf{G}, \mathbf{H}  ) &=& {f}_{\ell,r} ( \mathbf{w}, \mathbf{0}, \mathbf{0})
+ \sum_{j=1}^{\ell -1}  \ \sum_{1 \leq i_1 \leq \ldots \leq i_j \leq K}
\left( \sum_{k=1}^{\ell -j} {c}^{(k)}_{\ell, r: i_1, \ldots, i_{j}} ( \mathbf{G}, \mathbf{H} )  \right) w_{i_1} \ldots w_{i_j},
\notag
\end{eqnarray}
and
\begin{equation}
\label{eqn hyaaa3}
\mathfrak{U}_{\ell,r} (\mathbf{y}, \mathbf{H} ) = \sum_{j=1}^{\ell -1} \ \sum_{1 \leq t_1 \leq \ldots \leq t_{j} \leq M}
\left( \sum_{k=0}^{\ell - j} {\hat{c} }^{(k)}_{\ell, r:  t_1, \ldots, t_{j}} ( \mathbf{H} )  \right) y_{t_1} \ldots y_{t_{j}}
+ {F}_{\ell,r}( \mathbf{0}, \mathbf{y}, \mathbf{0} ).
\end{equation}
We let
$$
\mathfrak{X}_{\ell,r} (\mathbf{a}, \mathbf{b}, \mathbf{H} ) = {P}_{\ell,r}(\mathbf{a}, \mathbf{b}) +
{c}^{\sharp}_{\ell, r} ( \mathbf{H}^{} )
+
\sum_{k=1}^{\ell - 1} {c}^{\flat}_{\ell, r: k}( \mathbf{H} ) \ \ (2 \leq \ell <  d, 1 \leq r \leq r_{\ell}),
\notag
$$
and also
$$
\mathfrak{X}_{d,r} (\mathbf{a}, \mathbf{b}, \mathbf{H} ) = F_{d, r}(\mathbf{0}, \mathbf{0}, \mathbf{z})
+
\sum_{k=1}^{d - 1} {c}^{\flat}_{d, r: k}( \mathbf{H} ) \ \ (1 \leq r \leq r_d).
\notag
$$
Then for $\mathbf{z} = (\mathbf{a}, \mathbf{b}) \in Z(\mathbf{H})$ and $ \mathbf{y} \in  Y(\mathbf{G};\mathbf{H})$, we have
$$
{f}_{\ell,r} (\mathbf{w}, \mathbf{y}, \mathbf{z}  ) =  \mathfrak{C}_{\ell,r} (\mathbf{w}, \mathbf{G}, \mathbf{H}  ) +
\mathfrak{U}_{\ell,r} (\mathbf{y}, \mathbf{H} ) +\mathfrak{X}_{\ell,r} (\mathbf{a}, \mathbf{b}, \mathbf{H} )
\ \ \ (2 \leq \ell \leq d, 1 \leq r \leq r_{\ell}).
$$

We define the following three exponential sums,
\begin{eqnarray}
\notag
S_0( \boldsymbol{\alpha}, \mathbf{G}, \mathbf{H} ) &=&
\sum_{ \mathbf{w} \in [0,X]^K } \Lambda(\mathbf{w})  \
e  \Big{(} \sum_{1 \leq r \leq r_1}  \alpha_{1,r} (c_{1,r} w_r + \widetilde{f}_{1,r} (\mathbf{w}, \mathbf{0}, \mathbf{0})
 )
\\
\notag
&+& \sum_{2 \leq \ell \leq d} \sum_{1 \leq r \leq r_{\ell}} \alpha_{\ell,r}  \cdot \mathfrak{C}_{\ell,r} (\mathbf{w}, \mathbf{G}, \mathbf{H}  )  \Big{)},
\end{eqnarray}
\begin{eqnarray}
\notag
S_1( \boldsymbol{\alpha}, \mathbf{G}, \mathbf{H}) &=& \sum_{\mathbf{y} \in Y(\mathbf{G};\mathbf{H})} \Lambda (\mathbf{y}) \  e \Big{(}
\sum_{1 \leq r \leq r_1} \alpha_{1,r} (
{g}_{r,1} \widetilde{y}_1 + \ldots  +  {g}_{r,r_1} \widetilde{y}_{r_1}
+  c'_{1,r}(\mathbf{G}, \mathbf{H}) )
\\
&+& \sum_{2 \leq \ell \leq d} \sum_{1 \leq r \leq r_{\ell}} \alpha_{\ell,r}  \cdot  \mathfrak{U}_{\ell,r} (\mathbf{y}, \mathbf{H}  )  \Big{)},
\notag
\end{eqnarray}
and
\begin{eqnarray}
\notag
S_2( \boldsymbol{\alpha}, \mathbf{H} ) &=&  \sum_{\mathbf{z} = (\mathbf{a}, \mathbf{b}) \in Z(\mathbf{H}) } \Lambda (\mathbf{z})  \
e \Big{(} \sum_{1 \leq r \leq r_1}
\alpha_{1,r} ({g}'_{r,1} \widetilde{z}_1 + \ldots  +  {g}'_{r,r_1} \widetilde{z}_{r_1} + {c}^{\sharp}_{1, r} ( \mathbf{H}^{} ))
\\
&+&
\sum_{2 \leq \ell \leq d} \sum_{1 \leq r \leq r_{\ell}}
{\alpha}_{\ell,r} \cdot  \mathfrak{X}_{\ell,r} (\mathbf{a}, \mathbf{b}, \mathbf{H} )  \Big{)}.
\notag
\end{eqnarray}

Let
\begin{eqnarray}
\mathcal{L}_1(X) 
= \{ \mathbf{H} \in \mathbb{Z}^{R'_1 } : Z(\mathbf{H}) \not = \emptyset \},
\notag
\end{eqnarray}
and for each $\mathbf{H} \in \mathcal{L}_1(X)$,  let
\begin{eqnarray}
\mathcal{L}_2(X ; \mathbf{H}) 
= \{ \mathbf{G} \in \mathbb{Z}^{R'_2 } : Y(\mathbf{G}, \mathbf{H} ) \not = \emptyset \}.
\notag
\end{eqnarray}
We have the following bounds on the cardinalities of these sets,
$$
|\mathcal{L}_1(X)| \ll X^{D'_1} \  \mbox{ and } \  |\mathcal{L}_2(X; \mathbf{H})| \ll X^{D'_2}.
$$
It is not difficult to deduce the first inequality. The implicit constant in the second inequality is independent of $\mathbf{H}$, and to see this we note
that given $A^{(s)}_i$ with $1 < s \leq d-1$ and $1 \leq i \leq | \overline{\mathcal{R}}^{(s)}(\Phi) |$,
we have
\begin{equation}
\notag
|A^{(s)}_i(\mathbf{y}, \mathbf{z})| = \Big{|} \sum_{k=0}^{s} \sum_{1 \leq i_1 \leq \ldots \leq i_k \leq M} \widetilde{\Psi}^{(s - k)}_{s, i: i_1, \ldots, i_k}(\mathbf{z}) y_{i_1} \ldots y_{i_k}  \Big{|} \ll X^s
\end{equation}
for any $(\mathbf{y}, \mathbf{z}) \in [0, X]^{n - K} \cap \mathbb{Z}^{n- K}$, and similarly for the linear forms
of $\overline{\mathcal{R}}^{(1)}_+(\Phi)$.
Therefore, we obtain by applying the Cauchy-Schwarz inequality
\begin{eqnarray}
\label{minor arc ineq 1}
&& \Big{|} \int_{\mathfrak{m}(C)} T(\mathbf{f}; \boldsymbol{\alpha} ) \ \mathbf{d} \boldsymbol{\alpha}  \Big{|}^2
\\
&\leq&
\Big{|}
\sum_{\mathbf{H} \in \mathcal{L}_1(X)} \sum_{\mathbf{G} \in \mathcal{L}_2(X ; \mathbf{H})} \int_{\mathfrak{m}(C)} \
\sum_{\substack{ \mathbf{w} \in [0,X]^K \\  \mathbf{y} \in Y(\mathbf{G}; \mathbf{H}) \\ \mathbf{z} = (\mathbf{a}, \mathbf{b}) \in Z(\mathbf{H})  }}
\Lambda(\mathbf{w}) \Lambda(\mathbf{y}) \Lambda(\mathbf{z})  \cdot
\notag
\\
&&e \Big{(} \sum_{1 \leq r \leq r_1}  \alpha_{1,r} (c_{1,r} w_r + \widetilde{f}_{1,r} (\mathbf{w}, \mathbf{0}, \mathbf{0})
+ ({g}_{r,1} \widetilde{y}_1 + \ldots  +  {g}_{r,r_1} \widetilde{y}_{r_1}) + c'_{1,r}(\mathbf{G}, \mathbf{H})
\notag
\\
\notag
&& + ({g}'_{r,1} \widetilde{z}_1 + \ldots  +  {g}'_{r,r_1} \widetilde{z}_{r_1})   + {c}^{\sharp}_{1, r} ( \mathbf{H}^{} )  ) \Big{)} \cdot
\\
\notag
&&e \left( \sum_{2 \leq \ell \leq d} \sum_{1 \leq r \leq r_{\ell}}
 \alpha_{\ell,r} \cdot (\mathfrak{C}_{\ell,r} (\mathbf{w}, \mathbf{G}, \mathbf{H}  ) +
 \mathfrak{U}_{\ell,r} (\mathbf{y}, \mathbf{H} ) + \mathfrak{X}_{\ell,r} (\mathbf{a}, \mathbf{b}, \mathbf{H} )   \  )
\right)
\  \mathbf{d} \boldsymbol{\alpha}  \Big{|}^2
\notag
\end{eqnarray}
\begin{eqnarray}
\notag
&\ll&
X^{D'_1 + D'_2}
\sum_{\mathbf{H} \in \mathcal{L}_1(X)} \sum_{\mathbf{G} \in \mathcal{L}_2(X ; \mathbf{H})} \Big{|}  \int_{\mathfrak{m}(C)}
  S_0 ( \boldsymbol{\alpha}, \mathbf{G}, \mathbf{H} )
S_1( \boldsymbol{\alpha}, \mathbf{G}, \mathbf{H} )S_2( \boldsymbol{\alpha}, \mathbf{H} )
 \  \mathbf{d}  \boldsymbol{\alpha} \Big{|}^2
\notag
\\
&\ll&
X^{D'_1 + D'_2} \  \left( \sup_{\substack {\mathbf{H} \in \mathcal{L}_1(X) \\ \mathbf{G} \in \mathcal{L}_2(X ; \mathbf{H}) }}  \sup_{\boldsymbol{\alpha} \in \mathfrak{m}(C) } | S_0 ( \boldsymbol{\alpha}, \mathbf{G}, \mathbf{H} ) |^2 \right)   \sum_{\mathbf{H} \in \mathcal{L}_1(X)} \sum_{\mathbf{G} \in \mathcal{L}_2(X ; \mathbf{H})}   \|S_1(\cdot, \mathbf{G}, \mathbf{H} ) \|_2^2 \ \| S_2(\cdot, \mathbf{H} ) \|_2^2,
\notag
\end{eqnarray}
where $\| \cdot \|_2$ denotes the $L^2$-norm on $[0,1]^R$.

By the orthogonality relation, it follows that
\begin{eqnarray}
\notag
\|S_1(\cdot, \mathbf{G}, \mathbf{H} ) \|_2^2 \ \| S_2(\cdot, \mathbf{H} ) \|_2^2
&\leq&
(\log X)^{2n - 2K} \mathcal{N}^{}_1 (\mathbf{G}; \mathbf{H}) \mathcal{N}^{}_2 (\mathbf{H}),
\notag
\end{eqnarray}
where
\begin{eqnarray}
\notag
&&\mathcal{N}^{ }_1 (\mathbf{G}; \mathbf{H}) =  | \{ (\mathbf{y}, \mathbf{y}' )  \in Y(\mathbf{G}; \mathbf{H}) \times Y(\mathbf{G}; \mathbf{H}):
\mathfrak{U}_{\ell,r} (\mathbf{y}, \mathbf{H}  ) = \mathfrak{U}_{\ell,r} (\mathbf{y}', \mathbf{H}  ) \ \  (2 \leq \ell \leq d, 1 \leq r \leq r_{\ell}),
\\
&&{g}_{r,1} \widetilde{y}_1 + \ldots  +  {g}_{r,r_1} \widetilde{y}_{r_1} = {g}_{r,1} \widetilde{y}'_1 + \ldots  +  {g}_{r,r_1} \widetilde{y}'_{r_1}
\ \ (1 \leq r \leq r_1)
\} |,
\notag
\end{eqnarray}
and with $\mathbf{z} = (\mathbf{a}, \mathbf{b})$ and $\mathbf{z}' = (\mathbf{a}', \mathbf{b}')$,
\begin{eqnarray}
&&\mathcal{N}_2 (\mathbf{H}) =  |\{ (\mathbf{z}, \mathbf{z}')  \in Z(\mathbf{H}) \times Z(\mathbf{H}):
\mathfrak{X}_{\ell,r} (\mathbf{a}, \mathbf{b}, \mathbf{H} )  = \mathfrak{X}_{\ell,r} (\mathbf{a}', \mathbf{b}', \mathbf{H} ) \ \   (2 \leq \ell \leq d, 1 \leq r \leq r_{\ell} ),
\notag
\\
&&{g}'_{r,1} \widetilde{z}_1 + \ldots  +  {g}'_{r,r_1} \widetilde{z}_{r_1} = {g}'_{r,1} \widetilde{z}'_1 + \ldots  +  {g}'_{r,r_1} \widetilde{z}'_{r_1}
\ \ (1 \leq r \leq r_1)
\}|.
\notag
\end{eqnarray}
Here $\widetilde{y}'_i$ $(1 \leq i \leq r_1)$ are $r_1$ of the $\mathbf{y}'$ variables in the exact same way
$\widetilde{y}_i \ (1 \leq i \leq r_1)$ are $r_1$ of the $\mathbf{y}$ variables. Similarly,
$\widetilde{z}'_i$ $(1 \leq i \leq r_1)$ are $r_1$ of the $\mathbf{z}'$ variables in the exact same way
$\widetilde{z}_i \ (1 \leq i \leq r_1)$ are $r_1$ of the $\mathbf{z}$ variables. Other instances of notation
of this type should be interpreted in a similar manner.

With these notations, we may further bound ~(\ref{minor arc ineq 1}) as follows
\begin{eqnarray}
\label{minor arc ineq 2}
\\
\notag
\Big{|} \int_{\mathfrak{m}(C)} T(\mathbf{f}; \boldsymbol{\alpha} ) \ \mathbf{d} \boldsymbol{\alpha}  \Big{|}^2
\ll
(\log X)^{2n  - 2K}  X^{D'_1+D'_2} \ \left( \sup_{\substack {\mathbf{H} \in \mathcal{L}_1(X) \\ \mathbf{G} \in \mathcal{L}_2(X ; \mathbf{H}) }}
\sup_{ \boldsymbol{\alpha} \in \mathfrak{m}(C) } | S_0 ( \boldsymbol{\alpha}, \mathbf{G}, \mathbf{H} ) |^2 \right)
\ \mathcal{W},
\end{eqnarray}
where
$$
\mathcal{W} = \sum_{\mathbf{H} \in \mathcal{L}_1(X)} \sum_{\mathbf{G} \in \mathcal{L}_2(X; \mathbf{H})} \mathcal{N}_1 (\mathbf{G}; \mathbf{H})  \mathcal{N}_2 (\mathbf{H}).
$$
We can express $\mathcal{W}$ as the number of solutions $\mathbf{y}, \mathbf{y}'  \in [0,X]^M \cap \mathbb{Z}^M$ and  $\mathbf{z} = (\mathbf{a}, \mathbf{b}), \mathbf{z}'= (\mathbf{a}', \mathbf{b}') \in [0,X]^{n-M-K} \cap \mathbb{Z}^{n-M-K} $
of the system
\begin{eqnarray}
{\mathcal{R}_+}(\Psi) ( \mathbf{z}) &=& {\mathcal{R}_+}(\Psi) (\mathbf{z}') = \mathbf{H}
\label{first system}
\\
\overline{\mathcal{R}}_+(\Phi) (\mathbf{y},  Z(\mathbf{H})) &=& \overline{\mathcal{R}}_+(\Phi) (\mathbf{y}', Z(\mathbf{H}))= \mathbf{G}
\notag
\\
\mathfrak{U}_{\ell,r} (\mathbf{y}, \mathbf{H}  )  &=&  \mathfrak{U}_{\ell,r} (\mathbf{y}', \mathbf{H}  ) \ \   ( 2 \leq \ell \leq d, 1 \leq r \leq r_{\ell} )
\notag
\\
{g}_{r,1} \widetilde{y}_1 + \ldots  +  {g}_{r,r_1} \widetilde{y}_{r_1} &=& {g}_{r,1} \widetilde{y}'_1 + \ldots  +  {g}_{r,r_1} \widetilde{y}'_{r_1}
\ \ (1 \leq r \leq r_1)
\notag
\\
\mathfrak{X}_{\ell,r} (\mathbf{a}, \mathbf{b}, \mathbf{H} )  &=& \mathfrak{X}_{\ell,r} (\mathbf{a}', \mathbf{b}', \mathbf{H} ) \ \ (2 \leq \ell \leq d, 1 \leq r \leq r_{\ell} )
\notag
\\
{g}'_{r,1} \widetilde{z}_1 + \ldots  +  {g}'_{r,r_1} \widetilde{z}_{r_1} &=& {g}'_{r,1} \widetilde{z}'_1 + \ldots  +  {g}'_{r,r_1} \widetilde{z}'_{r_1}
\ \ (1 \leq r \leq r_1)
\notag
\end{eqnarray}
for any $\mathbf{H}\in \mathcal{L}_1(X)$ and $\mathbf{G}\in \mathcal{L}_2(X; \mathbf{H})$.
By ${\mathcal{R}_+}(\Psi) ( \mathbf{z}) = {\mathcal{R}_+}(\Psi) (\mathbf{z}') = \mathbf{H}$, we mean that
$V^{(s)}_i(\mathbf{z}) = V^{(s)}_i(\mathbf{z}') = H_{s,i}$, where $H_{s,i}$ is the corresponding term of $\mathbf{H}$,
for every $V^{(s)}_i \in {\mathcal{R}_+}(\Psi)$. The second set of equations in (\ref{first system}) should be interpreted in
a similar manner.

We know that the system of polynomials  $\overline{\mathcal{R}}_+(\Phi) (\mathbf{y},  Z(\mathbf{H}))$ is identical to $\overline{\mathcal{R}}_+(\Phi) (\mathbf{y},  \mathbf{z}_0)$ for any choice of $\mathbf{z}_0 \in Z(\mathbf{H})$. 
Similarly, it follows from (\ref{eqn hyaaa1}), (\ref{eqn hyaaa2}), and (\ref{eqn hyaaa3}) that the polynomial
$\mathfrak{U}_{\ell,r} (\mathbf{y}, \mathbf{H})$ is identical to $f_{\ell,r}( \mathbf{0}, \mathbf{y}, \mathbf{z}_0) - f_{\ell,r}( \mathbf{0}, \mathbf{0}, \mathbf{z}_0)$ for any choice of $\mathbf{z}_0 \in Z(\mathbf{H})$.
Furthermore, for $2 \leq \ell < d$ we know that each term of $\mathfrak{X}_{\ell,r} (\mathbf{a}, \mathbf{b}, \mathbf{H} )$ except for $P_{\ell,r} (\mathbf{a}, {\mathbf{b}} )$
is constant on $\mathbf{z} = (\mathbf{a}, \mathbf{b}) \in Z(\mathbf{H})$.
Therefore, since $\mathcal{R}_+(\Psi) (\mathbf{z}) = \mathbf{H}$
implies $\mathbf{z} \in Z(\mathbf{H})$, we can rearrange the system ~(\ref{first system}) and
deduce that $\mathcal{W}$ is the number of solutions $\mathbf{y}, \mathbf{y}'  \in [0,X]^M \cap \mathbb{Z}^M$
and $\mathbf{z}, \mathbf{z}' \in [0,X]^{n-M-K} \cap \mathbb{Z}^{n-M-K} $ of the following system
\begin{eqnarray}
\label{second system}
{\mathcal{R}_+}(\Psi) ( \mathbf{z}) &=& {\mathcal{R}_+}(\Psi) (\mathbf{z}')
\\
\overline{\mathcal{R}}_+(\Phi) (\mathbf{y},  \mathbf{z} ) &=& \overline{\mathcal{R}}_+(\Phi) (\mathbf{y}', \mathbf{z} )
\notag
\\
f_{\ell,r} ( \mathbf{0}, \mathbf{y}, \mathbf{z}) - f_{\ell,r}( \mathbf{0}, \mathbf{0}, \mathbf{z})   &=&  f_{\ell,r}( \mathbf{0}, \mathbf{y}', \mathbf{z}) - f_{\ell,r}( \mathbf{0}, \mathbf{0}, \mathbf{z})
\notag
\ \ ( 2 \leq \ell \leq d, 1 \leq r \leq r_{\ell} )
\notag
\\
{g}_{r,1} \widetilde{y}_1 + \ldots  +  {g}_{r,r_1} \widetilde{y}_{r_1} &=& {g}_{r,1} \widetilde{y}'_1 + \ldots  +  {g}_{r,r_1} \widetilde{y}'_{r_1}
\ \ (1 \leq r \leq r_1)
\notag
\\
{F}_{d,r}( \mathbf{0}, \mathbf{0}, \mathbf{z} ) &=& {F}_{d,r}( \mathbf{0}, \mathbf{0}, \mathbf{z}' ) \ \ (1 \leq r \leq r_{d} )
\notag
\\
P_{\ell,r}(\mathbf{a}, {\mathbf{b}} ) &=& P_{\ell,r}(\mathbf{a}', {\mathbf{b}}' ) \ \ (2 \leq \ell < d, 1 \leq r \leq r_{\ell} )
\notag
\\
{g}'_{r,1} \widetilde{z}_1 + \ldots  +  {g}'_{r,r_1} \widetilde{z}_{r_1} &=& {g}'_{r,1} \widetilde{z}'_1 + \ldots  +  {g}'_{r,r_1} \widetilde{z}'_{r_1}
\ \ (1 \leq r \leq r_1).
\notag
\end{eqnarray}

Our result then follows from the following two claims.

Claim 1: Given any $c>0$, for sufficiently large $C > 0$ we have
$$
\sup_{\substack {\mathbf{H} \in \mathcal{L}_1(X) \\ \mathbf{G} \in \mathcal{L}_2(X ; \mathbf{H}) }}  \sup_{\boldsymbol{\alpha} \in \mathfrak{m}(C) } | S_0 (\boldsymbol{\alpha}, \mathbf{G}, \mathbf{H} ) |
\ll \frac{X^{K}}{(\log X)^c}.
$$

Claim 2: We have the following bound on $\mathcal{W}$,
$$
\mathcal{W} \ll X^{2n - 2K - 2 \sum_{\ell = 1}^d \ell r_{\ell} - D'_1 - D'_2}.
$$

Let $c > 0$. By substituting the bounds from the two claims above into ~(\ref{minor arc ineq 2}), we obtain that
for sufficiently large $C > 0$ we have
$$
\int_{\mathfrak{m}(C)} T(\mathbf{f}; \boldsymbol{\alpha} ) \ \mathbf{d} \boldsymbol{\alpha}
\ll
\frac{X^{n -  \sum_{\ell = 1}^d \ell r_{\ell}} }{(\log X)^{c}},
$$
which is the bound in the statement of this proposition. Therefore, we are only left to prove Claims $1$ and $2$ to establish our proposition.
We now present the proof of Claim 2.
Claim $1$ is obtained via Weyl differencing, which is a technique based on the Cauchy-Schwarz inequality, and we prove it in Section \ref{sec claim 1} after the proof of Claim 2.

From ~(\ref{second system}), we can write
$$
\mathcal{W} = \sum_{\mathbf{z} = (\mathbf{a}, \mathbf{b})  \in   [0,X]^{n-M-K} } \mathcal{W}'_1(\mathbf{z}) \cdot \mathcal{W}'_2(\mathbf{z}),
$$
where $\mathcal{W}'_1(\mathbf{z})$ is the number of solutions $\mathbf{y}, \mathbf{y}' \in [0,X]^{M} \cap \mathbb{Z}^{M}$
to the system
\begin{eqnarray}
\overline{\mathcal{R}}_+(\Phi) (\mathbf{y}, \mathbf{z}) &=& \overline{\mathcal{R}}_+(\Phi) (\mathbf{y}', \mathbf{z}),
\notag
\\
f_{\ell,r}( \mathbf{0}, \mathbf{y}, \mathbf{z}) - f_{\ell,r}( \mathbf{0}, \mathbf{0}, \mathbf{z})   &=&  f_{\ell,r}( \mathbf{0}, \mathbf{y}', \mathbf{z}) - f_{\ell,r}( \mathbf{0}, \mathbf{0}, \mathbf{z})
\ \  ( 2 \leq \ell \leq d, 1 \leq r \leq r_{\ell} )
\notag
\\
{g}_{r,1} \widetilde{y}_1 + \ldots  +  {g}_{r,r_1} \widetilde{y}_{r_1} &=& {g}_{r,1} \widetilde{y}'_1 + \ldots  +  {g}_{r,r_1} \widetilde{y}'_{r_1}
\ \ (1 \leq r \leq r_1),
\notag
\end{eqnarray}
and $\mathcal{W}'_2(\mathbf{z})$ is the number of solutions $\mathbf{z}' = (\mathbf{a}', \mathbf{b}') \in  [0,X]^{n-M-K} \cap \mathbb{Z}^{n-M-K}$ to the system
\begin{eqnarray}
\mathcal{R}_+(\Psi) ( \mathbf{z}) &=& \mathcal{R}_+(\Psi) (\mathbf{z}')
\notag
\\
{F}_{d,r}( \mathbf{0}, \mathbf{0}, \mathbf{z} ) &=& {F}_{d,r}( \mathbf{0}, \mathbf{0}, \mathbf{z}' ) \ \ ( 1 \leq r \leq r_{d} )
\notag
\\
P_{\ell,r}(\mathbf{a},  {\mathbf{b}} ) &=& P_{\ell,r}(\mathbf{a}',  {\mathbf{b}}' ) \ \ (2 \leq \ell < d, 1 \leq r \leq r_{\ell} )
\notag
\\
{g}'_{r,1} \widetilde{z}_1 + \ldots  +  {g}'_{r,r_1} \widetilde{z}_{r_1} &=& {g}'_{r,1} \widetilde{z}'_1 + \ldots  +  {g}'_{r,r_1} \widetilde{z}'_{r_1}
\ \ (1 \leq r \leq r_1).
\notag
\end{eqnarray}
Define $\mathcal{W}_i :=\sum_{\mathbf{z}} \mathcal{W}'_i (\mathbf{z})^2 \ (i=1, 2)$ so that
we have $\mathcal{W}^2 \leq \mathcal{W}_1 \mathcal{W}_2$ by the Cauchy-Schwarz inequality.
We estimate $\mathcal{W}_1$ and $\mathcal{W}_2$ in Sections \ref{estimate for W_1} and \ref{estimate for W_2}, respectively.
In Section \ref{estimate for W_1}, we prove $\mathcal{W}_1  \ll X^{n +  3M -  K - 2\sum_{\ell = 1}^d \ell r_{\ell} - 2D'_2}$,
and in Section \ref{estimate for W_2} we prove $\mathcal{W}_2  \ll X^{3(n - M -  K) - 2\sum_{\ell = 1}^d \ell r_{\ell} - 2D'_1}$.
Combining these bounds for $\mathcal{W}_1$ and $\mathcal{W}_2$, we obtain
$$
\mathcal{W} \leq \mathcal{W}_1^{1/2} \mathcal{W}_2^{1/2} \ll X^{2n -2 K - 2\sum_{\ell = 1}^d \ell r_{\ell} - D'_1 - D'_2},
$$
which proves Claim 2.

\subsection{Estimate for $\mathcal{W}_1$}
\label{estimate for W_1}
We first estimate $\mathcal{W}_1$, which
we can deduce to be the number of solutions $\mathbf{y}, \mathbf{y}', \mathbf{v}, \mathbf{v}' \in [0,X]^M \cap \mathbb{Z}^M$
and $\mathbf{z} \in [0,X]^{n-M-K} \cap \mathbb{Z}^{n-M-K}$ satisfying the equations
\begin{eqnarray}
\label{system 1 in claim 2}
f_{\ell,r} ( \mathbf{0}, \mathbf{y}, \mathbf{z})
-
f_{\ell,r} ( \mathbf{0}, \mathbf{y}', \mathbf{z})
&=& 0 \ \ (2 \leq \ell \leq d, 1 \leq r \leq r_{\ell})
\\
f_{\ell,r}( \mathbf{0}, \mathbf{v}, \mathbf{z})
-
f_{\ell,r}( \mathbf{0}, \mathbf{v}', \mathbf{z})  &=& 0\ \ (2 \leq \ell \leq d, 1 \leq r \leq r_{\ell})
\notag
\\
\sum_{i=1}^{r_1 }{g}_{r,i} \widetilde{y}_i -  \sum_{i=1}^{r_1 }{g}_{r,i} \widetilde{y}'_i &=& 0 \ \ (1 \leq r \leq r_1)
\notag
\\
\sum_{i=1}^{r_1 }{g}_{r,i} \widetilde{v}_i -  \sum_{i=1}^{r_1 }{g}_{r,i} \widetilde{v}'_i &=& 0 \ \ (1 \leq r \leq r_1)
\notag
\\
\overline{\mathcal{R}}_+(\Phi) (\mathbf{y}, \mathbf{z}) - \overline{\mathcal{R}}_+(\Phi) (\mathbf{y}', \mathbf{z}) &=&  \mathbf{0}
\notag
\\
\overline{\mathcal{R}}_+(\Phi) (\mathbf{v}, \mathbf{z}) - \overline{\mathcal{R}}_+(\Phi) (\mathbf{v}', \mathbf{z}) &=&  \mathbf{0}.
\notag
\end{eqnarray}
Let $\overline{\mathcal{R}}_+^{(i)}(\Phi)$ denote the degree $i$ forms of $\overline{\mathcal{R}}_+(\Phi)$ $(1 \leq i < d)$.
By $\overline{\mathcal{R}}_+(\Phi) (\mathbf{y}, \mathbf{z}) - \overline{\mathcal{R}}_+(\Phi) (\mathbf{y}', \mathbf{z})$,
we mean the system of forms where its degree $i$ forms are
$$
\overline{\mathcal{R}}_+^{(i)}(\Phi) (\mathbf{y}, \mathbf{z}) - \overline{\mathcal{R}}_+^{(i)}(\Phi) (\mathbf{y}', \mathbf{z})
=
\{ {A}_j^{(i)}(\mathbf{y}, \mathbf{z}) - {A}_j^{(i)}(\mathbf{y}', \mathbf{z}) : 1 \leq j \leq  | \overline{\mathcal{R}}_+^{(i)}(\Phi)|  \},
$$
for each $1 \leq i \leq d-1$. Recall we have $\overline{\mathcal{R}}_+^{(i)}(\Phi) =\overline{\mathcal{R}}^{(i)}(\Phi)$ for $2 \leq i \leq d-1$.
We also define
$$
\overline{\mathcal{R}}_+(\Phi)^{} (\mathbf{v}, \mathbf{z}) - \overline{\mathcal{R}}_+(\Phi)^{} (\mathbf{v}', \mathbf{z})
$$
in a similar manner.

We consider the $h$-invariant of the system of polynomials on the left hand side of (\ref{system 1 in claim 2}), and show that
it is a regular system.
Recall we defined $Q_{\ell,r} (\mathbf{y}, \mathbf{z} )$ in (\ref{def Q})
and also remarked that it is the degree $\ell$ portion of the polynomial $f_{\ell,r}( \mathbf{0}, \mathbf{y}, \mathbf{z}) - f_{\ell,r}( \mathbf{0}, \mathbf{0}, \mathbf{z}).$ Therefore, the degree $\ell$ portion of the polynomial
$f_{\ell,r}( \mathbf{0}, \mathbf{y}, \mathbf{z})-f_{\ell,r}( \mathbf{0}, \mathbf{y}', \mathbf{z})$
is precisely $Q_{\ell,r} (\mathbf{y}, \mathbf{z} )- Q_{\ell,r} (\mathbf{y}', \mathbf{z} )$.
Thus the degree $d$ portions of the degree $d$ polynomials of the system (\ref{system 1 in claim 2}) are
$Q_{d,r}(\mathbf{y}, \mathbf{z}) - Q_{d,r}(\mathbf{y}', \mathbf{z})$, $Q_{d,r}(\mathbf{v}, \mathbf{z}) - Q_{d,r}(\mathbf{v}', \mathbf{z})$
$(1 \leq r \leq r_d)$. We let $h_d$ be the $h$-invariant of these degree $d$ forms.
Suppose for some $\boldsymbol{\lambda}, \boldsymbol{\mu} \in \mathbb{Q}^{r_d}$, not both $\mathbf{0}$, we have
\begin{eqnarray}
\label{eqn 1}
\\
\notag
\sum_{r=1}^{r_d} \lambda_r \cdot (Q_{d,r}(\mathbf{y}, \mathbf{z}) - Q_{d,r}(\mathbf{y}', \mathbf{z}) )
+ \mu_r \cdot (Q_{d,r}(\mathbf{v}, \mathbf{z}) - Q_{d,r}(\mathbf{v}', \mathbf{z}) )
= \sum_{j=1}^{h_d} \widetilde{U}_j \cdot \widetilde{V}_j,
\end{eqnarray}
where $\widetilde{U}_j = \widetilde{U}_j( \mathbf{y}, \mathbf{y}', \mathbf{v}, \mathbf{v}', \mathbf{z} )$ and $\widetilde{V}_j =
\widetilde{V}_j( \mathbf{y}, \mathbf{y}', \mathbf{v}, \mathbf{v}', \mathbf{z} )$ are rational forms of positive degree $(1 \leq j \leq h_d)$.
Without loss of generality, suppose $\boldsymbol{\lambda}  \not = \mathbf{0}$. If we set
$\mathbf{v} = \mathbf{v}' = \mathbf{y}' = \mathbf{0}$,
then the above equation (\ref{eqn 1}) becomes
\begin{eqnarray}
\sum_{r=1}^{r_d} \lambda_r \cdot Q_{d,r}(\mathbf{y}, \mathbf{z})
= \sum_{j=1}^{h_d}  \widetilde{U}_j (\mathbf{y}, \mathbf{0}, \mathbf{0}, \mathbf{0}, \mathbf{z}) \cdot \widetilde{V}_j (\mathbf{y}, \mathbf{0}, \mathbf{0}, \mathbf{0}, \mathbf{z}).
\notag
\end{eqnarray}
Therefore, we obtain from (\ref{hinv bound1}) and (\ref{hinv bound1'}) that
\begin{eqnarray}
h_d &\geq& h_d \left(  \{ Q_{d,r}(\mathbf{y}, \mathbf{z}) : 1 \leq r \leq r_d \}  \right)
\notag
\\
&\geq& \rho_{d,d}( 2R + 2 R_2) + 2 R_2 + 4 r_1
\notag
\\
&\geq& \rho_{d,d}( 2R + 2 R'_2 - 2 | \overline{\mathcal{R}}_+^{(1)}(\Phi)| - 2 r_1) + 2 | \overline{\mathcal{R}}_+^{(1)}(\Phi) | + 2 r_1.
\notag
\end{eqnarray}

We now estimate the $h$-invariant of the degree $\ell$ polynomials of the system ~(\ref{system 1 in claim 2}) for each $2 \leq \ell \leq d-1$.
The degree $\ell$ portion of the degree $\ell$ polynomials of the system (\ref{system 1 in claim 2}) are precisely
$Q_{\ell,r}(\mathbf{y}, \mathbf{z}) - Q_{\ell,r}(\mathbf{y}', \mathbf{z})$, $Q_{\ell,r}(\mathbf{v}, \mathbf{z}) - Q_{\ell, r}(\mathbf{v}', \mathbf{z})$
$(1 \leq r \leq r_{\ell})$, and the forms of
$\overline{\mathcal{R}}_+^{(\ell)}(\Phi) (\mathbf{y}, \mathbf{z}) - \overline{\mathcal{R}}_+^{(\ell)}(\Phi) (\mathbf{y}', \mathbf{z})$
and $\overline{\mathcal{R}}_+^{(\ell)}(\Phi) (\mathbf{v}, \mathbf{z}) - \overline{\mathcal{R}}_+^{(\ell)}(\Phi) (\mathbf{v}', \mathbf{z})$. We let
$h_{\ell}$ be the $h$-invariant of these degree $\ell$ forms.
Suppose for some $\boldsymbol{\lambda}, \boldsymbol{\mu} \in \mathbb{Q}^{r_{\ell}}$ and $\boldsymbol{\gamma},\boldsymbol{\gamma}'  \in \mathbb{Q}^{|\overline{\mathcal{R}}^{(\ell)}(\Phi)|}$, not all zero vectors, we have
\begin{eqnarray}
\label{eqn 2''}
&&\sum_{r=1}^{r_{\ell}} \lambda_r \cdot (Q_{\ell,r}(\mathbf{y}, \mathbf{z}) - Q_{\ell, r}(\mathbf{y}', \mathbf{z}) )
+ \mu_r \cdot (Q_{\ell,r}(\mathbf{v}, \mathbf{z}) - Q_{\ell,r}(\mathbf{v}', \mathbf{z}) )
\\
&+&
 \sum_{1 \leq j \leq |\overline{\mathcal{R}}^{(\ell)}(\Phi) |} \gamma_j ( {A}_j^{(\ell)}(\mathbf{y}, \mathbf{z}) - {A}_j^{(\ell)}(\mathbf{y}', \mathbf{z}) )
+  \gamma'_j ( {A}_j^{(\ell)}(\mathbf{v}, \mathbf{z}) - {A}_j^{(\ell)}(\mathbf{v}', \mathbf{z}) )
= \sum_{j=1}^{h_{\ell}} \widetilde{U}_j \cdot \widetilde{V}_j,
\notag
\end{eqnarray}
where $\widetilde{U}_j = \widetilde{U}_j( \mathbf{y}, \mathbf{y}', \mathbf{v}, \mathbf{v}', \mathbf{z} )$ and $\widetilde{V}_j =
\widetilde{V}_j( \mathbf{y}, \mathbf{y}', \mathbf{v}, \mathbf{v}', \mathbf{z} )$ are rational forms of positive degree $(1 \leq j \leq h_{\ell})$.
We must consider two cases, $\boldsymbol{\gamma} = \boldsymbol{\gamma}' = \mathbf{0}$ and at least one of $\boldsymbol{\gamma}$
and $\boldsymbol{\gamma}'$ not being a zero vector.
If $\boldsymbol{\gamma} = \boldsymbol{\gamma}' =  \mathbf{0}$, then at least one of $\boldsymbol{\lambda}$ or $\boldsymbol{\mu}$
is not a zero vector. Without loss of generality, suppose $\boldsymbol{\lambda} \not = \mathbf{0}$.
Then by setting $\mathbf{v} = \mathbf{v}' = \mathbf{y}' =  \mathbf{0}$, we have
$$
\sum_{r=1}^{r_{\ell}} \lambda_r Q_{\ell,r}(\mathbf{y}, \mathbf{z})
= \sum_{j=1}^{h_{\ell}} \widetilde{U}_j (\mathbf{y}, \mathbf{0}, \mathbf{0}, \mathbf{0}, \mathbf{z})  \cdot \widetilde{V}_j (\mathbf{y}, \mathbf{0}, \mathbf{0}, \mathbf{0}, \mathbf{z} ).
$$
Consequently, we obtain from (\ref{hinv bound1}) and (\ref{hinv bound1'}) that
\begin{eqnarray}
h_{\ell} \geq h_{\ell}( \{ Q_{\ell,r}(\mathbf{y}, \mathbf{z}) : 1 \leq r \leq r_{\ell} \}  )
\geq \rho_{d,  \ell } ( 2R + 2 R_2 ) + 2 R_2 + 4 r_1.
\notag
\end{eqnarray}

For the second case, suppose without loss of generality that $\boldsymbol{\gamma} \not = \mathbf{0}$.
First we set $\mathbf{v} = \mathbf{v}' = \mathbf{0}$ and simplify the equation (\ref{eqn 2''}) to
\begin{eqnarray}
\label{eqn 1'}
&&\sum_{r=1}^{r_{\ell}} \lambda_r \cdot (Q_{\ell,r}(\mathbf{y}, \mathbf{z}) - Q_{\ell,r}(\mathbf{y}', \mathbf{z}) )
+ \sum_{1 \leq j \leq |\overline{\mathcal{R}}^{(\ell)}(\Phi) |} \gamma_j ( {A}_j^{(\ell)}(\mathbf{y}, \mathbf{z}) - {A}_j^{(\ell)}(\mathbf{y}', \mathbf{z}) )
\\
&=& \sum_{j=1}^{h_{\ell}} \widetilde{U}_j (\mathbf{y}, \mathbf{y}', \mathbf{0}, \mathbf{0}, \mathbf{z}) \cdot \widetilde{V}_j (\mathbf{y}, \mathbf{y}', \mathbf{0}, \mathbf{0}, \mathbf{z}).
\notag
\end{eqnarray}
Recall every monomial of $Q_{\ell,r}(\mathbf{y}, \mathbf{z})$ contains at least one of the $\mathbf{y}$ variables.
Thus it follows from the definition of the $h$-invariant, (\ref{bound on M}), and (\ref{bound on M1}) that
\begin{eqnarray}
&&h_{\ell}( Q_{\ell,r}(\mathbf{y}, \mathbf{z}) )
\notag
\\
&\leq& M
\notag
\\
&\leq& d R ( R^{2} + 1 )^{d-2} 2^d \Big{(} \rho_{d, d} (2 R + 2 R_2) + 2 R_2 + 4 r_1 +  2 R C''_0 \Big{)} +
d  R^3 ( R^{2} + 1 )^{d - 2}  (2 R_2 + 1).
\notag
\end{eqnarray}
Therefore, by moving the term $\sum_{r=1}^{r_{\ell}} \lambda_r \cdot (Q_{\ell,r}(\mathbf{y}, \mathbf{z}) - Q_{\ell,r}(\mathbf{y}', \mathbf{z}) )$ to the right hand side of the equation (\ref{eqn 1'}), we obtain via Lemma \ref{Lemma 2 in CM} and $(4)$ of Proposition \ref{prop reg par} that
\begin{eqnarray}
&&
\\
&&
h_{\ell} + 2 r_{\ell} M
\notag
\\
&\geq&
h_{\ell} \left( \overline{\mathcal{R}}_+^{(\ell)}(\Phi) (\mathbf{y}, \mathbf{z}) - \overline{\mathcal{R}}_+^{(\ell)}(\Phi) (\mathbf{y}', \mathbf{z}) \right)
\notag
\\
&\geq&
h_{\ell} \left( \overline{\mathcal{R}}_+^{(\ell)}(\Phi) (\mathbf{y}, \mathbf{z}) - \overline{\mathcal{R}}_+^{(\ell)}(\Phi) (\mathbf{y}', \mathbf{z}) ; \mathbf{z}  \right)
\notag
\\
&\geq&
h_{\ell} \left( \overline{\mathcal{R}}_+^{(\ell)}(\Phi) (\mathbf{y}, \mathbf{z}) , \overline{\mathcal{R}}_+^{(\ell)}(\Phi) (\mathbf{y}', \mathbf{z}) ; \mathbf{z}  \right)
\notag
\\
&=&
h_{\ell} \left( \overline{\mathcal{R}}^{(\ell)}(\Phi) (\mathbf{y}, \mathbf{z}) , \overline{\mathcal{R}}^{(\ell)}(\Phi) (\mathbf{y}', \mathbf{z}) ; \mathbf{z}  \right)
\notag
\\
&=&
h_{\ell} \left( \overline{\mathcal{R}}^{(\ell)}(\Phi) (\mathbf{y}, \mathbf{z}) ; \mathbf{z}  \right)
\notag
\\
&\geq&
\mathcal{F}_{\ell}( R_2)
\notag
\\
&=&
\rho_{d,  d }( 2R + 2 R_2 ) + 2 R_2 + 4 r_1
\notag
\\
&+& 2 R \left( d R ( R^{2} + 1 )^{d-2} 2^d \Big{(} \rho_{d, d} (2 R + 2 R_2) + 2 R_2 + 4 r_1 +  2 R C''_0 \Big{)} +
d  R^3 ( R^{2} + 1 )^{d - 2}  (2 R_2 + 1) \right).
\notag
\end{eqnarray}
Thus it follows that
\begin{eqnarray}
h_{\ell}
\geq
\rho_{d, d}(2R + 2 R_2 ) + 2 R_2 + 4 r_1
\geq
\rho_{d, \ell}(2R + 2 R_2 ) + 2 R_2 + 4 r_1.
\notag
\end{eqnarray}
Therefore, in either case we obtain
$$
h_{\ell} \geq \rho_{d, \ell}(2R + 2 R_2 ) + 2 R_2 + 4 r_1 \geq \rho_{d,\ell}( 2R + 2 R'_2 - 2 | \overline{\mathcal{R}}_+^{(1)}(\Phi) | - 2 r_1)
+ 2 | \overline{\mathcal{R}}_+^{(1)}(\Phi) | + 2 r_1.
$$

Finally, we also have to show that the linear forms of the system (\ref{system 1 in claim 2}) are linearly independent over $\mathbb{Q}$.
Recall the linear forms of
$$
\Big{\{}  \sum_{i=1}^{r_1 }{g}_{r,i} \widetilde{y}_i : 1 \leq r \leq r_1 \Big{\}} \bigcup \  \overline{\mathcal{R}}_+^{(1)}(\Phi) (\mathbf{y}, \mathbf{z})
$$ are linearly independent over $\mathbb{Q}$,
and by construction they are only in the $\mathbf{y}$ variables.
It is then a basic exercise in linear algebra to verify that the linear forms of
\begin{eqnarray}
&&\overline{\mathcal{R}}_+^{(1)}(\Phi) (\mathbf{y}, \mathbf{z}) - \overline{\mathcal{R}}_+^{(1)}(\Phi) (\mathbf{y}', \mathbf{z}) \
\bigcup
\ \overline{\mathcal{R}}_+^{(1)}(\Phi) (\mathbf{v}, \mathbf{z}) - \overline{\mathcal{R}}_+^{(1)}(\Phi) (\mathbf{v}', \mathbf{z})
\\
\notag
&\bigcup& \Big{\{}  \sum_{i=1}^{r_1 }{g}_{r,i} \widetilde{y}_i -  \sum_{i=1}^{r_1 }{g}_{r,i} \widetilde{y}'_i : 1 \leq r \leq r_1  \Big{\}}
\bigcup \Big{\{}  \sum_{i=1}^{r_1 }{g}_{r,i} \widetilde{v}_i -  \sum_{i=1}^{r_1 }{g}_{r,i} \widetilde{v}'_i  : 1 \leq r \leq r_1 \Big{\}}
\end{eqnarray}
are linearly independent over $\mathbb{Q}$.

Therefore, we obtain from Corollary \ref{cor Schmidt} that
$$
\mathcal{W}_1 \ll X^{n + 3M - K - 2 \sum_{\ell = 1}^d \ell r_{\ell} - 2 D'_2}.
$$

\subsection{Estimate for $\mathcal{W}_2$}
\label{estimate for W_2} We now estimate $\mathcal{W}_2$,
which we can deduce to be the number of solutions $\mathbf{z}, \mathbf{z}', \mathbf{z}'' \in [0,X]^{n-M-K } \cap \mathbb{Z}^{n-M-K}$
satisfying the equations
\begin{eqnarray}
\label{system 2 in claim 2}
{F}_{d,r} ( \mathbf{0}, \mathbf{0},  \mathbf{z} ) - {F}_{d,r} ( \mathbf{0}, \mathbf{0},  \mathbf{z}' ) &=& 0 \ \ ( 1 \leq r \leq r_{d})
\\
{F}_{d,r}( \mathbf{0}, \mathbf{0},  \mathbf{z} ) - {F}_{d,r} ( \mathbf{0}, \mathbf{0},  \mathbf{z}'' )  &=& 0 \ \ ( 1 \leq r \leq r_{d})
\notag
\\
\notag
{P}_{\ell,r}( \mathbf{a}, \mathbf{b} ) - {P}_{\ell,r} ( \mathbf{a}', \mathbf{b}') &=& 0 \ \ (2 \leq \ell < d, 1 \leq r \leq r_{\ell})
\\
{P}_{\ell,r}(\mathbf{a}, \mathbf{b}) - {P}_{\ell,r} ( \mathbf{a}'', \mathbf{b}'' )  &=& 0 \ \ (2 \leq \ell < d, 1 \leq r \leq r_{\ell})
\notag
\\
\sum_{i=1}^{r_1 }{g}'_{r,i} \widetilde{z}_i -  \sum_{i=1}^{r_1 }{g}'_{r,i} \widetilde{z}'_i &=& 0 \ \ (1 \leq r \leq r_1)
\notag
\\
\sum_{i=1}^{r_1 }{g}'_{r,i} \widetilde{z}_i -  \sum_{i=1}^{r_1 }{g}'_{r,i} \widetilde{z}''_i &=& 0 \ \ (1 \leq r \leq r_1)
\notag
\\
\mathcal{R}_+(\Psi) ( \mathbf{z}) - \mathcal{R}_+(\Psi) (\mathbf{z}') &=& \mathbf{0}
\notag
\\
\mathcal{R}_+(\Psi) ( \mathbf{z}) - \mathcal{R}_+(\Psi) (\mathbf{z}'') &=& \mathbf{0},
\notag
\end{eqnarray}
where $\mathbf{z} = (\mathbf{a}, \mathbf{b}), \mathbf{z}' = (\mathbf{a}', \mathbf{b}')$, and $\mathbf{z}'' = (\mathbf{a}'', \mathbf{b}'')$.
Here the notations $\mathcal{R}_+(\Psi) ( \mathbf{z}) - \mathcal{R}_+(\Psi) (\mathbf{z}')$ and $\mathcal{R}_+(\Psi) ( \mathbf{z}) - \mathcal{R}_+(\Psi) (\mathbf{z}'')$ should be interpreted in a similar manner as in Section \ref{estimate for W_1}.

We consider the $h$-invariant of the system of forms on the left hand side of ~(\ref{system 2 in claim 2}), and show that
it is a regular system. The degree $d$ forms of the system ~(\ref{system 2 in claim 2})
are precisely $ {F}_{d,r} ( \mathbf{0}, \mathbf{0},  \mathbf{z} ) - {F}_{d,r} ( \mathbf{0}, \mathbf{0},  \mathbf{z}' )$
and ${F}_{d,r}( \mathbf{0}, \mathbf{0},  \mathbf{z} ) - {F}_{d,r} ( \mathbf{0}, \mathbf{0},  \mathbf{z}'' )$ $(1 \leq r \leq r_{d})$,
and we let $h_d$ be the $h$-invariant of these degree $d$ forms.
Suppose for some $\boldsymbol{\lambda}, \boldsymbol{\mu} \in \mathbb{Q}^{r_d}$, not both $\mathbf{0}$, we have
\begin{eqnarray}
\label{eqn f_M U V}
&&\sum_{r=1}^{r_d} \lambda_r \cdot ( {F}_{d,r}( \mathbf{0}, \mathbf{0},  \mathbf{z} ) - {F}_{d,r}( \mathbf{0}, \mathbf{0},  \mathbf{z}' ) )
+ \sum_{r=1}^{r_d}  \mu_r \cdot ( {F}_{d,r}( \mathbf{0}, \mathbf{0},  \mathbf{z} ) - {F}_{d,r}( \mathbf{0}, \mathbf{0},  \mathbf{z}'' )  )
\\
&=& \sum_{j=1}^{h_d} \widetilde{U}_j \cdot \widetilde{V}_j,
\notag
\end{eqnarray}
where $\widetilde{U}_j = \widetilde{U}_j( \mathbf{z}, \mathbf{z}', \mathbf{z}'' )$ and $\widetilde{V}_j =
\widetilde{V}_j( \mathbf{z}, \mathbf{z}', \mathbf{z}'' )$ are rational forms of positive degree $(1 \leq j \leq h_d)$.
We consider two cases, $(\boldsymbol{\lambda} + \boldsymbol{\mu} ) \not = \mathbf{0}$ and $( \boldsymbol{\lambda} + \boldsymbol{\mu} ) = \mathbf{0}$.
Suppose $(\boldsymbol{\lambda} + \boldsymbol{\mu} ) \not = \mathbf{0}$. If we set $\mathbf{z}' = \mathbf{z}'' = \mathbf{0}$, then the above equation
(\ref{eqn f_M U V}) becomes
\begin{eqnarray}
\sum_{r=1}^{r_d} ({\lambda}_r + {\mu}_r ) \cdot  F_{d,r}( \mathbf{0}, \mathbf{0},  \mathbf{z} )
=  \sum_{j=1}^{h_d} \widetilde{U}_j (\mathbf{z}, \mathbf{0}, \mathbf{0} ) \cdot \widetilde{V}_j (\mathbf{z}, \mathbf{0}, \mathbf{0} ).
\notag
\end{eqnarray}
Thus we obtain
$$
h_d \geq h_d( \mathbf{F}_{d}(\mathbf{0}, \mathbf{0}, \mathbf{z}) ).
$$
On the other hand, suppose $( \boldsymbol{\lambda} + \boldsymbol{\mu} ) = \mathbf{0}$, then the above equation ~(\ref{eqn f_M U V}) simplifies to
$$
- \sum_{r=1}^{r_d} \lambda_r \cdot ( {F}_{d,r} ( \mathbf{0}, \mathbf{0},  \mathbf{z}' ) - {F}_{d,r} ( \mathbf{0}, \mathbf{0},  \mathbf{z}'' ) )
= \sum_{j=1}^{h_d} \widetilde{U}_j \cdot \widetilde{V}_j.
$$
From this equation, 
by setting $\mathbf{z}'' = \mathbf{0}$ we obtain
$$
h_d \geq h_d( \mathbf{F}_{d}(\mathbf{0}, \mathbf{0}, \mathbf{z}) ).
$$
Therefore, in either case we obtain from (\ref{hinv bound3}), (\ref{def Fd z}), and (\ref{hinv bound3'}) that
$$
h_d \geq h_d( \mathbf{F}_{d}(\mathbf{0}, \mathbf{0}, \mathbf{z}) )  \geq \rho_{d,d}( 2R + 2  R_1) + 2 R_1 + 4 r_1 \geq
\rho_{d,d}( 2R + 2  R'_1 - 2 | \mathcal{R}^{(1)}_+ (\Psi) | - 2 r_1 ) + 2 |  \mathcal{R}^{(1)}_+ (\Psi) | + 2 r_1.
$$

We now estimate the $h$-invariant of the degree $\ell$ forms of the system (\ref{system 2 in claim 2}) for each $2 \leq \ell < d$.
Recall we have $\mathcal{R}_+ (\Psi) = ( \mathcal{R}^{(d-1)}(\Psi), \ldots,  \mathcal{R}^{(2)}(\Psi),  \mathcal{R}^{(1)}_+ (\Psi))$.
The degree $\ell$ forms of the system (\ref{system 2 in claim 2}) are precisely
$P_{\ell,r}(\mathbf{a}, \mathbf{b}) - P_{\ell,r}(\mathbf{a}', \mathbf{b}')$, $P_{\ell,r}(\mathbf{a}, \mathbf{b}) - P_{\ell, r}(\mathbf{a}'', \mathbf{b}'')$
$(1 \leq r \leq r_{\ell})$, and the forms of
$\mathcal{R}^{(\ell)}(\Psi) (\mathbf{z}) - \mathcal{R}^{(\ell)}(\Psi) (\mathbf{z}')$
and $\mathcal{R}^{(\ell)}(\Psi) (\mathbf{z}) - \mathcal{R}^{(\ell)}(\Psi) (\mathbf{z}'')$. We let
$h_{\ell}$ be the $h$-invariant of these degree $\ell$ forms.
Then for some $\boldsymbol{\lambda}, \boldsymbol{\mu} \in \mathbb{Q}^{r_{\ell}}$ and $\boldsymbol{\gamma}, \boldsymbol{\gamma}' \in \mathbb{Q}^{| {\mathcal{R}}^{(\ell)} (\Psi) |}$, not all zero vectors, we have
\begin{eqnarray}
\label{eqn f_M U V 1}
&&\sum_{r = 1}^{r_{\ell}} \lambda_{r} (P_{\ell,r} (\mathbf{a}, \mathbf{b}) - P_{\ell,r} (\mathbf{a}', \mathbf{b}' )
+  \mu_{r} (P_{\ell,r} (\mathbf{a}, \mathbf{b}) - P_{\ell,r} (\mathbf{a}'', \mathbf{b}'') )
\\ \notag
&+&\sum_{j=1}^{| {\mathcal{R}}^{(\ell)} (\Psi)  |} \gamma_j \cdot ( V^{(\ell)}_j (\mathbf{z}) - V^{(\ell)}_j (\mathbf{z}' ) ) +  \gamma'_j \cdot ( V^{(\ell)}_j (\mathbf{z}) - V^{(\ell)}_j (\mathbf{z}'' ) )= \sum_{j=1}^{h_{\ell}} \widetilde{U}_j \cdot \widetilde{V}_j,
\end{eqnarray}
where $\widetilde{U}_j = \widetilde{U}_j( \mathbf{z}, \mathbf{z}', \mathbf{z}'')$ and $\widetilde{V}_j =
\widetilde{V}_j( \mathbf{z}, \mathbf{z}', \mathbf{z}'')$ are rational forms of positive degree $(1 \leq j \leq h_{\ell})$.
We consider two cases, $\boldsymbol{\gamma} = \boldsymbol{\gamma}' = \mathbf{0}$ and
at least one of $\boldsymbol{\gamma} $ and $ \boldsymbol{\gamma}'$ is not a zero vector.

First we suppose that $\boldsymbol{\gamma} = \boldsymbol{\gamma}' = \mathbf{0}$. In this case, at leat one of $\boldsymbol{\lambda}$
and $\boldsymbol{\mu}$ is not a zero vector. Without loss of generality, suppose $\boldsymbol{\lambda} \not = \mathbf{0}$.
Then by setting $\mathbf{z} = \mathbf{z}''$ and $\mathbf{z}' = \mathbf{0}$,
the equation (\ref{eqn f_M U V 1}) becomes
$$
\sum_{1 \leq r \leq r_{\ell}} \lambda_{r} P_{\ell,r} (\mathbf{a}, \mathbf{b})
= \sum_{j=1}^{h_{\ell}} \widetilde{U}_j (\mathbf{z}, \mathbf{0}, \mathbf{z}) \cdot \widetilde{V}_j(\mathbf{z}, \mathbf{0}, \mathbf{z}).
$$
Therefore, it follows from (\ref{hinv bound2}) and (\ref{hinv bound2'}) that
\begin{eqnarray}
\notag
h_{\ell} &\geq& h_{\ell} (  \{ P_{\ell,r} (\mathbf{a}, \mathbf{b}) : 1 \leq r \leq r_{\ell} \} )
\\
&\geq& \rho_{d,\ell}( 2R + 2  R_1) + 2 R_1 + 4 r_1
\notag
\\
&\geq& \rho_{d,\ell}(2R + 2R'_1 - 2|\mathcal{R}_+^{(1)}(\Psi)|- 2 r_1) + 2|\mathcal{R}_+^{(1)}(\Psi)|+ 2 r_1.
\notag
\end{eqnarray}

Next we suppose at least one of $\boldsymbol{\gamma} $ and $ \boldsymbol{\gamma}'$ is not a zero vector. Without loss
of generality, suppose $\boldsymbol{\gamma} \not = \mathbf{0}$. We consider two further subcases,
$(\boldsymbol{\gamma} + \boldsymbol{\gamma}') \not = \mathbf{0}$ and $(\boldsymbol{\gamma} + \boldsymbol{\gamma}')  = \mathbf{0}$.

Suppose $(\boldsymbol{\gamma} + \boldsymbol{\gamma}') \not = \mathbf{0}$. In this case, we set $\mathbf{z}' = \mathbf{z}'' = \mathbf{0}$, and the equation ~(\ref{eqn f_M U V 1}) simplifies to
\begin{eqnarray}
\label{eqn 2}
\\
\notag
\sum_{1 \leq r \leq r_{\ell}} (\lambda_{r} +  \mu_{r}) P_{\ell,r} (\mathbf{a}, \mathbf{b})
 +
\sum_{j=1}^{|\mathcal{R}^{(\ell)}(\Psi)| } ( \gamma_j + \gamma'_j ) \cdot  V^{(\ell)}_j (\mathbf{z}) = \sum_{j=1}^{h_{\ell}} \widetilde{U}_j( \mathbf{z}, \mathbf{0}, \mathbf{0}) \cdot  \widetilde{V}_j ( \mathbf{z}, \mathbf{0}, \mathbf{0}).
\end{eqnarray}
Recall every monomial of $P_{\ell,r} (\mathbf{a}, \mathbf{b})$ contains at least one of the $\mathbf{a}$ variables.
Thus it follows from the definition of the $h$-invariant, (\ref{bound on M'}), and (\ref{bound on M'1}) that
$$
h_{\ell}(P_{\ell,r} (\mathbf{a}, \mathbf{b})) \leq M' \leq d R ( R^{2} + 1 )^{d-2} 2^d \Big{(} \rho_{d, d} (2 R + 2 R_1) + 2 R_1 + 4 r_1 +  2 R C''_0 \Big{)}.
$$
Therefore, by moving the term $
\sum_{1 \leq r \leq r_{\ell}} (\lambda_{r} +  \mu_{r}) P_{\ell,r} (\mathbf{a}, \mathbf{b})$
to the right hand side of the equation (\ref{eqn 2}),
we obtain via $(3)$ of Proposition \ref{prop reg par} that
\begin{eqnarray}
h_{\ell} + M'r_{\ell} &\geq&  h_{\ell}( \mathcal{R}^{(\ell)}(\Psi) )
\notag
\\
&\geq& \mathcal{F}'_{\ell}(R_1)
\notag
\\
&=& \rho_{d,d}(2R + 2R_1) + 2 R_1 + 4 r_1
\notag
\\
&+& 2 R \Big{(} d R ( R^{2} + 1 )^{d-2} 2^d \Big{(} \rho_{d, d} (2 R + 2 R_1) + 2 R_1 + 4 r_1 +  2 R C''_0 \Big{)} \Big{)}.
\notag
\end{eqnarray}
Thus we obtain
\begin{eqnarray}
\\
\notag
h_{\ell} \geq  \rho_{d,d}(2R + 2R_1) + 2 R_1 + 4 r_1 \geq \rho_{d,\ell}(2R + 2R'_1 - 2|\mathcal{R}_+^{(1)}(\Psi)|- 2 r_1) + 2|\mathcal{R}_+^{(1)}(\Psi)|+ 2 r_1.
\end{eqnarray}

On the other hand, we now suppose $(\boldsymbol{\gamma} + \boldsymbol{\gamma}') = \mathbf{0}$.
By setting $\mathbf{z} =  \mathbf{z}'' = \mathbf{0}$, the equation ~(\ref{eqn f_M U V 1}) simplifies to
\begin{eqnarray}
- \sum_{r = 1}^{r_{\ell}} \lambda_r \cdot P_{\ell,r} (\mathbf{a}', \mathbf{b}' )
- \sum_{j=1}^{|\mathcal{R}^{(\ell)}(\Psi)|}  \gamma_j \cdot    V^{(\ell)}_j (\mathbf{z}')  = \sum_{j=1}^{h_{\ell}}
\widetilde{U}_j (\mathbf{0}, \mathbf{z}', \mathbf{0}) \cdot \widetilde{V}_j (\mathbf{0}, \mathbf{z}', \mathbf{0}).
\notag
\end{eqnarray}
Then by a similar argument as above, we have
\begin{eqnarray}
h_{\ell} + M' r_{\ell} &\geq& h_{\ell}(\mathcal{R}^{(\ell)}(\Psi) )
\notag
\\
&\geq& \mathcal{F}'_{\ell}(R_1)
\notag
\\
&=& \rho_{d,d}(2R + 2R_1) + 2 R_1 + 4 r_1
\notag
\\
&+&2 R \Big{(} d R ( R^{2} + 1 )^{d-2} 2^d \Big{(} \rho_{d, d} (2 R + 2 R_1) + 2 R_1 + 4 r_1 +  2 R C''_0 \Big{)} \Big{)}.
\notag
\end{eqnarray}
Therefore, we also obtain
\begin{eqnarray}
\\
\notag
h_{\ell} &\geq& \rho_{d,d}(2R + 2R_1) + 2 R_1 + 4 r_1  \geq \rho_{d,\ell}(2R + 2R'_1 - 2|\mathcal{R}_+^{(1)}(\Psi)|- 2 r_1) + 2|\mathcal{R}_+^{(1)}(\Psi)|+ 2 r_1
\end{eqnarray}
in this case.

We also have to show that the linear forms of the system ~(\ref{system 2 in claim 2}),
\begin{eqnarray}
\label{system of linear forms2}
&&\{ \mathcal{R}_+^{(1)}(\Psi)(\mathbf{z}) - \mathcal{R}_+^{(1)}(\Psi)(\mathbf{z}') \} \bigcup \  \{ \mathcal{R}_+^{(1)}(\Psi)(\mathbf{z}) - \mathcal{R}_+^{(1)}(\Psi)(\mathbf{z}'') \}
\\
\notag
&\bigcup& \Big{\{}  \sum_{i=1}^{r_1 }{g}'_{r,i} \widetilde{z}_i -  \sum_{i=1}^{r_1 }{g}'_{r,i} \widetilde{z}'_i : 1 \leq r \leq r_1  \Big{\}}
\bigcup \ \Big{\{}  \sum_{i=1}^{r_1 }{g}'_{r,i} \widetilde{z}_i -  \sum_{i=1}^{r_1 }{g}'_{r,i} \widetilde{z}''_i  : 1 \leq r \leq r_1  \Big{\}},
\end{eqnarray}
are linearly independent over $\mathbb{Q}$.
Recall the linear forms of
$$
\mathcal{R}_+^{(1)}(\Psi)(\mathbf{z})  \bigcup \   \Big{\{}  \sum_{i=1}^{r_1 }{g}'_{r,i} \widetilde{z}_i : 1 \leq r \leq r_1  \Big{\}}
$$ are linearly independent over $\mathbb{Q}$.
Using this fact, the verification of linear independence over $\mathbb{Q}$ of the system of linear forms ~(\ref{system of linear forms2}) is a
basic exercise in linear algebra.

Therefore, we obtain by Corollary \ref{cor Schmidt} that
$$
\mathcal{W}_2 \ll X^{3(n -M - K) - 2 \sum_{\ell = 1}^d \ell r_{\ell} - 2 D'_1}.
$$

\subsection{Proof of Claim 1}
\label{sec claim 1}
Recall we defined
\begin{eqnarray}
\notag
S_0( \boldsymbol{\alpha}, \mathbf{G}, \mathbf{H} ) &=&
\sum_{ \mathbf{w} \in [0,X]^K } \Lambda(\mathbf{w})
\ e \Big{(} \sum_{1 \leq r \leq r_1}  \alpha_{1,r} (c_{1,r} \mathbf{w}^{\mathbf{j}_{1, r}} + \widetilde{f}_{1,r} (\mathbf{w}, \mathbf{0}, \mathbf{0})
 )
\\
\label{def S_0'}
&+& \sum_{2 \leq \ell \leq d} \sum_{1 \leq r \leq r_{\ell}} \alpha_{\ell,r}  \cdot \mathfrak{C}_{\ell,r} (\mathbf{w}, \mathbf{G}, \mathbf{H}  )  \Big{)},
\end{eqnarray}
where
$$
\mathfrak{C}_{\ell,r} (\mathbf{w}, \mathbf{G}, \mathbf{H}  ) = {f}_{\ell,r} ( \mathbf{w}, \mathbf{0}, \mathbf{0})
+ \sum_{j=1}^{\ell -1}  \ \sum_{1 \leq i_1 \leq \ldots \leq i_j \leq K}
\left( \sum_{k=1}^{\ell -j} {c}^{(k)}_{\ell, r: i_1, \ldots, i_{j}} ( \mathbf{G}, \mathbf{H} )  \right) w_{i_1} \ldots w_{i_j}.
$$
Also recall we defined the monomials $\mathbf{w}^{\mathbf{j}_{\ell, r}}$ $(1 \leq \ell \leq d, 1 \leq r \leq r_{\ell})$ in (\ref{set of eqn 2}).
If we consider the expression in the exponent of (\ref{def S_0'}),
$$
\sum_{1 \leq r \leq r_1}  \alpha_{1,r} (c_{1,r} \mathbf{w}^{\mathbf{j}_{1, r}} + \widetilde{f}_{1,r} (\mathbf{w}, \mathbf{0}, \mathbf{0}) )
+ \sum_{2 \leq \ell \leq d} \sum_{1 \leq r \leq r_{\ell}} \alpha_{\ell,r}  \cdot \mathfrak{C}_{\ell,r} (\mathbf{w}, \mathbf{G}, \mathbf{H}  ),
$$
as a polynomial in $\mathbf{w}$ with real coefficients, then it follows from the discussion after (\ref{discussion on the coeff})
that the coefficient of $\mathbf{w}^{\mathbf{j}_{\ell, r}}$ of this polynomial is $c_{\ell, r} \alpha_{\ell, r}$.
Furthermore, this polynomial does not contain any monomial divisible by $\mathbf{w}^{\mathbf{j}_{\ell, r} }$ other than itself.

We need to introduce few definitions and lemmas before we can begin with the proof of Claim 1.
Let $1 \leq \ell \leq d$, $q \in \mathbb{N}$, and  $a_{\ell} \in \mathbb{Z} / q \mathbb{Z}$. For $q > 1$ we define
$$
\mathfrak{N}^{(\ell)}_{a_{ \ell}, q} (C_0) =\{ \xi_{\ell} \in [0,1) : | \xi_{\ell} -  a_{\ell}/q | \leq (\log X)^{C_0} X^{-\ell}  \},
$$
and when $q=1$ we let
$$
\mathfrak{N}^{(\ell)}_{0, 1} (C_0) =\{ \xi_{\ell} \in [0,1) :  \min \{ |\xi_{\ell} |,  |\xi_{\ell} - 1 | \} \leq (\log X)^{C_0} X^{-\ell}  \}.
$$
We set
$$
\mathfrak{N}(C_0) = \bigcup_{q \leq (\log X)^{C_0}} \  \bigcup_{ \substack{ \gcd (a_{d}, \ldots, a_{1},q) = 1  \\ a_{d}, \ldots, a_1 \in \mathbb{Z} / q \mathbb{Z}}  }
\mathfrak{N}^{(d)}_{a_{d}, q} (C_0) \times \ldots \times \mathfrak{N}^{(1)}_{a_{1}, q} (C_0),
$$
and denote
$$
\mathfrak{n}(C_0) = [0,1)^d \backslash \mathfrak{N}(C_0).
$$
Let $\mathbb{U}_q$ be the group of units in $\mathbb{Z}/q \mathbb{Z}$. When $q=1$ we let $\mathbb{U}_1 = \{0\}$.
Let us also denote
$$
\mathfrak{n}^{(\ell)}(C_0) = [0,1) \backslash  \left( \bigcup_{q \leq (\log X)^{C_0}} \bigcup_{a_{\ell} \in \mathbb{U}_q }
\mathfrak{N}^{(\ell)}_{a_{\ell}, q} (C_0) \right).
$$
Suppose $\boldsymbol{\xi} = (\xi_d, \ldots, \xi_1) \in [0,1)^d$ satisfies $\xi_{\ell} \in \mathfrak{n}^{(\ell)}(C_0)$ for some $1 \leq \ell \leq d$.
Then it is clear that $\boldsymbol{\xi} \in \mathfrak{n}(C_0)$.

We have the following lemma which is a special case of \cite[Ch.VI, \S 1, Theorem 10]{H}.
\begin{lem}\cite[Ch.VI, \S 1, Theorem 10]{H}
\label{lemma from Hua1}
Let $\ell \geq 1$, $\alpha_{\ell-1}, \ldots , \alpha_1, \alpha_0 \in \mathbb{R}$, and $\gcd(a,q) = 1$ with $(\log X)^{\sigma} < q \leq X^{\ell}(\log X)^{-\sigma}$.
Suppose we have $\sigma_0 > 0$ such that $\sigma \geq 2^{6 \ell} (\sigma_0 + 1)$.
Then we have
$$
\sum_{\substack{p \leq X \\ p \text{ prime}}} e \left( \frac{a}{q} p^{\ell} + \alpha_{\ell - 1}p^{\ell - 1} + \ldots +  \alpha_{1}p + \alpha_0 \right)
\ll \frac{X}{(\log X)^{\sigma_0}},
$$
where the implicit constant depends only on $\ell$.
\end{lem}
From this lemma we can obtain the following, which is essentially a special case of \cite[Ch.X, \S 5, Lemma 10.8]{H}.
\begin{lem}\cite[Ch.X, \S 5, Lemma 10.8]{H}
\label{lemma from Hua}
Suppose $\ell \geq 1$ and $\alpha_{\ell}, \ldots , \alpha_1 \in \mathbb{R}$.
Let
$$
T_1 (\alpha_{\ell}, \ldots , \alpha_1) = \sum_{x \in [0,X]} \Lambda(x) e(\alpha_{\ell} x^{\ell} + \ldots + \alpha_1 x ).
$$
Given any $c_0 > 0$, for sufficiently large $C_0 > 0$ we have
$$
|T_1 (\alpha_{\ell}, \ldots , \alpha_1)| \ll \frac{X}{(\log X)^{c_0} }
$$
for any $\alpha_{\ell}, \ldots , \alpha_1 \in \mathbb{R}$ with $\alpha_{\ell} \in \mathfrak{n}^{(\ell)}(C_0)$.
Here the implicit constant depends only on $\ell$.
\end{lem}
\begin{proof}
By Dirichlet's theorem on diophantine approximation, there exist $a, q \in \mathbb{Z}$ such that
$\gcd(a,q) = 1$, $1 \leq q \leq X^{\ell} (\log X)^{- C_0}$,
and
\begin{equation}
\label{ineq in lemma of Hua}
| q \alpha_{\ell} - a | < \frac{(\log X)^{C_0}}{X^{\ell}}.
\end{equation}
Since we have
\begin{equation}
\Big{|} \alpha_{\ell} - \frac{a}{q} \Big{|} < \frac{(\log X)^{C_0}}{qX^{\ell}} \leq \frac{(\log X)^{C_0}}{X^{\ell}},
\end{equation}
it follows from the definition of $\mathfrak{n}^{(\ell)}(C_0)$ that $q > (\log X)^{C_0}$.
Let $\beta_{\ell} = \alpha_{\ell} - \frac{a}{q}$.
Then we obtain from (\ref{ineq in lemma of Hua}) that
$$
|\beta_{\ell}| = \Big{|} \alpha_{\ell} - \frac{a}{q} \Big{|} < \frac{(\log X)^{C_0}}{qX^{\ell}} \leq \frac{1}{X^{\ell}}.
$$
We now have the set up to apply Lemma \ref{lemma from Hua1}.
Let us define
$$
T_0 (\alpha_{\ell}, \ldots , \alpha_1) = \sum_{\substack{1 \leq p \leq X \\ p \text{ prime }}} e(\alpha_{\ell} p^{\ell} + \ldots + \alpha_1 p ).
$$
By following the argument in the proof of \cite[Ch.X, \S 5, Lemma 10.8]{H},
we obtain that given any $c_0 > 0$, for $C_0 >0$ sufficiently large we have
$$
|T_0(\alpha_{\ell}, \ldots , \alpha_1)| \ll \frac{X}{(\log X)^{c_0} },
$$
where the implicit constant depends only on $\ell$.
From here we obtain via partial summation the required bound on $T_1(\alpha_{\ell}, \ldots , \alpha_1)$.
\end{proof}

Recall $\| \alpha \|$ is the distance from $\alpha \in \mathbb{R}$ to the closest integer.
The following is a special case of \cite[Lemma 14.1]{S}.
\begin{lem}\cite[Lemma 14.1]{S}.
\label{Birch lemma 2.3}
Suppose $\lambda \in \mathbb{R}$, $A > 1$, and $Z > 0$. Let $\mathcal{N}(Z)$ be the number of integers $v$ such that
\begin{eqnarray}
|v| \leq ZA \ \  \  \text{ and } \ \ \   \| \lambda v \| \leq Z A^{-1}.
\end{eqnarray}
Then for $0 < Z_1 \leq Z_2 < 1$ we have
$$
\mathcal{N}(Z_1) \gg (Z_1/Z_2) \ \mathcal{N}(Z_2),
$$
where the implicit constant is an absolute constant.
\end{lem}

We now begin with the proof of Claim 1.
Let $M_0$ be the diagonal $R \times R$ matrix where its diagonal entries from the top left corner to the
right bottom corner are $c_{d,1}, c_{d,2},\ldots,$ $c_{d,r_d}, c_{d-1,1},$ $c_{d-1,2},\ldots, c_{d-1,r_{d-1}},$ $\ldots ,$
$c_{1,1}, c_{1,2},\ldots, c_{1,r_1}$ in this order. Clearly $M_0$ is an invertible matrix.
Let $\gamma_{\ell,r} = \alpha_{\ell, r} c_{\ell, r}$.
Consider the polynomial in the exponent of (\ref{def S_0'}) as a polynomial in the $\mathbf{w}$ variables.
Then we know that the coefficient of $\mathbf{w}^{\mathbf{j}_{\ell, r} }$ of this polynomial is $\gamma_{\ell,r}$.
We also have
$$
M_0 \cdot \left[ {\begin{array}{c}
\alpha_{d,1} \\ \vdots  \\  \alpha_{1, r_1}
\end{array}} \right] =  \left[ {\begin{array}{c}
\gamma_{d,1} \\ \vdots  \\  \gamma_{1, r_1}
\end{array}} \right] \in \mathbb{R}^{R }.
$$
Suppose $\boldsymbol{\gamma} \in \mathfrak{M}^{}(C')$ for some $C'>0$, then there exist $\mathbf{a} \in \mathbb{Z}^R$ and $q \in \mathbb{N}$ such
that $\gcd(\mathbf{a}, q) = 1$, $0< q \leq (\log X)^{C'}$, and $| \gamma_{\ell,r} - a_{\ell,r} /q | \leq (\log X)^{C'}/ X^{\ell}$ $(1 \leq \ell \leq d, 1 \leq r \leq r_{\ell})$.
Let us denote
$$
M_0^{-1} \cdot \left[ {\begin{array}{c}
a_{d,1} /q \\ \vdots  \\  a_{1,r_{1}} /q
\end{array}} \right]  =  \left[ {\begin{array}{c}
a'_{d,1} /q' \\ \vdots  \\  a'_{1,r_{1} } /q'
\end{array}} \right]
\ \ \text{ and } \ \
M_0^{-1} \cdot \left[ {\begin{array}{c}
\gamma_{d,1} - a_{d,1}/ q \\ \vdots  \\  \gamma_{1,r_{1}} - a_{1,r_{1}}/q
\end{array}} \right]  =  \left[ {\begin{array}{c}
\beta'_{d,1} \\ \vdots  \\  \beta'_{1, r_{1}}
\end{array}} \right].
$$
It is easy to deduce that
$$
q' \leq (\log X)^{ C' + 1} \ \ \text{ and } \ \ |\beta'_{\ell,r}| \leq  \frac{(\log X)^{C' + 1}}{X^{\ell}} \ \ (1 \leq \ell \leq d, 1 \leq r \leq r_{\ell})
$$
for $X$ sufficiently large with respect to $c_{d,1}, \ldots,  c_{1,r_1}$.
Since $\alpha_{\ell,r} = \frac{a'_{\ell,r} }{q'} + \beta'_{\ell,r}$, we see that $ \boldsymbol{\alpha} \in \mathfrak{M}(C' + 1)$.
Now since $\boldsymbol{\alpha} \in \mathfrak{m}(C)$, it follows from this argument that $\boldsymbol{\gamma} \in \mathfrak{m}(C-1)$.
Then there exist $\ell$ and $r$ such that $\gamma_{\ell, r} \in \mathfrak{n}^{(\ell)}( C'')$, where $C'' = (C-1)/R$, by the following reason.
Suppose $\gamma_{\ell, r} \not \in \mathfrak{n}^{(\ell)}( C'')$ $(1 \leq \ell \leq d, 1 \leq r \leq r_{\ell})$.
Then for each $1 \leq \ell \leq d, 1 \leq r \leq r_{\ell}$ there exist $q_{\ell, r} \in \mathbb{N}$ and $a_{\ell, r} \in \mathbb{Z}$ such that
$$
q_{\ell,r} \leq (\log X)^{ C''} \ \ \text{ and } \ \ |\gamma_{\ell,r} - a_{\ell, r}/q_{\ell, r} | \leq  \frac{(\log X)^{C''}}{X^{\ell}}.
$$
By taking $q$ to be the appropriate factor of the lowest common multiple of
$q_{d,1}, \ldots,  q_{1,r_1}$, we see that $\boldsymbol{\gamma} \in \mathfrak{M}(C-1)$, which is a contradiction.

Throughout the remainder of this section, we fix $\ell$ and $r$ to be such that $\gamma_{\ell, r} \in \mathfrak{n}^{(\ell)}( C'')$.
Following \cite{CM}, we consider two cases depending on $\mathbf{w}^{ \mathbf{j}_{\ell,r} }$:
Case 1 is when $\mathbf{w}^{ \mathbf{j}_{\ell,r} }$ contains only one distinct variable, and Case 2 is when it has more than
one distinct variable.

Case 1: Without loss of generality, suppose $\mathbf{w}^{ \mathbf{j}_{\ell,r} } = w_{1}^{\ell}$.
We may bound $S_0(\boldsymbol{\alpha}, \mathbf{G}, \mathbf{H})$ as follows
\begin{eqnarray}
\label{S'0 case 1}
&&S_0( \boldsymbol{\alpha}, \mathbf{G}, \mathbf{H} ) \leq (\log X)^{K-1} \cdot
\\
\notag
&& \sum_{ w_K \in [0,X] } \ldots \sum_{ w_2 \in [0,X] } \Big{|} \sum_{ w_{1}  \in [0,X] } \Lambda (w_1)
\ e  \left( \gamma_{\ell,r} w_{1}^{\ell} + \tau (w_1, w_2, \ldots, w_K,  \mathbf{G}, \mathbf{H})  \right)  \Big{|},
\notag
\end{eqnarray}
where $\tau (w_1, w_2, \ldots, w_K,  \mathbf{G}, \mathbf{H})$ has degree strictly less than $\ell$ as a polynomial in $w_1$
with coefficients possibly dependent on $w_2, \ldots, w_K,  \mathbf{G}, \mathbf{H}$.
This follows from the fact that the coefficient of $w_{1}^{\ell}$ of the polynomial in the exponent of (\ref{def S_0'}) is $\gamma_{\ell,r}$, and that there are no other monomials divisible by $w_{1}^{\ell}$.

Therefore, since $\gamma_{\ell, r} \in \mathfrak{n}^{(\ell)}( C'' )$ we may apply Lemma \ref{lemma from Hua} with $c_0 = c + K-1$ to the inner sum of (\ref{S'0 case 1}) and obtain
$$
S_0( \boldsymbol{\alpha}, \mathbf{G}, \mathbf{H} ) \ll (\log X)^{K-1} X^{K-1} \frac{X}{(\log X)^{c + K-1}} = \frac{X^K}{(\log X)^c}.
$$

Case 2: We have that $\mathbf{w}^{ \mathbf{j}_{\ell,r} }$ contains at least two distinct variables. In particular, we
must have $\ell > 1$. By relabeling if necessary, let $\mathbf{w}^{ \mathbf{j}_{\ell,r} } = w_{1}^{j_{1}} \ldots w_{k}^{j_{k}}$
where $j_{1},$  \ldots, $j_{k} > 0$.
We know that the coefficient of $\mathbf{w}^{ \mathbf{j}_{\ell,r} }$ of the polynomial in the exponent of (\ref{def S_0'}) is $\gamma_{\ell,r}$.
In this case, we may bound $S_0(\boldsymbol{\alpha}, \mathbf{G}, \mathbf{H})$ as follows
\begin{eqnarray}
\label{S'0 case 2}
|S_0( \boldsymbol{\alpha}, \mathbf{G}, \mathbf{H} )| &\leq& (\log X)^{K-k} \cdot \sum_{ w_K \in [0,X] } \ldots \sum_{ w_{k + 1} \in [0,X] } | S(w_1, \ldots, w_K,  \mathbf{G}, \mathbf{H} )|,
\end{eqnarray}
where
\begin{eqnarray}
S(w_1, \ldots, w_K, \mathbf{G}, \mathbf{H} )
\notag
=
\sum_{ w_{1}  \in [0,X] } \ldots \sum_{ w_{k} \in [0,X] } \Lambda (w_1) \ldots \Lambda(w_k)
\ e  \left( \gamma_{\ell,r} \mathbf{w}^{ \mathbf{j}_{\ell,r} } + \Theta (w_1, \ldots, w_{k})  \right),
\notag
\end{eqnarray}
and $\Theta( w_1, \ldots , w_k )= $ $\Theta (w_1, \ldots, w_{k} : w_{k+1}, \ldots,  w_K,  \mathbf{G}, \mathbf{H})$ is a polynomial in $w_1, \ldots, w_k$
with coefficients possibly dependent on $w_{k+1}, \ldots, w_K,  \mathbf{G}, \mathbf{H}$. By construction,
we also know that this polynomial does not have any monomial divisible by $\mathbf{w}^{ \mathbf{j}_{\ell,r} }$.

We now apply Weyl differencing $\ell$ times, where we apply it $j_{i}$ times to the variable $w_{i}$ for each $1 \leq i \leq k$.
The point is that with this process every monomial of $\gamma_{\ell,r} \mathbf{w}^{ \mathbf{j}_{\ell,r} } +
\Theta( w_1, \ldots , w_k )$ for which at least one of $w_i$ has degree strictly less than $j_i$ will vanish,
in particular every monomial of $\Theta( w_1, \ldots , w_k )$ will vanish.
Let $\widetilde{c} = j_1! \ldots j_k!$. As a result, we obtain
\begin{eqnarray}
\label{S'_0 power}
&&|S(w_1, \ldots, w_K, \mathbf{G}, \mathbf{H} )|^{2^{\ell}}
\\
&\ll&  (\log X)^{ k 2^{\ell} } X^{k  2^{\ell}  - \ell} \sum_{  \substack{v_{i} \in [-X,X] \\ 1 \leq i \leq \ell-1  }}  \min\{ X, \| \widetilde{c}\gamma_{\ell, r} v_{1}  \ldots  v_{\ell - 1} \|^{-1} \}.
\notag
\end{eqnarray}
Since this is a standard application of Weyl differencing, and also similar to the argument in
\cite[pp. 725-726]{CM}, we leave the details to the reader.


Let
$$
\mathcal{A}_X :=\{ (v_1, \ldots, v_{\ell - 1}) \in [-X,X]^{\ell - 1} \cap \mathbb{Z}^{\ell - 1} \}
: \| \widetilde{c} \gamma_{\ell, r} v_1 \ldots v_{\ell -1} \| \leq \frac{1}{ X} \}.
$$
For any $1 \leq X' < X$, we define the set
$$
\mathcal{A}_{X,X'} :=\{ (v_1, \ldots, v_{\ell - 1}) \in [-X/X',X/X']^{\ell - 1} \cap \mathbb{Z}^{\ell - 1} \}
: \| \widetilde{c} \gamma_{\ell, r} v_1 \ldots v_{\ell -1} \| \leq \frac {1}{ X (X')^{\ell - 1}} \}.
$$
By applying Lemma \ref{Birch lemma 2.3} successively in the variables $v_1, \ldots, v_{\ell - 1}$, we obtain
$$
|\mathcal{A}_{X}| \ll (X')^{\ell -1} |\mathcal{A}_{X,X'}|.
$$
Let $X' = X (\log X)^{-C''/d}$. Suppose there exists $(v_1, \ldots, v_{\ell - 1}) \in \mathcal{A}_{X,X'}$
such that $(v_1 \ldots v_{\ell - 1}) \not = 0$. Then we have
$$
| \widetilde{c} v_1 \ldots v_{\ell - 1}| \leq (\log X)^{C''} \ \ \text{ and  } \ \ \| \widetilde{c} \gamma_{\ell, r} v_1 \ldots v_{\ell -1} \| \leq \frac {(\log X)^{C''}}{ X^{\ell}}
$$
for $X$ sufficiently large with respect to $\ell$,
and this contradicts the fact that $\gamma_{\ell, r} \in \mathfrak{n}^{(\ell)}(C'')$. Thus at least one of
$v_1, \ldots, v_{\ell - 1}$ must be $0$. Therefore, we have
$$
|\mathcal{A}_{X,X'}| \ll (\log X)^{(\ell - 2)C'' /d },
$$
and consequently,
\begin{equation}
\label{bound on AX}
|\mathcal{A}_{X}| \ll \left( \frac{X}{(\log X)^{C''/d}} \right)^{\ell -1} |\mathcal{A}_{X,X'}| \ll \frac{X^{\ell - 1}}{(\log X)^{C''/d}}.
\end{equation}

We now proceed in a similar manner as in \cite[Lemma 13.2]{S}.
First let us deal with the case $\ell > 2$.
Let $N_0(v'_1, \ldots, v'_{\ell - 2})$ be the number of points $v_{\ell-1} \in [-X, X] \cap \mathbb{Z}$
such that $( v'_1, \ldots,  v'_{\ell - 2}, v_{\ell-1}  ) \in \mathcal{A}_X$.
Then we have
\begin{equation}
\label{eqn for AX}
|\mathcal{A}_{X}| = \sum_{v_{1} \in [-X, X] } \ldots \sum_{v_{\ell - 2} \in [-X, X] } N_0(v_1, \ldots, v_{\ell-2}),
\end{equation}
and let $N_0 = |\mathcal{A}_{X}|$ when $\ell = 2$.

Let us write $\{ \alpha \}$ for the fractional part of a real number $\alpha$.
Then for any set of integers $v_1, \ldots, v_{\ell - 2}$, and $a \in \mathbb{Z}$ with $0 \leq a < X$, the inequality
\begin{equation}
\label{ineq C1 1}
\frac{a}{X}  \leq \{ \widetilde{c} \gamma_{\ell, r} v_1 \ldots v_{\ell - 1} \} < \frac{a + 1}{X}
\end{equation}
cannot hold for more than $2 N_0(v_1, \ldots, v_{\ell-2})$ integer points $v_{\ell - 1}$ lying inside $[-X,X]$
for the following reason. Suppose this is indeed the case, and let $v_{\ell - 1} \in [-X,X]$ be one integer which satisfies (\ref{ineq C1 1}).
By the pigeon hole principle, we know that at least one of
$$
\{ v'_{\ell - 1} \in [-X,X] \cap \mathbb{Z} : v'_{\ell - 1} \text{ satisfies (\ref{ineq C1 1}) and } v_{\ell - 1}  v'_{\ell - 1} \geq 0  \}
$$
or
$$
\{ v'_{\ell - 1} \in [-X,X] \cap \mathbb{Z}  : v'_{\ell - 1} \text{ satisfies (\ref{ineq C1 1}) and } v_{\ell - 1}  v'_{\ell - 1} < 0  \}
$$
has cardinality greater than $N_0(v_1, \ldots, v_{\ell-2})$.
Suppose it is the former set (we can argue in a similar fashion for the latter set as well).
If $v_{\ell - 1} $ and $v'_{\ell - 1}$ are two distinct points that satisfy (\ref{ineq C1 1}), then we have
$$
\| \widetilde{c} \gamma_{\ell, r} v_1 \ldots v_{\ell - 2} (v_{\ell - 1} - v'_{\ell - 1}) \| < \frac{1}{X}
$$
and $(v_{\ell - 1} - v'_{\ell - 1})  \in [-X, X] \cap \mathbb{Z}$.
Consequently, we have $( v_1, \ldots,  v_{\ell - 2},  v_{\ell - 1} - v'_{\ell - 1}) \in \mathcal{A}_X$
from which we can obtain contradiction.
Therefore, we obtain the following inequalities
\begin{eqnarray}
&& \sum_{v_{\ell - 1} \in [-X, X]  } \min \left( X,  \| \widetilde{c} \gamma_{\ell, r} v_1 \ldots v_{\ell - 1}   \|^{-1}   \right)
\label{bound with N0}
\\
&\ll& N_0(v_1, \ldots, v_{\ell-2}) \sum_{0 \leq a \leq X}^{}  \min \left( X,  \max \left( \frac{X}{a}, \frac{X}{|X - a -1|}  \right)   \right)
\notag
\\
&\ll& N_0(v_1, \ldots, v_{\ell-2}) X \log X.
\notag
\end{eqnarray}
Thus via (\ref{bound on AX}), (\ref{eqn for AX}), and (\ref{bound with N0}), we have the following bound for (\ref{S'_0 power}),
\begin{eqnarray}
\notag
&&|S(w_1, \ldots, w_K, \mathbf{G}, \mathbf{H} )|^{2^{\ell}}
\\
&\leq& (\log X)^{ 2^{\ell} k } X^{2^{\ell} k - \ell}
\sum_{v_{1} \in [-X, X] }
\ldots
\sum_{v_{\ell - 1} \in [-X, X] } \min \left( X,  \|\widetilde{c} \gamma_{\ell, r} v_1 \ldots v_{\ell - 1}   \|^{-1}   \right)
\notag
\\
&\ll& (\log X)^{ 2^{\ell} k } X^{2^{\ell} k - \ell}
\sum_{v_{1} \in [-X, X] }
\ldots
\sum_{v_{\ell - 2} \in [-X, X] } N_0(v_1, \ldots, v_{\ell-2}) X \log X
\notag
\\
&=& (\log X)^{2^{\ell} k } X^{2^{\ell} k - \ell} \ |\mathcal{A}_{X}| \  X \log X
\notag
\\
&\leq&  X^{2^{\ell} k} (\log X)^{2^{\ell}k + 1 - C''/d},
\notag
\end{eqnarray}
and hence
$$
|S(w_1, \ldots, w_K, \mathbf{G}, \mathbf{H} )| \ll  X^{k} (\log X)^{k  + 2^{- \ell}(1 - C''/d) }.
$$
Therefore, we obtain from (\ref{S'0 case 2}) that
\begin{eqnarray}
\notag
|S_0( \boldsymbol{\alpha}, \mathbf{G}, \mathbf{H} )| &\ll& (\log X)^{K-k} \cdot \sum_{ w_K \in [0,X] } \ldots \sum_{ w_{k + 1} \in [0,X] }  X^{k} (\log X)^{k  + 2^{- \ell}(1 - C''/d) }
\\
\notag
&\ll& (\log X)^{K} X^{K} (\log X)^{ 2^{- \ell}(1 - C''/d) }.
\end{eqnarray}
The case $\ell = 2$ can be dealt with in a similar and more simple manner.
Recall from above $C'' = (C-1)/ R$ and $K \leq dR$. Thus we make sure $C$ is sufficiently large with respect to $d$ and $R$.
This completes the proof of Claim 1, and hence the proof of Proposition \ref{prop minor arc bound} as well.
\end{proof}

\section{Technical estimates}
\label{sec tech 1}
In this section, we collect results related to Weyl differencing
that are necessary in obtaining estimates for the singular integral and the singular series defined in (\ref{defn of sing int}) and (\ref{def sigular series}), respectively.

Let us denote $\mathfrak{B}_0 = [0,1]^n$.
Let $\boldsymbol{\alpha} = (\boldsymbol{\alpha}_d, \ldots, \boldsymbol{\alpha}_1) \in \mathbb{R}^R$,
where $R = r_1 + \ldots + r_d$ and $\boldsymbol{\alpha}_{\ell} = (\alpha_{\ell, 1}, \ldots , \alpha_{\ell, r_{\ell}} ) \in \mathbb{R}^{r_{\ell}}$ $(1 \leq \ell \leq d)$.
We define
$$
\| \boldsymbol{\alpha} \| = \max_{ \substack{  1 \leq \ell \leq d  \\ 1 \leq r \leq r_{\ell}}} \| \alpha_{\ell, r} \|
\ \ \
\text{  and  } \ \  \
| \boldsymbol{\alpha} | = \max_{ \substack{  1 \leq \ell \leq d  \\ 1 \leq r \leq r_{\ell}}} | \alpha_{\ell, r} |.
$$

Let $\mathbf{u} = ( \mathbf{u}_{d}, \ldots, \mathbf{u}_{1} )$ be a system of polynomials in $\mathbb{Q}[x_1, \ldots, x_n]$, where
$\mathbf{u}_{\ell}  = ( u_{\ell,1}, \ldots, u_{\ell, r_{\ell}} )$ is the subsystem of degree $\ell$ polynomials of $\mathbf{u}$  $(1 \leq \ell \leq d)$.
We let $\mathbf{U} = ( \mathbf{U}_{d}, \ldots, \mathbf{U}_{1} )$ be the system of forms, where for each $1 \leq \ell \leq d$,
$\mathbf{U}_{\ell} = ( U_{\ell,1}, \ldots, U_{\ell, r_{\ell}} )$
and $U_{\ell, r}$ is the degree $\ell$ portion of $u_{\ell, r}$ $(1 \leq r \leq r_{\ell} )$.
We define the following exponential sum associated to  $\mathbf{u}$,
\begin{equation}
\label{def of S 1+}
S( \boldsymbol{\alpha}) = S( \mathbf{u},  \mathfrak{B}_0 ;\boldsymbol{\alpha}) := \sum_{\mathbf{x} \in P \mathfrak{B}_0 \cap \mathbb{Z}^n}
e \left( \sum_{1 \leq \ell \leq d} \sum_{ 1 \leq r \leq r_{\ell} } {\alpha}_{\ell, r} \cdot {u}_{\ell, r}  (\mathbf{x})  \right).
\end{equation}

Let $1 < \ell \leq d$ and $r_{\ell} > 0$. We let $\mathbb{M}_{\ell} = \mathbb{M}_{\ell} (\mathbf{U}_{\ell})$ be
the affine variety in $(\mathbb{C}^n)^{\ell-1}$ associated to $\mathbf{U}_{\ell}$, for which the definition
we provide in (\ref{def of MU}) of Appendix \ref{appendix B}. For $R_0>0$, we denote $z_{R_0} (\mathbb{M}_{\ell})$ to be the number of integer points $(\mathbf{x}_1, \ldots, \mathbf{x}_{\ell-1} )$ on
$\mathbb{M}_{\ell}$ such that $$\max_{1 \leq i \leq \ell -1} \max_{1 \leq j \leq n}  | x_{ij} | \leq R_0,$$
where $\mathbf{x}_i = (x_{i1}, \ldots, x_{in}) \ (1 \leq i \leq \ell-1)$. We define $g_{\ell}( \mathbf{U}_{\ell} )$
to be the largest real number such that
\begin{equation}
\label{def gd+}
z_P(\mathbb{M}_{\ell}) \ll P^{n({\ell}-1) - g_{\ell}( \mathbf{U}_{\ell} ) + \varepsilon}
\end{equation}
holds for each $\varepsilon >0$. It was proved in \cite[pp. 280, Corollary]{S} that
\begin{equation}
\label{h and g+}
h_{\ell}( \mathbf{U}_{\ell} ) < \frac{\ell!}{  (\log 2)^{\ell} }  \left( g_{\ell}( \mathbf{U}_{\ell}  ) + ({\ell}-1)r_{\ell} (r_{\ell} - 1)  \right).
\end{equation}
Let
$$
\gamma_{\ell} = \frac{2^{{\ell}-1} ({\ell}-1) r_{\ell}}{ g_{\ell}( \mathbf{U}_{\ell} ) }
$$
when $r_{\ell} >0$ and $g_{\ell}( \mathbf{U}_{\ell} ) > 0$.
We let
$\gamma_{\ell} = 0$ if $r_{\ell} = 0$, and let
$\gamma_{\ell} = + \infty$ if $r_{\ell} > 0$ and $g_{\ell}( \mathbf{U}_{\ell} ) = 0$.
For $\ell$ with $r_{\ell}>0$, we also define
\begin{equation}
\label{def gamma'+}
\gamma'_{\ell} = \frac{ 2^{{\ell}-1} }{ g_{\ell}( \mathbf{U}_{\ell} ) } = \frac{ \gamma_{\ell} }{ ({\ell}-1) r_{\ell} }.
\end{equation}

We need the following lemma to obtain estimates on the singular integral.
Let
$$
\mathcal{I}( \mathfrak{B}_0 ,  \boldsymbol{\tau}) = \int_{\mathbf{v} \in \mathfrak{B}_0 } e \left( \sum_{\ell = 1}^d \sum_{r=1}^{r_{\ell}} \tau_{\ell,r}  \cdot U_{\ell,r}(\mathbf{v})  \right) \ \mathbf{d}\mathbf{v}.
$$
\begin{lem} \cite[Lemma 2.7]{Y}
\label{Lemma 8.1 in S+}
Suppose $\mathbf{u}$ has coefficients in $\mathbb{Z}$, and that
$\mathcal{B}_1(\mathbf{u}_1)$ is sufficiently large with respect to $r_d, \ldots, r_1$, and $d$.
Furthermore, suppose  $\gamma_2, \ldots, \gamma_d$ are sufficiently small with respect to $r_d, \ldots, r_1$, and $d$.
Then we have
\begin{equation}
\label{(3.9) is S}
\mathcal{I}( \mathfrak{B}_0,  \boldsymbol{\tau}) \ll \min (1 , |\boldsymbol{\tau}|^{-R - 1} ),
\end{equation}
where the implicit constant depends at most on $d$, $r_d, \ldots, r_1$, and $\mathbf{U}$.
\end{lem}

We refer the reader to \cite{Y} for a proof of this lemma. The proof in \cite{Y} is similar to that of \cite[Lemma 8.1]{S},
which is for systems without linear polynomials. However, due to the presence of linear polynomials it requires
some justification not available in \cite{S}.

We also need to deal with certain situations where the coefficients of $\mathbf{u}$ may depend on $P$. There are essentially two different
scenarios we have to consider, first of which we refer to as the following.
\newline

Condition $(\star')$: The polynomials of $\mathbf{u}$ have coefficients in $\mathbb{Z}$, and the coefficients of $\mathbf{U}$ do not depend on $P$.
However, for each $u_{\ell, r}(\mathbf{x})$ $(1 \leq \ell \leq d, 1 \leq r \leq r_{\ell})$ the coefficients of its monomials whose
degrees are strictly less than $\ell$ may depend on $P$.
\newline

We have the following result when $\mathbf{u}$ satisfies Condition $(\star')$.
\begin{cor} 
\label{cor 15.1 in S'+}
Suppose $\mathbf{u}$ satisfies Condition $(\star')$.
Let $S( \boldsymbol{\alpha}) $ be the sum associated to $\mathbf{u}$ as in ~(\ref{def of S 1+}).
Suppose $\varepsilon' > 0$ is sufficiently small and $Q>0$ satisfies
$$
Q \gamma'_d < 1.
$$
Then one of the following alternatives must hold:

$(i)$ $|  S( \boldsymbol{\alpha})  | \leq P^{n-Q}$.

$(ii)$ There exists $n_0 \in \mathbb{N}$ such that
$$
n_0 \ll P^{Q \gamma_d + \varepsilon'} \text{  and  } \|  n_0 \boldsymbol{\alpha}_d \| \ll P^{ -d + Q \gamma_d + \varepsilon'}.
$$
The implicit constants depend only on $n, d, r_d, \varepsilon', Q$, and $\mathbf{U}_{d}$.
\end{cor}

Next we present the result in our second scenario for when the coefficients of $\mathbf{u}$ may depend on $P$.
Let $u^{(j)}_{\ell,r}(\mathbf{x})$ be the homogeneous degree $j$ portion of the polynomial $u_{\ell,r}(\mathbf{x})$.
In the following corollary, for $j < \ell$ the coefficients of $u^{(j)}_{\ell,r}(\mathbf{x})$ may be in $\mathbb{Q}$ and also depend on $P$, but in a controlled manner. On the other hand, the coefficients of
$U_{\ell,r}(\mathbf{x})$ do not depend on $P$.
\begin{cor}  
\label{cor 15.1 in S+}
Suppose $\mathbf{u}$ has coefficients in $\mathbb{Q}$, and further suppose $\mathbf{U}$ has coefficients in $\mathbb{Z}$.
Let $Q > 0$ and $\varepsilon >0$. Let $2 \leq \ell \leq d$ with $r_{\ell} > 0$.
If $\ell =d$, then let $\theta = 0$ and $q=1$. On the other hand, if
$2 \leq \ell < d$, then suppose $0 \leq \theta < 1/4$ and that there is $q \in \mathbb{N}$ with
$$
q \leq P^{\theta}, \ \ \    q \boldsymbol{\alpha}_{j} \in \mathbb{Z}^{r_j} \ \   (\ell < j \leq d),
$$
and
$$
q \alpha_{\ell', r} u^{(j)}_{\ell', r}(\mathbf{x}) \in \mathbb{Z}[x_1, \ldots, x_n]
$$
for every $\ell < \ell' \leq d, 0 \leq j < \ell', 1 \leq r \leq r_{\ell'}$.
Let $S( \boldsymbol{\alpha})$ be the sum associated to $\mathbf{u}$ as in ~(\ref{def of S 1+}). 
Suppose
$$
4 \theta + Q \gamma'_{\ell} < 1.
$$
Then one of the following alternatives must hold:

$(i)$ $|  S( \boldsymbol{\alpha})  | \leq P^{n-Q}$.

$(ii)$ There exists $n_0 \in \mathbb{N}$ such that
$$
n_0 \ll P^{ Q \gamma_{\ell} + \varepsilon } \text{  and  } \|  n_0 q \boldsymbol{\alpha}_{\ell} \| \ll P^{ -\ell + 4 \theta + Q \gamma_{\ell} + \varepsilon}.
$$
\newline
The implicit constants
depend at most on $n,d, r_d, \ldots, r_1, Q, \varepsilon$, and $\mathbf{U}$.
\end{cor}

We present the details of proof of Corollaries \ref{cor 15.1 in S'+} and \ref{cor 15.1 in S+} in Appendix \ref{appendix B}.

\section{Hardy-Littlewood Circle Method: Major Arcs}
\label{section major arcs}
For $\mathbf{x} = (x_1, \ldots, x_n)$, let us denote $ \widehat{\mathbf{x}} = (x_1, \ldots, x_{n-r_1})$.
In this section, we consider the system of equations
\begin{equation}
\label{set of eqn 3}
f_{\ell, r} (\mathbf{x}) = 0 \ \ (1 \leq \ell \leq d, 1 \leq r \leq r_{\ell}),
\end{equation}
where we assume $\mathbf{f}$ is of the shape
$$
f_{\ell, r} (\mathbf{x}) = f_{\ell, r} (\widehat{\mathbf{x}}) \in \mathbb{Z}[x_1, \ldots, x_{n-r_1}] \ \ (2 \leq \ell \leq d, 1 \leq r \leq r_{\ell}),
$$
and
$$
f_{1, r} (\mathbf{x}) = c_{1, r} x_{n - r_1 + r}  +  \widetilde{f}_{1, r} (\widehat{\mathbf{x}}) \ \ (1 \leq r \leq r_{1}),
$$
where $c_{1,r} \in \mathbb{Z} \backslash \{ 0 \}$ and $\widetilde{f}_{1, r} (\widehat{\mathbf{x}}) \in \mathbb{Z}[x_1, \ldots, x_{n-r_1}]$.
We further assume $\mathbf{f}$ satisfies the following:
$h_d(\mathbf{f}_{d}), \ldots, h_2(\mathbf{f}_{2})$, and $\mathcal{B}_1(\mathbf{f}_1)$ are all sufficiently large with respect to
$d$ and $r_d, \ldots, r_1$. Clearly systems with these assumptions contain $\mathbf{f}$ in (\ref{set of eqn 2}) as a special case. We also denote
$F_{\ell, r}$ to be the homogeneous degree $\ell$ portion of $f_{\ell, r}$ $(1 \leq \ell \leq d, 1 \leq r \leq r_{\ell})$,
and let $\mathbf{F}_{\ell} = (F_{\ell, 1}, \ldots, F_{\ell, r_{\ell}})$ $(1 \leq \ell \leq d)$.

Let $\mathfrak{B}_0 = [0,1]^n \subseteq \mathbb{R}^n$. Given $\mathbf{b} \in (\mathbb{Z} / q \mathbb{Z} )^n$, we define
$$
\boldsymbol{\psi}_{\mathbf{b}}(\mathbf{t}) = \psi_{{b}_1}({t}_1) \ldots \psi_{{b}_n}({t}_n),
$$
where
$$
\psi_{{b}_j}({t}_j) = \sum_{ \substack{ 0 \leq v \leq t_j  \\ v \equiv b_j (\text{mod } q) } } \Lambda(v).
$$
We use the notation $\mathbf{x} \equiv \mathbf{b} \ (\text{mod } q)$ to mean $x_i \equiv b_i \ (\text{mod } q)$ for each $1 \leq i \leq n$.
Suppose for $\boldsymbol{\alpha} \in [0,1)^R$, we have $\boldsymbol{\alpha} = \mathbf{a}/q + \boldsymbol{\beta}$ where $\mathbf{a} \in (\mathbb{Z}/ q\mathbb{Z})^R$.
Then we have
\begin{eqnarray}
&&T(\mathbf{f}; \boldsymbol{\alpha})
\\
&=&
\sum_{\mathbf{x} \in [0,X]^n } \Lambda(\mathbf{x}) \ e \left( \sum_{\ell = 1}^d \sum_{r=1}^{r_{\ell}} \alpha_{\ell,r} \cdot f_{\ell,r}(\mathbf{x})  \right)
\notag
\\
&=&
\sum_{\mathbf{b} \in (\mathbf{Z}/ q\mathbf{Z})^n } \ \sum_{ \substack{ \mathbf{x} \in [0,X]^n   \\ \mathbf{x} \equiv \mathbf{b}  (\text{mod } q) } } \Lambda(\mathbf{x})  \
e \left( \sum_{\ell = 1}^d \sum_{r=1}^{r_{\ell}} a_{\ell,r} \cdot f_{\ell,r} (\mathbf{b}) / q  \right)
 e \left( \sum_{\ell = 1}^d \sum_{r=1}^{r_{\ell}} \beta_{\ell,r} \cdot f_{\ell,r} (\mathbf{x})  \right)
\notag
\\
&=&
\sum_{\mathbf{b} \in (\mathbf{Z}/ q\mathbf{Z})^n } \
e \left( \sum_{\ell = 1}^d \sum_{r=1}^{r_{\ell}} a_{\ell,r} \cdot f_{\ell,r} (\mathbf{b}) / q \right)
\int_{\mathbf{t} \in X \mathfrak{B}_0 } e \left( \sum_{\ell = 1}^d \sum_{r=1}^{r_{\ell}} \beta_{\ell,r} \cdot f_{\ell,r} (\mathbf{t})  \right)
\ \mathbf{d} \boldsymbol{\psi}_{\mathbf{b}}(\mathbf{t}),
\notag
\end{eqnarray}
where $\mathbf{d} \boldsymbol{\psi}_{\mathbf{b}}(\mathbf{t})$ denotes the product measure $d{\psi}_{{b}_1}({t}_1) \times \ldots \times d{\psi}_{{b}_n}({t}_n).$

Let $\phi$ be Euler's totient function. For a positive integer $q$, recall we put $\mathbb{U}_q$ for the group of units in $\mathbb{Z}/q \mathbb{Z}$.
Lemma \ref{Lemma 6 in CM} below follows immediately from the proof of  \cite[Lemma 6]{CM} as the proof does not depend on the
fact that the polynomials of the system all have the same degree.
\begin{lem}
\label{Lemma 6 in CM}
Let $c>0$, $C>0$, $q \leq (\log X)^C$, and  $\mathbf{b} \in (\mathbb{Z} / q \mathbb{Z})^n$. Suppose $\boldsymbol{\alpha} = \mathbf{a}/q + \boldsymbol{\beta} \in \mathfrak{M}_{\mathbf{a},q}(C)$.
Then we have
\begin{eqnarray}
&&\int_{\mathbf{t} \in X \mathfrak{B}_0 } e \left( \sum_{\ell = 1}^d \sum_{r=1}^{r_{\ell}} \beta_{\ell,r} \cdot f_{\ell,r}(\mathbf{t})  \right)
\ \mathbf{d} \boldsymbol{\psi}_{\mathbf{b}}(\mathbf{t})
\notag
\\
&=& \mathbf{1}_{ \mathbf{b} \in (\mathbb{U}_q)^n } \  \frac{1}{\phi(q)^n} \int_{\mathbf{v} \in X \mathfrak{B}_0 } e \left( \sum_{\ell = 1}^d \sum_{r=1}^{r_{\ell}} \beta_{\ell,r} \cdot f_{\ell,r} (\mathbf{v})  \right) \ \mathbf{d}\mathbf{v} + O( X^n / (\log X)^c ),
\notag
\end{eqnarray}
where $\mathbf{1}_{ \mathbf{b} \in (\mathbb{U}_q)^n }$ is $1$ if $\mathbf{b} \in (\mathbb{U}_q)^n$ and $0$ otherwise.
\end{lem}

Let $\varepsilon > 0$. We simplify the above integral by a change of variable as follows
\begin{eqnarray}
&& \int_{\mathbf{v} \in X \mathfrak{B}_0 } e \left( \sum_{\ell = 1}^d \sum_{r=1}^{r_{\ell}} \beta_{\ell,r} \cdot f_{\ell,r} (\mathbf{v})  \right) \ \mathbf{d}\mathbf{v}
\\
&=& \int_{\mathbf{v} \in X \mathfrak{B}_0 } e \left( \sum_{\ell = 1}^d \sum_{r=1}^{r_{\ell}} \beta_{\ell,r} \cdot F_{\ell,r} (\mathbf{v})  \right) \ \mathbf{d}\mathbf{v} + O(X^{n - 1 + \varepsilon})
\notag
\\
&=& X^n \ \mathcal{I}(\mathfrak{B}_0,  \boldsymbol{\beta}' )  + O(X^{n - 1 + \varepsilon}),
\notag
\end{eqnarray}
where
$$
\beta'_{\ell, r} = X^{\ell} \beta_{\ell, r} \ \ (1 \leq \ell \leq d, 1 \leq r \leq r_{\ell}),
$$
and
$$
\mathcal{I}( \mathfrak{B}_0 ,  \boldsymbol{\tau}) = \int_{\mathbf{v} \in \mathfrak{B}_0 } e \left( \sum_{\ell = 1}^d \sum_{r=1}^{r_{\ell}} \tau_{\ell,r}  \cdot F_{\ell,r}(\mathbf{v})  \right) \ \mathbf{d}\mathbf{v}.
$$
We define
$$
J(L) = \int_{ | \boldsymbol{\tau} |  \leq  L } \mathcal{I}( \mathfrak{B}_0,  \boldsymbol{\tau}) \ \mathbf{d} \boldsymbol{\tau}.
$$

By our assumptions on $\mathbf{f}$ and (\ref{h and g+}), we know we can apply Lemma \ref{Lemma 8.1 in S+}
and obtain $\mathcal{I}( \mathfrak{B}_0,  \boldsymbol{\tau}) \ll \min (1 , |\boldsymbol{\tau}|^{-R - 1} ).$
With this estimate, it is an easy exercise to show that
\begin{equation}
\label{defn of sing int}
\mu(\infty) = \int_{\boldsymbol{\tau} \in \mathbb{R}^R} \mathcal{I}( \mathfrak{B}_0,  \boldsymbol{\tau}) \ \mathbf{d} \boldsymbol{\tau},
\end{equation}
which is called the \textit{singular integral}, exists, and that
\begin{equation}
\label{(3.9') is S}
\Big{|} \mu(\infty) - J(L) \Big{|} \ll L^{- 1}.
\end{equation}
We note that $\mu(\infty)$ is the same as what is defined in \cite[(2.3)]{BHB}, and we have
\begin{equation}
\label{mu infty positive}
\mu(\infty) > 0
\end{equation}
provided that the system of equations
$$
F_{\ell, r}(\mathbf{x}) = 0 \ \ (1 \leq \ell \leq d, 1 \leq r \leq r_{\ell})
$$
has a non-singular real solution in $(0,1)^n$. The argument used to show this fact is standard
and we refer the reader to see for example \cite[Chapter 16]{D}, or the explanation in \cite{BHB}.

We define the following sums:
\begin{equation}
\label{defn Stilde}
\mathcal{S}_{ \mathbf{a}, q } = \sum_{\mathbf{k} \in (\mathbb{U}_q)^n} e \left(
\sum_{\ell = 1}^d \sum_{r=1}^{r_{\ell}}  f_{\ell,r} (\mathbf{k})  \cdot a_{\ell,r} /q \right),
\end{equation}
$$
B( q ) = \sum_{ \substack{ \gcd (\mathbf{a},q) = 1 \\  \mathbf{a} \in (\mathbb{Z} / q\mathbb{Z} )^R } } \frac{1}{\phi(q)^n} \ \mathcal{S}_{ \mathbf{a}, q }, 
$$
and
\begin{equation}
\label{def sigular series partial}
\mathfrak{S}(X) = \sum_{q \leq (\log X)^C}  B( q ).
\end{equation}

By combining Lemma \ref{Lemma 6 in CM} with the definitions given above, we have the following.
\begin{lem}\cite[Lemma 8]{CM}
\label{Lemma 8 in CM} Given any $c>0$, $C > 0$, and $q \leq (\log X)^C$,  we have
$$
\int_{\mathfrak{M}_{\mathbf{a}, q}(C)} T(\mathbf{f}; \boldsymbol{\alpha} ) \ \mathbf{d} \boldsymbol{\alpha}
=
\frac{X^{n - \sum_{\ell=1}^d \ell r_{\ell} }}{\phi(q)^n} \ \mathcal{S}_{\mathbf{a}, q } \  J( (\log X)^C ) + O\left(  \frac{X^{n - \sum_{\ell=1}^d \ell r_{\ell}}}{(\log X)^{c} } \right).
$$
\end{lem}

Therefore, we obtain the following estimate as a consequence of the definition of the major arcs, (\ref{(3.9') is S}), and Lemma \ref{Lemma 8 in CM}.
\begin{lem}
\label{lemma major arc estimate}
Given any $c>0$ and $C > 0$, we have
$$
\int_{\mathfrak{M}(C) }T(\mathbf{f}; \boldsymbol{\alpha} ) \ \mathbf{d} \boldsymbol{\alpha}
=
\mathfrak{S}(X) \mu(\infty) X^{n - \sum_{\ell=1}^d \ell r_{\ell} } + O\left( \mathfrak{S}(X) \frac{ X^{n - \sum_{\ell=1}^d \ell r_{\ell}}}{(\log X)^C} +  \frac{X^{n - \sum_{\ell=1}^d \ell r_{\ell} }}{(\log X)^c} \right).
$$
\end{lem}

We still have to deal with the term $\mathfrak{S}(X)$, and this is done in the following section.
\subsection{Singular Series}
\label{section singular series}
In order to estimate the term $\mathfrak{S}(X)$, 
we begin by obtaining estimates for the exponential sum $\mathcal{S}_{ \boldsymbol{a}, q }$ defined in (\ref{defn Stilde}).
We define $g_{\ell}( \mathbf{F}_{\ell} )$ as in (\ref{def gd+}).
It then follows from (\ref{h and g+}) 
that
$$
h_{\ell}( \mathbf{F}_{\ell}  ) < (\log 2)^{-\ell} \cdot \ell! \cdot \left( g_{\ell}(\mathbf{F}_{\ell} ) + (\ell-1) r_{\ell} (r_{\ell} -1) \right)
$$
for $2 \leq \ell \leq d$ with $r_{\ell} > 0$.
From this inequality, for $2 \leq \ell \leq d$ with $r_{\ell} > 0$ we see that we can assume $g_{\ell}(\mathbf{F}_{\ell} )$ to be sufficiently large
with respect to $d$ and $r_d, \ldots, r_1$.

Let
$$
Q  = 1 +  \max \Big{\{} \  \frac{1 + R(800 d^3 + 2)}{800 d^3  + 1},  \frac{R+1}{1 - \frac{1}{800 d^3 + 1}}  \ \Big{\}}.
$$
With our assumptions in this section, $Q$ satisfies the following,
$$
4  \left( \gamma_2 Q  +  \gamma_3 Q + \ldots  + \gamma_d Q + \frac{1}{800 d } \right) < \frac{1}{100 d},
$$
\begin{eqnarray}
\label{Q bound 1'}
&&Q \cdot r_\ell (\ell-1) \cdot 2^{\ell -1} \left( \frac{ (\log 2)^{\ell} (h_{\ell}( \mathbf{F}_{\ell}  ) - (800 d^3  + 1) Q)  }{ \ell !} -  (\ell-1) r_{\ell} (r_{\ell} -1) \right)^{-1}
\\
&<& \frac{1}{1600 d^3 + 2} \ \ \ \ (2 \leq \ell \leq d),
\notag
\end{eqnarray}
and
\begin{equation}
\label{Q bound 1'''}
0 < Q < \frac{d-1}{d(r_1 + 1)} (\gamma_2 + 4 \gamma_3 + \ldots + 4^{d-2} \gamma_d)^{-1},
\end{equation}
where $\gamma_{\ell}$ is defined (with respect to $\mathbf{F}_{\ell}$ here) after (\ref{h and g+}).
We fix this value of $Q$ throughout the remainder of this section. Also since $\mathcal{B}_1(\mathbf{F}_1)$ is sufficiently large with respect to
$d$ and $r_d, \ldots, r_1$, we have $\mathcal{B}_1(\mathbf{F}_1) > Q$.

We consider two cases depending on $\mathbf{a}$ to bound $\mathcal{S}_{ \boldsymbol{a}, q }$ when $q$ is a prime power. These cases are
treated separately in Lemmas \ref{to bound local factor} and \ref{to bound local factor+}.
\begin{lem}
\label{to bound local factor}
Let $p$ be a prime and let $q = p^t$, $t \in \mathbb{N}.$ Let $\mathbf{a}  = (\mathbf{a}_d, \ldots, \mathbf{a}_1)  \in (\mathbb{Z} / q \mathbb{Z})^R$ with $\gcd(\mathbf{a},q)=1$. Furthermore, suppose there exists $\ell \in \{ 2, \ldots, d \}$ such that $\gcd(\mathbf{a}_{\ell},q)=1$. Then we have the following bounds
\begin{eqnarray}
\notag
\mathcal{S}_{ \boldsymbol{a}, q } \ll
\left\{
    \begin{array}{ll}
         q^{n-Q},
         &\mbox{if } t \leq 800 d^3 + 1 ,\\
         p^Q q^{n-Q},
         &\mbox{if } t > 800 d^3 + 1,
    \end{array}
\right.
\end{eqnarray}
where the implicit constants are independent of $p$.
\end{lem}

\begin{proof}
We consider the two cases $t \leq 800 d^3 + 1$ and $t >  800 d^3  + 1$ separately.
We begin with the case $t \leq 800 d^3 + 1$. In this case,
we apply the inclusion-exclusion principle to $\mathcal{S}_{ \mathbf{a}, q }$.
As a result, we obtain
\begin{eqnarray}
\label{bound on S tilde}
&&
\\
\mathcal{S}_{ \mathbf{a}, q }
&=&  \sum_{\mathbf{k} \in (\mathbb{U}_q)^{n}} e \left( \sum_{\ell = 1}^d \sum_{r=1}^{r_{\ell}}  f_{\ell,r}(\mathbf{k})  \cdot a_{\ell,r} /q  \right)
\notag
\\
&=&  \sum_{\mathbf{k} \in (\mathbb{Z} / q \mathbb{Z})^{n}} \prod_{i=1}^n \left( 1 -  \sum_{v_i \in \mathbb{Z} / p^{t-1} \mathbb{Z} } \mathbf{1}_{k_i = p v_i}  \right)
e \left( \sum_{\ell = 1}^d \sum_{r=1}^{r_{\ell}}  f_{\ell,r}(\mathbf{k})  \cdot a_{\ell,r} /q  \right)
\notag
\\
\notag
&=& \sum_{ I \subseteq \{ 1, 2, \ldots, n \}} (-1)^{|I|}
\sum_{ \mathbf{v} \in  (\mathbb{Z} / p^{t-1} \mathbb{Z} )^{|I|} } \
\sum_{\mathbf{k} \in (\mathbb{Z} / q \mathbb{Z})^{n} }
\mathfrak{F}_I(\mathbf{k}; \mathbf{v}) \ e \left( \sum_{\ell = 1}^d \sum_{r=1}^{r_{\ell}}  f_{\ell,r} (\mathbf{k})  \cdot a_{\ell,r} /q  \right),
\notag
\end{eqnarray}
where
$$
\mathbf{1}_{k_i = p v_i} =
\left\{
    \begin{array}{ll}
         1,
         &\mbox{if } k_i = p v_i ,\\
         0,
         &\mbox{if } k_i \not = p v_i,
    \end{array}
\right.
$$
and
$$
\mathfrak{F}_I(\mathbf{k}; \mathbf{v}) = \prod_{i \in I} \mathbf{1}_{k_i = p v_i}
$$
for $\mathbf{v} \in (\mathbb{Z} / p^{t-1} \mathbb{Z})^{|I|}$.
In other words, $\mathfrak{F}_I(\mathbf{k}; \mathbf{v})$ is the characteristic function of the set $H_{I, \mathbf{v}} = \{ \mathbf{k} \in (\mathbb{Z} / q \mathbb{Z})^{n} : k_i = p v_i \ (i \in I)\}$.
We now bound the summand in the final expression of ~(\ref{bound on S tilde}) by further considering two cases, $|I| \geq t Q$ and $|I| < t Q$.
In the first case $|I| \geq t Q$, we use the following trivial estimate
\begin{eqnarray}
\Big{|} \sum_{ \mathbf{v} \in  (\mathbb{Z} / p^{t-1} \mathbb{Z})^{|I|} } \
\sum_{\mathbf{k} \in (\mathbb{Z} / q \mathbb{Z})^n}
\mathfrak{F}_I(\mathbf{k}; \mathbf{v}) \ e \left( \sum_{\ell = 1}^d \sum_{r=1}^{r_{\ell}}  f_{\ell,r}(\mathbf{k})  \cdot a_{\ell,r} /q  \right) \Big{|}
&\leq&
p^{(t-1) |I|} (p^t)^{n  - |I|}
\notag
\\
&=& q^{n - |I|/t}
\notag
\\
&\leq& q^{n  - Q}.
\notag
\end{eqnarray}

On the other hand, suppose  $|I| < t Q$. Let us label $\mathbf{s} = (s_1, \ldots , s_{n - |I|} )$
to be the remaining variables of $\mathbf{x}$ after setting $x_i = 0$ for each $i \in I$.
For each $1 \leq \ell \leq d, 1 \leq r \leq r_{\ell}$, let
$$
\mathfrak{g}_{\ell,r}(\mathbf{s}) ={f}_{\ell,r}(\mathbf{x}) |_{x_i = p v_i ( i \in I )},
$$
or equivalently the polynomial $\mathfrak{g}_{\ell,r}(\mathbf{s})$ is obtained by
substituting $x_i = p v_i \ (i \in I)$ to the polynomial ${f}_{\ell,r}(\mathbf{x})$.
Thus $\mathfrak{g}_{\ell,r} (\mathbf{s})$ is a polynomial in $n - |I|$ variables whose coefficients may depend on $p$.
With these notations we have
\begin{eqnarray}
\sum_{\mathbf{k} \in (\mathbb{Z} / q \mathbb{Z})^{n}} \mathfrak{F}_I(\mathbf{k}; \mathbf{v}) \ e \left( \sum_{\ell = 1}^d \sum_{r=1}^{r_{\ell}}  f_{\ell,r}(\mathbf{k})  \cdot a_{\ell,r} /q  \right)
\notag
&=&
\sum_{\mathbf{s} \in (\mathbb{Z} / q \mathbb{Z})^{n - |I|}}e \left( \sum_{\ell = 1}^d \sum_{r=1}^{r_{\ell}}  \mathfrak{g}_{\ell,r} (\mathbf{s})  \cdot a_{\ell,r} /q  \right)
\\
\notag
&=&
\sum_{\mathbf{s} \in [0, q-1]^{n - |I|} }e \left( \sum_{\ell = 1}^d \sum_{r=1}^{r_{\ell}}  \mathfrak{g}_{\ell,r} (\mathbf{s})  \cdot a_{\ell,r} /q  \right).
\end{eqnarray}
We can also deduce easily that the homogeneous degree $\ell$ portion of the polynomial $\mathfrak{g}_{\ell,r}(\mathbf{s})$, which we denote
$\mathfrak{G}_{\ell,r}(\mathbf{s})$, is obtained by substituting $x_i = 0 \ (i \in I)$ to ${F}_{\ell,r}(\mathbf{x})$.
Hence, we have
$$
\mathfrak{G}_{\ell,r}(\mathbf{s}) =  {F}_{\ell,r}(\mathbf{x}) |_{x_i = 0 \ (i \in I)},
$$
and in particular, it is independent of $p$.
Thus the system of polynomials $\mathfrak{g}_{\ell,r}(\mathbf{s})$ $(1 \leq \ell \leq d, 1 \leq r \leq r_{\ell})$
satisfies Condition $(\star')$.
It also follows by Lemma \ref{h ineq 1'} that
$$
h_{\ell}( \{\mathfrak{G}_{\ell,r}: 1 \leq r \leq r_{\ell} \} ) \geq h_{\ell}( \mathbf{F}_{\ell} ) - |I| > h_{\ell}( \mathbf{F}_{\ell}  ) - (800 d^3 + 1)Q  \ \  (2 \leq \ell \leq d).
$$
By our choice of $Q$, namely (\ref{Q bound 1'}), and from (\ref{h and g+}), we have
$$
Q \gamma'_{\ell} \leq Q \gamma_{\ell} < \frac{1}{1600 d^3 + 2} < 1 \ \ \ (2 \leq \ell \leq d),
$$
where
$\gamma'_{\ell}$ and $\gamma_{\ell}$ are defined with respect to $\{\mathfrak{G}_{\ell,r} : 1 \leq r \leq r_{\ell}  \}$ here.

Take $\varepsilon > 0$ sufficiently small. Let us suppose that $p$ and $t$ are sufficiently large with respect to the coefficients of $\mathbf{F}$,
$n, d, r_d, \ldots, r_1, \varepsilon$, and $Q$,
which implies that $q$ is sufficiently large with respect to the coefficients of $\mathfrak{G}_{\ell,r_{\ell}} (\mathbf{s})$ $(1 \leq \ell \leq d, 1 \leq r \leq r_{\ell})$.
Suppose we have
\begin{equation}
\label{ineq ss 1}
\sum_{\mathbf{s} \in [0, q-1]^{n - |I|}  }e \left( \sum_{\ell = 1}^d \sum_{r=1}^{r_{\ell}}  \mathfrak{g}_{\ell,r} (\mathbf{s})  \cdot a_{\ell,r} /q  \right) > (q-1)^{n - |I| - Q}.
\end{equation}
Then by Corollary \ref{cor 15.1 in S'+} there must exist $n_0 \in \mathbb{N}$ such that
$$
n_0 \ll q^{Q \gamma_d + \varepsilon} \ \ \text{ and } \| n_0 (\mathbf{a}_{d} / q)  \| \ll q^{-d + Q \gamma_d + \varepsilon}.
$$
Since $p$ and $t$ are sufficiently large, we have
$n_0 < q^{ 1/ (1600 d^3 + 2) }$ and $\| n_0 (\mathbf{a}_{d} / q)  \| < q^{-d + 1/(1600 d^3  + 2) }$,
because $Q \gamma_d + \varepsilon < 1/(1600 d^3  + 2)$. Then it follows that $n_0 < p^{ t/(1600 d^3  + 2) } < p$.
Suppose now that not all entries of $n_0 \mathbf{a}_{d}$ is divisible by $q$.
In this case, we obtain
$$
\frac{1}{q} \leq \| n_0 ( \mathbf{a}_{d} / q)  \| < \frac{1}{q^{d - 1/(1600 d^3  + 2)}},
$$
which is a contradiction. Thus all of the entries of $n_0 \mathbf{a}_{d}$ must be divisible
by $q = p^t$ and since $\gcd(n_0, p) = 1$, it follows that all of the entries of $\mathbf{a}_{d}$
must be divisible by $q$.
Therefore, we can simplify the exponential sum in consideration since $e(m) = 1$ when $m \in \mathbb{Z}$, and the inequality (\ref{ineq ss 1}) becomes
\begin{eqnarray}
\sum_{\mathbf{s} \in [0, q-1]^{n - |I|}}e \left( \sum_{\ell = 1}^d \sum_{r=1}^{r_{\ell}}  \mathfrak{g}_{\ell,r} (\mathbf{s})  \cdot a_{\ell,r} /q  \right)
&=&
\sum_{\mathbf{s} \in [0, q-1]^{n - |I|}}e \left( \sum_{\ell = 1}^{d-1} \sum_{r=1}^{r_{\ell}}  \mathfrak{g}_{\ell,r} (\mathbf{s})  \cdot a_{\ell,r} /q  \right)
\\
\notag
&>& (q-1)^{n - |I| - Q}.
\notag
\end{eqnarray}
We may repeat the argument as above, because we are now dealing with
the system of polynomials $\mathfrak{g}_{\ell,r} (\mathbf{s})$ $(1 \leq \ell \leq d-1, 1 \leq r \leq r_{\ell})$.
Again by Corollary \ref{cor 15.1 in S'+} there must exist $n'_0 \in \mathbb{N}$ such that
$$
n'_0 \ll q^{Q \gamma_{d-1} + \varepsilon} \ \ \text{ and } \| n'_0 (\mathbf{a}_{d-1} / q)  \| \ll q^{-(d-1) + Q \gamma_{d-1} + \varepsilon}.
$$
Since $p$ and $t$ are sufficiently large, we have
$n'_0 < q^{ 1/(1600 d^3  + 2) }$ and $\| n'_0 (\mathbf{a}_{d-1} / q)  \| < q^{-(d-1) + 1/( 1600 d^3 + 2) }$,
because $Q \gamma_{d-1} + \varepsilon < 1/(1600 d^3  + 2)$. Then it follows that $n'_0 < p^{ t/(1600 d^3 + 2) } < p$.
Suppose now that not all entries of $n'_0 \mathbf{a}_{d-1}$ is divisible by $q$.
In this case, we obtain
$$
\frac{1}{q} \leq \| n'_0 ( \mathbf{a}_{d-1} / q)  \| < \frac{1}{q^{(d-1) - 1/( 1600 d^3 + 2)}},
$$
which is a contradiction. Thus all of the entries of $n'_0 \mathbf{a}_{d-1}$ must be divisible
by $q = p^t$ and since $\gcd(n'_0, p) = 1$, it follows that all of the entries of $\mathbf{a}_{d-1}$
must be divisible by $q$.
It is then clear that we can repeat the argument and keep reducing until we obtain
that all of the entries of $\mathbf{a}_{\ell}$ must be divisible by $q$ for each $2 \leq \ell \leq d$.
We remark that if there exists $\ell'$ with $r_{\ell'} = 0$, then we simply skip the case $\ell = \ell'$ during this process.
Thus we have $\gcd(\mathbf{a}_{\ell}, q) > 1$ for each $2 \leq \ell \leq d$, which is a contradiction.
As a result we must have
$$
\sum_{\mathbf{s} \in [0, q-1]^{n - |I|}}e \left( \sum_{\ell = 1}^d \sum_{r=1}^{r_{\ell}}  \mathfrak{g}_{\ell,r} (\mathbf{s})  \cdot a_{\ell,r}  /q  \right) \ll q^{n - |I| - Q}.
$$
Thus we obtain
$$
\sum_{ \mathbf{v} \in  (\mathbb{Z} / p^{t-1} \mathbb{Z})^{|I|} } \
\sum_{\mathbf{k} \in (\mathbb{Z} / q \mathbb{Z})^n}
\mathfrak{F}_I(\mathbf{k}; \mathbf{v}) \ e \left( \sum_{\ell = 1}^d \sum_{r=1}^{r_{\ell}}  f_{\ell,r}(\mathbf{k})  \cdot a_{\ell,r} /q  \right)
\ll (p^{t-1})^{|I|} q^{n - |I| - Q} \leq q^{n-Q}.
$$
Consequently, by combining the estimates for the two cases $|I| \geq tQ$ and $|I| < tQ$, we obtain
$$
\mathcal{S}_{ \mathbf{a}, q} \ll q^{n-Q}
$$
when $t \leq 800 d^2 + 1$.

We now consider the case $t> 800 d^3 + 1$. By the definition of $\mathcal{S}_{ \mathbf{a}, q }$, we have
\begin{eqnarray}
\label{widetile S part 3-1}
\mathcal{S}_{ \mathbf{a}, q } &=&  \sum_{\mathbf{k} \in (\mathbb{U}_q)^{n }} e \left( \sum_{\ell = 1}^d \sum_{r=1}^{r_{\ell}}  f_{\ell,r} (\mathbf{k})  \cdot a_{\ell,r} /q  \right)
\\
&=& \sum_{\mathbf{k} \in (\mathbb{U}_p)^{n}} \  \sum_{ \mathbf{y} \in (\mathbb{Z} / p^{t-1} \mathbb{Z} )^{n}  }
e \left( \sum_{\ell = 1}^d \sum_{r=1}^{r_{\ell}}  f_{\ell,r} (\mathbf{k} + p \mathbf{y} )  \cdot a_{\ell,r} /q  \right)
\notag
\\
&=& \sum_{\mathbf{k} \in (\mathbb{U}_p)^{n}} \  \sum_{ \mathbf{y} \in [0,p^{t-1}-1]^{n}   } e \left( \sum_{\ell = 1}^d \sum_{r=1}^{r_{\ell}}  f_{\ell,r} (\mathbf{k} + p \mathbf{y} )  \cdot a_{\ell,r} /q  \right).
\notag
\end{eqnarray}

For each fixed $\mathbf{k} \in (\mathbb{U}_p)^{n}$, we have
$$
f_{\ell,r} (\mathbf{k} + p \mathbf{y})  = p^{\ell}  F_{\ell,r} ( \mathbf{y}) + \omega_{\ell, r: p, \mathbf{k}}(\mathbf{y}) \ \ \ (1 \leq \ell \leq d, 1 \leq r \leq r_{\ell}),
$$
where $\omega_{\ell, r: p, \mathbf{k}}(\mathbf{y})$ is a polynomial in $\mathbf{y}$ of degree at most $\ell-1$ and its coefficients are integers which may depend on $p$ and $\mathbf{k}$.
We let
$$
u_{\ell,r} (\mathbf{y}) = {F}_{\ell,r} ( \mathbf{y}) + \frac{1}{p^{\ell}} \
\omega_{\ell, r: p, \mathbf{k}} (\mathbf{y}) \ \ \ (1 \leq \ell \leq d, 1 \leq r \leq r_{\ell}).
$$
We can then express the inner sum of the last expression of (\ref{widetile S part 3-1}) as
$$
\sum_{ \mathbf{y} \in [0,p^{t-1}-1]^{n}   } e \left( \sum_{\ell = 1}^d \sum_{r=1}^{r_{\ell}}  u_{\ell,r} (\mathbf{y} )  \cdot a_{\ell,r} / (q/p^{\ell})  \right).
$$
We have that $\mathbf{u}$ has coefficients in $\mathbb{Q}$, and $\mathbf{U}$ has coefficients in $\mathbb{Z}$.
Let $\alpha_{\ell,r} = a_{\ell,r} /p^{t-\ell}$ $(1 \leq \ell \leq d, 1 \leq r \leq r_{\ell})$, and $P = (p^{t-1}-1)$.

Recall we have set $Q$ to satisfy
$$
4  \left( \gamma_2 Q  +  \gamma_3 Q + \ldots  + \gamma_d Q + \frac{1}{800 d} \right) < \frac{1}{100 d},
$$
where $\gamma_{\ell}$ is defined with respect to $\mathbf{F}_{\ell}$ here.
Suppose we have
$$
\Big{|} \sum_{ \mathbf{y} \in [0, P]^{n}   } e \left( \sum_{\ell = 1}^d \sum_{r=1}^{r_{\ell}}  u_{\ell,r} (\mathbf{y} )  \cdot \alpha_{\ell,r}  \right) \Big{|} > P^{n - Q}.
$$
Then by Corollary \ref{cor 15.1 in S+}, there must exist $n_d \in \mathbb{N}$ such that
$$
n_d \ll P^{Q \gamma_d + \varepsilon} \ \ \text{ and } \  \  \| n_d \boldsymbol{ \alpha}_{d}  \| \ll P^{-d + Q \gamma_d + \varepsilon}.
$$
For $P = p^{t-1} - 1$ sufficiently large, we have
$$
n_d < p^{ (t-1)/ (100d) } \ \ \text{ and } \  \  \| n_d (\mathbf{a}_{d} / p^{t-d})  \| < p^{(t-1)(-d + 1/(100d) )},
$$
because $Q \gamma_d + \varepsilon < 1/(100d)$.
Suppose now that not all entries of $n_d \mathbf{a}_{d}$ is divisible by $p^{t-d}$.
In this case, we obtain
$$
\frac{1}{p^{t-d}} \leq \| n_d ( \mathbf{a}_{d} / p^{t-d} )  \| < \frac{1}{p^{(t-1)(d - 1/ (100d) )}},
$$
which is a contradiction. Thus all of the entries of $n_d \mathbf{a}_{d}$ must be divisible
by $p^{t-d}$. In particular,
$$
n_d \boldsymbol{\alpha}_{d} = n_d ( \mathbf{a}_{d} / p^{t-d} ) \in \mathbb{Z}^{r_d},
$$
and
we can assume without loss of generality that $n_d$ is a power of $p$ satisfying $n_d < p^{ (t-1)/ (100d) }$.
Since $t-d > (t-1)/ (100d)$, it follows that every entry of $\mathbf{a}_{d}$ is divisible by $p$.

From the inequality $t > 800 d^3 + 1$, we have $p^d < p^{ (t-1) / (800 d^2) }$.
Thus we have
$$
(n_d \ p^d)  \alpha_{d, r} \ \frac{1}{p^{d}} \ \omega_{d, r: p, \mathbf{k}} (\mathbf{y}) \in \mathbb{Z}[y_1, \ldots, y_n] \ \ (1 \leq r \leq r_{d}),
$$
and
$$
n_d \ p^d \leq  P^{Q \gamma_d + 2 \varepsilon} P^{1/ (800 d^2 )}.
$$

With this set up, we can apply Corollary \ref{cor 15.1 in S+} again with $\ell = d-1$ and
$$
\theta = Q \gamma_d +   \frac{1}{800 d^2} + \varepsilon_d < \frac{1}{100 d} < \frac{1}{4},
$$
where
$\varepsilon_d> 0$ is sufficiently small, and deduce that there must exist $n_{d-1} \in \mathbb{N}$ such that
$$
n_{d-1} \ll P^{Q \gamma_{d-1} + \varepsilon} \ \ \text{ and } \| n_{d-1} n_d \ p^d \boldsymbol{ \alpha}_{d-1}  \| \ll  P^{ -(d-1) + 4 \theta + Q \gamma_{d-1} + \varepsilon}.
$$
For $P = p^{t-1} - 1$ sufficiently large, we have
$$
n_{d-1} < p^{ (t-1)/ (100d) } \ \ \text{ and } \ \  \| n_{d-1}  n_d \  p^d (\mathbf{a}_{d-1} / p^{t-(d-1)})  \| < p^{(t-1)(-(d-1) + 1/(100d) )},
$$
because
$$
Q \gamma_{d-1} + 4 \theta +  \varepsilon = Q \gamma_{d-1} + 4 \left( Q \gamma_d +   \frac{1}{800 d^2 } + \varepsilon_d \right) +  \varepsilon < \frac{1}{100d}.
$$
Suppose now that not all entries of $(n_{d-1}  n_d \ p^d \mathbf{a}_{d-1})$ is divisible by $p^{t-(d-1)}$.
In this case, we obtain
$$
\frac{1}{p^{t-(d-1)}} \leq \| n_{d-1} n_d \ p^d ( \mathbf{a}_{d-1} / p^{t-(d-1)} )  \| < \frac{1}{p^{(t-1)( (d-1) - 1/ (100d) )}},
$$
which is a contradiction. Thus all of the entries of $(n_{d-1}  n_d \ p^d \mathbf{a}_{d-1})$ must be divisible
by $p^{t-(d-1)}$. In particular,
$$
n_{d-1} n_d \ p^d \boldsymbol{\alpha}_{d-1} = n_{d-1} n_d \ p^d ( \mathbf{a}_{d-1} / p^{t-(d-1)} ) \in \mathbb{Z}^{r_{d-1}},
$$
and
we can assume without loss of generality that
$n_{d-1}$ is a power of $p$ satisfying $n_{d-1} < p^{ (t-1)/ (100d) }$.
Since $ t - (d-1) > 2(t-1)/(100 d) + d$, it follows that every entry of $\mathbf{a}_{d-1}$ is divisible by $p$.

We have
$$
(n_{d-1} n_d \ p^{d + (d-1)})  \alpha_{d, r} \ \frac{1}{p^{d}} \ \omega_{d, r: p, \mathbf{k}} (\mathbf{y}) \in \mathbb{Z}[y_1, \ldots, y_n] \ \  (1 \leq r \leq r_d),
$$
$$
(n_{d-1} n_d \ p^{d + (d-1)})  \alpha_{d-1, r} \ \frac{1}{p^{d-1}} \ \omega_{d-1, r: p, \mathbf{k}} (\mathbf{y}) \in \mathbb{Z}[y_1, \ldots, y_n]  \ \ (1 \leq r \leq r_{d-1}),
$$
and
$$
n_{d-1} n_d \ p^{d+(d-1)} \leq P^{Q \gamma_{d-1} + 2\varepsilon}  P^{Q \gamma_d} P^{2/ (800 d^2)}.
$$
With this set up, we can apply Corollary \ref{cor 15.1 in S+} again with $\ell = d-2$ and
\begin{eqnarray}
\notag
\theta &=& Q \gamma_{d-1} + \varepsilon_{d-1} + Q \gamma_d + \frac{2}{ 800 d^2 }
\\
&<&
Q \gamma_{d-1} + \varepsilon_{d-1} + Q \gamma_d + \frac{d}{ 800 d^2 }
\\
\notag
&<& \frac{1}{100 d}
\\
\notag
&<& \frac14,
\end{eqnarray}
where $\varepsilon_{d-1} > 0$ is sufficiently small.
At this point it is clear that we can repeat the process, in fact we continue in this manner until $\ell = 2$.
We remark that if there exists $\ell'$ with $r_{\ell'} = 0$, then we simply skip the case $\ell = \ell'$ during this process.
As a result, we obtain that every entry of $\mathbf{a}_{d}, \ldots, \mathbf{a}_{2}$ is divisible by $p$.
Thus we have $\gcd(\mathbf{a}_{\ell}, q) > 1$ for each $2 \leq \ell \leq d$, which is a contradiction.
Therefore, we obtain
\begin{eqnarray}
\Big{|} \sum_{ \mathbf{y} \in [0,p^{t-1}-1]^{n}   } e \left( \sum_{\ell = 1}^d \sum_{r=1}^{r_{\ell}}  f_{\ell,r} (\mathbf{k} + p \mathbf{y} )  \cdot a_{\ell,r} /q  \right) \Big{|}
&=&
\notag
\Big{|} \sum_{ \mathbf{y} \in [0, P]^{n}   } e \left( \sum_{\ell = 1}^d \sum_{r=1}^{r_{\ell}}  u_{\ell,r} (\mathbf{y} )  \cdot \alpha_{\ell,r}  \right) \Big{|}
\\
&\ll& P^{n - Q}
\notag
\\
&\ll& (p^{t-1})^{n - Q}.
\notag
\end{eqnarray}

Thus we can bound ~(\ref{widetile S part 3-1}) as follows
\begin{eqnarray}
\mathcal{S}_{ \mathbf{a}, q } &\leq&
\sum_{\mathbf{k} \in \mathbb{U}_p^n}
\Big{|} \sum_{ \mathbf{y} \in [0, p^{t-1}-1]^n   } e \left(  \sum_{\ell = 1}^d \sum_{r=1}^{r_{\ell}} f_{\ell,r} (\mathbf{k} + p \mathbf{y} ) \cdot  a_{\ell, r}/q  \right) \Big{|}
\notag
\\
&\ll&
p^n (p^{t-1})^{n - Q}
\notag
\\
&=&
p^{Q} q^{n - Q}.
\notag
\end{eqnarray}

\end{proof}

\begin{lem}
\label{to bound local factor+}
Let $p$ be a prime and let $q = p^t$, $t \in \mathbb{N}$. Let $\mathbf{a}  =(\mathbf{a}_d, \ldots, \mathbf{a}_1)  \in (\mathbb{Z} / q \mathbb{Z})^R$ with $\gcd(\mathbf{a},q)=1$. Furthermore, suppose $\gcd(\mathbf{a}_{\ell},q)>1$ for $2 \leq \ell \leq d$, and $\gcd(\mathbf{a}_{1},q) = 1$. Then we have
\begin{eqnarray}
\notag
\mathcal{S}_{ \boldsymbol{a}, q } \ll
\left\{
    \begin{array}{ll}
         q^{n-Q},
         &\mbox{if } t \leq 800 d^3 + 1 ,\\
         p^Q q^{n-Q},
         &\mbox{if } t > 800 d^3 + 1,
    \end{array}
\right.
\end{eqnarray}
where the implicit constants depend only on $n$ and the coefficients of $\mathbf{F}_1$, and in particular they are independent of $p$.
\end{lem}

\begin{proof}
First we consider the case $t> 1$.
Since $\gcd(\mathbf{a}_1,q)=1$, there exists $1 \leq r' \leq r_1$ such that $\gcd(a_{1,r'}, p) = 1$.
By our assumption on $\mathbf{f}$, we have
\begin{eqnarray}
\mathcal{S}_{ \mathbf{a}, q }
&=&  \sum_{\mathbf{k} \in (\mathbb{U}_q)^{n}} e \left( \sum_{\ell = 1}^d \sum_{r=1}^{r_{\ell}}  f_{\ell,r}(\mathbf{k})  \cdot a_{\ell,r} /q  \right)
\notag
\\
\notag
&=& \left( \prod_{1 \leq r \leq r_1} \sum_{ k_{n - r_1 + r} \in \mathbb{U}_q} e( c_{1,r} k_{n - r_1 + r} \cdot a_{1,r} /q)\right)
\sum_{ \widehat{\mathbf{k} } \in (\mathbb{U}_q)^{n-r_1}} e \Big{(} \sum_{1 \leq r \leq r_1} \widetilde{f}_{1,r}( \widehat{\mathbf{k} } ) \cdot a_{1,r} /q
\\
\notag
&+&   \sum_{\ell = 2}^d \sum_{r=1}^{r_{\ell}}  f_{\ell,r}( \widehat{\mathbf{k} } )  \cdot a_{\ell,r} /q  \Big{)}.
\end{eqnarray}
If $\gcd(p, c_{1,r'})=1$, then $\gcd( a_{1, r'} \cdot c_{1,r'}, \ q) = 1$. Consequently, we have
$$
\sum_{ k_{n - r_1 + r'} \in \mathbb{U}_q} e( (a_{1,r'} \cdot c_{1,r'}) k_{n - r_1 + r'}  /q) = \sum_{ k \in \mathbb{U}_q} e(  k  /q) =0,
$$
because $t>1$.
In this case, it follows that
$$
\mathcal{S}_{ \mathbf{a}, q } = 0.
$$
Otherwise, we have $p | c_{1,r'}$. Let $c_{1,r'} = p^{i_0} m_0$ where $p \nmid m_0$.
By a similar argument, we have
\begin{eqnarray}
\notag
\sum_{ k_{n - r_1 + r'} \in \mathbb{U}_q} e( (a_{1,r'} \cdot c_{1,r'}) k_{n - r_1 + r'}  /q) =
\left\{
    \begin{array}{lll}
         0,
         &\mbox{if }  t > i_0 + 1, \\
         - p^{i_0},
         &\mbox{if } t = i_0 + 1, \\
         \phi(q),
         &\mbox{if } t \leq i_0.
    \end{array}
\right.
\end{eqnarray}
Therefore, for all but finite possibilities (depending only on $c_{1, 1}, \ldots, c_{1, r_1})$ of $q$ 
we always have $\mathcal{S}_{ \mathbf{a}, q } = 0$.
Thus for $t>1$ we see that we can obtain the bounds in the statement of the lemma with the implicit constant depending only on  $c_{1, 1}, \ldots, c_{1, r_1}.$

In the case $t = 1$, since $e(m)=1$ for $m \in \mathbb{Z}$, we have by our hypothesis that
\begin{eqnarray}
\notag
\mathcal{S}_{ \mathbf{a}, p }
&=& \sum_{\mathbf{k} \in (\mathbb{U}_p)^{n}} e \left(  \sum_{r=1}^{r_{1}}  f_{1,r}(\mathbf{k})  \cdot a_{1,r} /p  \right).
\end{eqnarray}
We bound this sum in a similar manner as in Lemma \ref{to bound local factor}.
We apply the inclusion-exclusion principle  and obtain
\begin{eqnarray}
\label{eqn baaaahhhhh+}
&&  \sum_{\mathbf{k} \in (\mathbb{U}_p)^{n}} e \left(  \sum_{r=1}^{r_{1}}  f_{1,r}(\mathbf{k})  \cdot a_{1,r} /p  \right)
\\
&=&  \sum_{\mathbf{k} \in (\mathbb{Z} / p \mathbb{Z})^{n}} \prod_{i=1}^n \left( 1 -  \mathbf{1}_{k_i = 0}  \right) \
e \left( \sum_{r=1}^{r_{1}}  f_{1,r}(\mathbf{k})  \cdot a_{1,r} /p  \right)
\notag
\\
\notag
&=& \sum_{ I \subseteq \{ 1, 2, \ldots, n \}} (-1)^{|I|}
\sum_{\mathbf{k} \in (\mathbb{Z} / p \mathbb{Z})^{n} }
\mathfrak{F}_I(\mathbf{k}) \ e \left( \sum_{r=1}^{r_{1}}  f_{1,r} (\mathbf{k})  \cdot a_{1,r} /p  \right),
\notag
\end{eqnarray}
where
$$
\mathbf{1}_{k_i = 0} =
\left\{
    \begin{array}{ll}
         1,
         &\mbox{if } k_i = 0 ,\\
         0,
         &\mbox{if } k_i \not = 0,
    \end{array}
\right.
$$
and
$$
\mathfrak{F}_I(\mathbf{k}) = \prod_{i \in I} \mathbf{1}_{k_i = 0}.
$$
In other words, $\mathfrak{F}_I(\mathbf{k})$ is the characteristic function of the set $H_{I} = \{ \mathbf{k} \in (\mathbb{Z} / p \mathbb{Z})^{n} : k_i = 0 \ (i \in I)\}$.
We now bound the summand in the final expression of ~(\ref{eqn baaaahhhhh+}) by further considering two cases, $|I| \geq Q$ and $|I| < Q$.
In the first case $|I| \geq Q$, we use the following trivial estimate
\begin{eqnarray}
\Big{|}
\sum_{\mathbf{k} \in (\mathbb{Z} / p \mathbb{Z})^n}
\mathfrak{F}_I(\mathbf{k}) \ e \left( \sum_{r=1}^{r_1}  f_{1,r}(\mathbf{k})  \cdot a_{1,r} /p  \right) \Big{|}
\leq
p^{n  - |I|}
\leq p^{n  - Q}.
\notag
\end{eqnarray}
On the other hand, suppose  $|I| < Q$. Let us label $\mathbf{s} = (s_1, \ldots , s_{n - |I|} )$
to be the remaining variables of $\mathbf{x}$ after setting $x_i = 0$ for each $i \in I$.
For each $1 \leq r \leq r_{1}$, let
$$
\mathfrak{g}_{1,r}(\mathbf{s}) ={f}_{1,r}(\mathbf{x}) |_{x_i = 0 \ ( i \in I )},
$$
or equivalently the polynomial $\mathfrak{g}_{1,r}(\mathbf{s})$ is obtained by
substituting $x_i = 0 \ (i \in I)$ to the polynomial ${f}_{1,r}(\mathbf{x})$.
Thus $\mathfrak{g}_{1,r} (\mathbf{s})$ is a polynomial in $n - |I|$ variables.
Let us denote
$$
\mathfrak{g}_{1,r} (\mathbf{s}) = \sum_{i=1}^{n-|I|} c'_{r,i} \  s_i + c'_{r,0} \ \ \ (1 \leq r \leq r_1).
$$
With these notations, we have
\begin{eqnarray}
&&\sum_{\mathbf{k} \in (\mathbb{Z} / p \mathbb{Z})^{n}} \mathfrak{F}_I(\mathbf{k}) \ e \left( \sum_{r=1}^{r_{1}}  f_{1,r}(\mathbf{k})  \cdot a_{1,r} /p  \right)
\label{eqn baaaahhhhh}
\\
\notag
&=&
\sum_{\mathbf{s} \in [0, p-1]^{n - |I|} }e \left( \sum_{r=1}^{r_{1}}  \mathfrak{g}_{1,r} (\mathbf{s})  \cdot a_{1,r} /p  \right)
\\
&=&
e \left( \sum_{r=1}^{r_1} c'_{r,0} \ a_{1,r} /p \right)
\sum_{\mathbf{s} \in [0, p-1]^{n - |I|} }e \left( \sum_{r=1}^{r_{1}}  \sum_{i=1}^{n-|I|} c'_{r,i} s_i  \cdot a_{1,r} /p  \right)
\notag
\\
&=&
e \left( \sum_{r=1}^{r_1} c'_{r,0} \ a_{1,r} /p\right)
\sum_{\mathbf{s} \in [0, p-1]^{n - |I|} }e \left( \sum_{i=1}^{n-|I|}  \frac{s_i}{p}  \Big{(} \sum_{r=1}^{r_{1}}   c'_{r,i}  \cdot a_{1,r} \Big{)}   \right)_.
\notag
\end{eqnarray}
We can also deduce easily that the homogeneous linear portion of the polynomial $\mathfrak{g}_{1,r}(\mathbf{s})$, which we denote
$\mathfrak{G}_{1,r}(\mathbf{s}) = \sum_{i=1}^{n-|I|} c'_{r,i} \ s_i$, is obtained by substituting $x_i = 0 \ (i \in I)$ to ${F}_{1,r}(\mathbf{x})$.
Hence, we have
$$
\mathfrak{G}_{1,r}(\mathbf{s}) =  {F}_{1,r}(\mathbf{x}) |_{x_i = 0 \ (i \in I)}.
$$
It then follows by Lemma \ref{Lemma 3 in CM} that
$$
\mathcal{B}_1( \{\mathfrak{G}_{1,r}: 1 \leq r \leq r_{1} \} ) \geq \mathcal{B}_1( \mathbf{F}_{1} ) - |I| > \mathcal{B}_1( \mathbf{F}_{1} ) - Q >0.
$$
In particular, it follows that $\mathfrak{G}_{1,1}(\mathbf{s}), \ldots, \mathfrak{G}_{1,r_1}(\mathbf{s})$ are linearly independent over $\mathbb{Q}$.
Thus for $p$ sufficiently large with respect to the coefficients of $\mathbf{F}_1$, the coefficient matrix
of $\mathfrak{G}_{1,1}(\mathbf{s}),$ $\ldots,$ $\mathfrak{G}_{1,r_1}(\mathbf{s})$ has full rank modulo $p$.
Therefore, it follows that if
$$
\sum_{r=1}^{r_{1}}   c'_{r,i}  \cdot a_{1,r} \equiv 0 \ (\text{mod }p)
$$
for each $1 \leq i \leq n - |I|$, then it must be that $a_{1,1} \equiv \ldots \equiv a_{1,r_1} \equiv 0 \ (\text{mod }p)$.
Since we have $\gcd(\mathbf{a}_1, p) = 1$, this is a contradiction.
Thus without loss of generality suppose
$$
\varsigma = \sum_{r=1}^{r_{1}}   c'_{r,1}  \cdot a_{1,r} \not \equiv 0 \ (\text{mod }p).
$$
Then equation (\ref{eqn baaaahhhhh}) becomes
\begin{eqnarray}
&&\Big{|} \sum_{\mathbf{k} \in (\mathbb{Z} / p \mathbb{Z})^{n}} \mathfrak{F}_I(\mathbf{k}) \ e \left( \sum_{r=1}^{r_{1}}  f_{1,r}(\mathbf{k})  \cdot a_{1,r} /p  \right) \Big{|}
\notag
\\
&=&
\Big{|}
\left( \sum_{0 \leq s_1 \leq p-1} e (  \varsigma  s_1/ p  )  \right)
\sum_{ \substack{0 \leq s_i \leq p-1 \\ 2 \leq i \leq n - |I| } } e \left( \sum_{i=2}^{n-|I|}  \frac{s_i}{p}  \Big{(} \sum_{r=1}^{r_{1}}   c'_{r,i}  \cdot a_{1,r} \Big{)}   \right) \Big{|}
\notag
\\
&=& 0,
\notag
\end{eqnarray}
because
$$
\sum_{0 \leq s_1 \leq p-1} e (  \varsigma  s_1/ p  ) = \sum_{0 \leq s_1 \leq p-1} e (  s_1/ p  ) = 0.
$$
Therefore, by combining the estimates for the two cases $|I| \geq Q$ and $|I| < Q$, we obtain
$$
\mathcal{S}_{\mathbf{a}, p} \ll p^{n - Q},
$$
where the implicit constant depends only on $n$ and the coefficients of $\mathbf{F}_1$.
\end{proof}

By a similar argument as in \cite[Chapter VIII, \S 2, Lemma 8.1]{H}, one can show that $B(q)$
is a multiplicative function of $q$. We leave the proof of the following lemma as a basic exercise
involving the Chinese remainder theorem and manipulating summations.
\begin{lem}
Suppose $q,q' \in \mathbb{N}$ and $\gcd (q,q') = 1$. Then we have
$$
B(q q') = B(q) B(q').
$$
\end{lem}

Recall we defined the term $\mathfrak{S}(N)$ in (\ref{def sigular series partial}).
For each prime $p$, we define
\begin{equation}
\label{def mu p}
\mu(p) =  1  + \sum_{t=1}^{\infty} B( p^t),
\end{equation}
which converges absolutely under our assumptions on $\mathbf{f}$.
Furthermore, under our assumptions on $\mathbf{f}$
the following limit exists
\begin{equation}
\label{def sigular series}
\mathfrak{S}(\infty) := \lim_{N \rightarrow \infty} \mathfrak{S}(N) = \prod_{p \ \text{prime}} \mu(p),
\end{equation}
which is called the \textit{singular series}.
We prove these statements in the following Lemma \ref{singular series lemma}.

\begin{lem}
\label{singular series lemma}
There exists $\delta_1 > 0$ such that for each prime $p$, we have
$$
\mu(p) = 1 + O(p^{-1 - \delta_1}),
$$
where the implicit constant is independent of $p$.
Furthermore, we have
$$
\Big{|} \mathfrak{S}(N) -  \mathfrak{S}(\infty) \Big{|} \ll (\log N)^{-C \delta_2 }
$$
for some $\delta_2 > 0$.
\end{lem}
Therefore, the limit in ~(\ref{def sigular series}) exists, and the product in ~(\ref{def sigular series}) converges.
We leave the details  that these two quantities are equal to the reader.
\begin{proof}
Recall our choice of $Q$ satisfies
$$
Q > \frac{1 + R(800 d^3 + 2)}{800 d^3 + 1}
\ \
\text{  and  }
\ \
Q > \frac{R+1}{1 - \frac{1}{800 d^3 + 1}} > R+1.
$$
For any $t \in \mathbb{N}$, we know that $\phi(p^t) = p^t(1 - 1/p) \geq \frac12 p^t$.
Therefore, by considering the two cases as in the statements of Lemmas \ref{to bound local factor} and \ref{to bound local factor+}, we obtain
\begin{eqnarray}
| \mu(p) - 1 |
&\leq&
\sum_{1 \leq t \leq 800 d^3 + 1} \Big{|} \sum_{ \substack{ \gcd (\mathbf{a}, p^t) = 1 \\  \mathbf{a} \in (\mathbb{Z} / p^t \mathbb{Z} )^R } } \frac{1}{\phi(p^t)^n} \  \mathcal{S}_{ \mathbf{a}, p^t }  \Big{|}
+
\sum_{ t > 800 d^3  + 1} \Big{|} \sum_{ \substack{ \gcd (\mathbf{a}, p^t) = 1 \\  \mathbf{a} \in (\mathbb{Z} / p^t \mathbb{Z} )^R } }   \frac{1}{\phi(p^t)^n} \  \mathcal{S}_{ \mathbf{a}, p^t } \Big{|}
\notag
\\
&\ll&
\sum_{1 \leq t \leq 800 d^3 + 1} p^{tR} p^{-nt} p^{nt - t Q }
+
\sum_{ t > 800 d^3  + 1} p^{tR} p^{-nt} p^{Q + nt - t Q }
\notag
\\
&\ll&
p^{R - Q }
+
p^Q p^{-(800 d^3 + 2)(Q-R)}
\notag
\\
&\ll&
p^{-1 - \delta_1},
\notag
\end{eqnarray}
for some $\delta_1 > 0$. We note that the implicit constants in $\ll$ are independent of $p$ here.

Let $q = p_1^{t_1} \ldots p_{v}^{t_{v}}$ be the prime factorization of $q \in \mathbb{N}$.
Without loss of generality, suppose we have $t_j \leq  800 d^3 + 1 \ (1 \leq j \leq v_0)$ and
$t_j >  800 d^3  + 1 \ (v_0 < j \leq v)$. By a similar calculation as above
and the multiplicativity of $B(q)$, it follows that
\begin{eqnarray}
B(q)
&=&
B(p_1^{t_1}) \ \ldots \  B(p_{v}^{t_{v}})
\notag
\\
&\ll&
\left( \prod_{j=1}^{v_0} p_j^{t_j R} p_j^{-n t_j}  p_j^{t_j (n - Q)} \right) \cdot \left(  \prod_{j=v_0 + 1}^{v}  p_j^{t_j R} p_j^{-n t_j} p_j^{Q} p_j^{t_j (n - Q)} \right)
\notag
\\
&=&
q^{R - Q} \cdot  \left(  \prod_{j=v_0 + 1}^{v}  p_j^{ Q } \right)
\notag
\\
&\leq&
q^{R - Q} \cdot  q^{Q/(800 d^3 + 1)}
\notag
\\
&\leq&
q^{- 1 - \delta_2},
\notag
\end{eqnarray}
for some $\delta_2 > 0$. We note that the implicit constant in $\ll$ is independent of $q$ here.
Therefore, we obtain
\begin{eqnarray}
\Big{|} \mathfrak{S}(N) -  \mathfrak{S}(\infty) \Big{|}
&\leq&
 \sum_{q > (\log N)^C } |  B(q) |
\notag
\\
&\ll&
\sum_{q > (\log N)^C } q^{- 1 - \delta_2}
\notag
\\
&\ll&
(\log N)^{- C \delta_2 }.
\notag
\end{eqnarray}
\end{proof}

Let $\nu_t(p)$ denote the number of solutions $\mathbf{x} \in (\mathbb{U}_{p^t})^n$
to the congruence relations
\begin{eqnarray}
f_{\ell, r}( \mathbf{x} ) \equiv 0 \  (\text{mod } p^t) \ \ (1 \leq \ell \leq d, 1 \leq r \leq r_{\ell}).
\end{eqnarray}
Then using the fact that
\begin{eqnarray}
\notag
\sum_{ {a} \in \mathbb{Z}/ p^t \mathbb{Z} } e \left(  m  \cdot {a}/p^t  \right) =
\left\{
    \begin{array}{ll}
         p^{t},
         &\mbox{if } p^t | m ,\\
         0,
         &\mbox{otherwise,}
    \end{array}
\right.
\end{eqnarray}
we deduce
\begin{eqnarray}
&&1 + \sum_{j=1}^t B(p^j)
\notag
\\
&=&
1 + \sum_{j=1}^t  \frac{1}{\phi(p^j)^n} \sum_{\mathbf{k} \in (\mathbb{U}_{p^j})^{n}}
\sum_{ \substack{ \gcd (\mathbf{a}, p^j) = 1 \\  \mathbf{a} \in (\mathbb{Z} / p^j \mathbb{Z} )^R } }
e \left(\sum_{\ell = 1}^d \sum_{r=1}^{r_{\ell}}  f_{\ell,r} (\mathbf{k})  \cdot a_{\ell,r} /p^j \right)
\notag
\\
&=& \frac{1}{\phi(p^t)^n} \sum_{\mathbf{k} \in (\mathbb{U}_{p^t})^{n}}   
\sum_{   \mathbf{a} \in (\mathbb{Z} / p^t\mathbb{Z} )^R }
e \left(\sum_{\ell = 1}^d \sum_{r=1}^{r_{\ell}}  f_{\ell,r} (\mathbf{k})  \cdot a_{\ell,r} /p^t \right)
\notag
\\
&=&  \frac{p^{tR}}{\phi(p^t)^n } \ \nu_t(p).
\notag
\end{eqnarray}
Therefore, under our assumptions on $\mathbf{f}$ we obtain
$$
\mu(p) = \lim_{t \rightarrow \infty}  \frac{  p^{tR} \ \nu_t(p) }{ \phi(p^t)^n }.
$$
We can then deduce by an application of  Hensel's lemma
that
$$
\mu(p) > 0,
$$
if the system ~(\ref{set of eqn 3}) has a non-singular solution in $\mathbb{Z}_p^{\times}$, the units of $p$-adic integers.
The details are left to the reader. From this it follows in combination with Lemma \ref{singular series lemma} that if the system ~(\ref{set of eqn 3}) has a non-singular solution in $\mathbb{Z}_p^{\times}$
for every prime $p$, then
\begin{equation}
\label{mu p positive}
\mathfrak{S}(\infty) = \prod_{p \ \text{prime}}\mu(p) > 0.
\end{equation}
By combining Lemmas \ref{lemma major arc estimate} and \ref{singular series lemma}, we obtain the following.
\begin{prop}
Let $\mathbf{f}$ be the polynomials in (\ref{set of eqn 3}).
Given any $c>0$, for sufficiently large $C>0$ we have
$$
\int_{\mathfrak{M}(C)} T(\mathbf{f}; \boldsymbol{\alpha} ) \ \mathbf{d} \boldsymbol{\alpha}
=  \mathfrak{S}(\infty) \mu(\infty) \  X^{n - \sum_{\ell = 1}^d \ell r_{\ell}  } + O\left( \frac{ X^{n - \sum_{\ell = 1}^d \ell r_{\ell}  } }{(\log X)^{c}} \right).
$$
\end{prop}
We note this proposition contains Proposition \ref{prop major arc bound} as a special case
with
\begin{equation}
\label{c f equals product}
\mathcal{C}(\mathbf{f}) = \mathfrak{S}(\infty) \mu(\infty).
\end{equation}
\section{Conclusions and further remarks}
\label{sec concln}
Let us refer to the polynomials in (\ref{set of eqn 1}) as $\mathbf{f}$,
and the polynomials in (\ref{set of eqn 2}) as $\boldsymbol{\mathfrak{f}}$ in this section. We let $\mathbf{F}$ and $\boldsymbol{\mathfrak{F}}$
be the systems of the highest degree homogeneous portions of $\mathbf{f}$ and $\boldsymbol{\mathfrak{f}}$, respectively.

As a consequence of Propositions  \ref{prop minor arc bound} and \ref{prop major arc bound}, we obtain the following asymptotic formula
for the system of equations (\ref{set of eqn 2}). We have that given any $c > 0$, there exists $C>0$ such that
\begin{eqnarray}
\mathcal{M}_{\boldsymbol{\mathfrak{f}}}(X)
&=& \int_{\mathbb{T}^R} T( \boldsymbol{\mathfrak{f}} ; \boldsymbol{\alpha} ) \ \mathbf{d}\boldsymbol{\alpha}
\label{final asymptotic}
\\
&=& \int_{\mathfrak{M}(C) }T( \boldsymbol{\mathfrak{f}} ; \boldsymbol{\alpha} )  \ \mathbf{d}\boldsymbol{\alpha} + \int_{\mathfrak{m}(C)} T( \boldsymbol{\mathfrak{f}}; \boldsymbol{\alpha} ) \ \mathbf{d} \boldsymbol{\alpha}
\notag
\\
&=&  \mathcal{C}(\boldsymbol{\mathfrak{f}}) \ X^{n - \sum_{\ell = 1}^d \ell r_{\ell}  } + O \left(  \frac{X^{n- \sum_{\ell = 1}^d \ell r_{\ell}}}{ (\log X)^c }  \right),
\notag
\end{eqnarray}
which proves Theorem \ref{main thm} for $\boldsymbol{\mathfrak{f}}$. 

Recall from Section \ref{sec initial set up} that transforming the system
$\mathbf{f}$  into $\boldsymbol{\mathfrak{f}}$ does not affect its solution set, in other words $V_{\mathbf{f}, \mathbf{0}} (\mathbb{Z})= V_{\boldsymbol{\mathfrak{f}}, \mathbf{0}} (\mathbb{Z}) $. Therefore, we in fact have
$$
\mathcal{M}_{\mathbf{f} }(X) = \mathcal{M}_{ \boldsymbol{\mathfrak{f}} }(X) = \mathcal{C}(\boldsymbol{\mathfrak{f}}) \ X^{n - \sum_{\ell = 1}^d \ell r_{\ell}  } + O \left(  \frac{X^{n- \sum_{\ell = 1}^d \ell r_{\ell}}}{ (\log N)^c }  \right).
$$
Since $\mathcal{C}(\boldsymbol{\mathfrak{f}})$ is a constant dependent only on $\boldsymbol{\mathfrak{f}}$, in turn it follows that it is a constant which depends only on $\mathbf{f}$. Thus by setting $\mathcal{C}(\mathbf{f}) = \mathcal{C}(\boldsymbol{\mathfrak{f}})$, we have obtained Theorem \ref{main thm}.

We also remark that if $V_{\mathbf{F}, \mathbf{0}}(\mathbb{R})$ has a non-singular real point in $(0,1)^n$,
then so does $V_{\boldsymbol{\mathfrak{F}}, \mathbf{0}}(\mathbb{R})$, and if the system of equations (\ref{set of eqn 1}) has a non-singular solution
in $\mathbb{Z}_p^{\times}$ for every prime $p$, then so does the system (\ref{set of eqn 2}). Under these conditions, it follows from (\ref{mu infty positive}), (\ref{mu p positive}), and (\ref{c f equals product}) that $\mathcal{C}(\mathbf{f}) = \mathcal{C}( \boldsymbol{\mathfrak{f}} ) > 0$.
We leave the details here to the reader.

Finally, we followed \cite{CM} and used the von-Mangoldt function $\Lambda$ as our weight for the exponential sum. Consequently,
$\mathcal{M}_{\mathbf{f} }(X)$ counts the number of solutions, with a logarithmic weight, to the equations $\mathbf{f} = \mathbf{0}$ whose coordinates are all prime powers.
Let $\mathbf{1}_{\mathcal{P}}$ denote the characteristic function of the set of prime numbers. For $\mathbf{x} = (x_1, \ldots, x_n)$,
we let $\mathbf{1}_{\mathcal{P}}(\mathbf{x}) = \mathbf{1}_{\mathcal{P}}(x_1) \ldots \mathbf{1}_{\mathcal{P}}(x_n)$ and $\log (\mathbf{x}) = \log(x_1) \ldots \log(x_n)$. Let us define
$$
\mathcal{M}'_{\mathbf{f}}(X) := \sum_{\mathbf{x} \in [0, X]^n} \log (\mathbf{x}) \ \mathbf{1}_{\mathcal{P}}(\mathbf{x})  \  \mathbf{1}_{V_{\mathbf{f}, \mathbf{0}}(\mathbb{C})} (\mathbf{x})
$$
with the convention that $\log (\mathbf{x})  \mathbf{1}_{\mathcal{P}}(\mathbf{x}) = 0$ if $x_i = 0$ for some $1 \leq i \leq n$.
The quantity $\mathcal{M}'_{\mathbf{f} }(X)$ counts the number of prime solutions, with a logarithmic weight, to the equations $\mathbf{f} = \mathbf{0}$.
We record the following result for $\mathcal{M}'_{\mathbf{f} }(X)$.
\begin{thm}
\label{main thm 2}
Under the same hypotheses as in Theorem \ref{main thm}, we have
$$
\mathcal{M}'_{\mathbf{f}}(X) = \mathcal{C}(\mathbf{f}) \ X^{n - \sum_{\ell = 1}^d \ell r_{\ell}  } + O \left(  \frac{X^{n- \sum_{\ell = 1}^d \ell r_{\ell}}}{ (\log N)^c }  \right),
$$
where $\mathcal{C}(\mathbf{f})$ is the same constant as in the statement of Theorem \ref{main thm}.
\end{thm}
We can obtain this asymptotic formula by changing the weight from $\Lambda (\mathbf{x})$ to  $\log (\mathbf{x})  \mathbf{1}_{\mathcal{P}}(\mathbf{x})$ in the proof of Theorem \ref{main thm}.
Since the resulting changes in the proof are minimal, we leave the details to the reader.
\appendix

\section{Proofs of the results in Section \ref{sec tech 1}}
\label{appendix B}
In this appendix, we provide proof for the results presented in Section \ref{sec tech 1}.
Let us denote $\mathfrak{B}_1 = [-1,1]^n$ and $\mathfrak{B}_0 = [0,1]^n$.
Let $\mathbf{x} = (x_1, \ldots, x_n)$ and $\mathbf{x}_j = (x_{j,1}, \ldots, x_{j,n})$ for $j \geq 1$.
Given a function $G(\mathbf{x})$, we define
$$
\Gamma_{\ell, G} (\mathbf{x}_1, \ldots, \mathbf{x}_{\ell}) = \sum_{t_1=0}^1 \ldots \sum_{t_{\ell}=0}^1 (-1)^{t_1 + \ldots + t_{\ell}} \
G( t_1 \mathbf{x}_1 + \ldots+ t_{\ell} \mathbf{x}_{\ell} ).
$$
Then it follows that $\Gamma_{{\ell}, G}$ is symmetric in its ${\ell}$ arguments, and that
$\Gamma_{{\ell}, G} (\mathbf{x}_1, \ldots,\mathbf{x}_{{\ell}-1}, \mathbf{0}) = 0$ \cite[Section 11]{S}.
It is clear from the definition that if $G'(\mathbf{x})$ is another function, then
$\Gamma_{{\ell}, G} + \Gamma_{{\ell}, G'} = \Gamma_{{\ell}, G + G'}.$
We also have that if $G$ is a form of degree $d$ and ${\ell} > d > 0$, then $\Gamma_{{\ell}, G}= 0$ \cite[Lemma 11.2]{S}.

For $\alpha \in \mathbb{R}$, let $\| \alpha \|$ denote the
distance from $\alpha$ to the closest integer. Let $\boldsymbol{\alpha} = (\boldsymbol{\alpha}_d, \ldots, \boldsymbol{\alpha}_1) \in \mathbb{R}^R$,
where $R = r_1 + \ldots + r_d$ and $\boldsymbol{\alpha}_{\ell} = (\alpha_{\ell, 1}, \ldots , \alpha_{\ell, r_{\ell}} ) \in \mathbb{R}^{r_{\ell}}$ $(1 \leq \ell \leq d)$.
We define
$$
\| \boldsymbol{\alpha} \| = \max_{ \substack{  1 \leq \ell \leq d  \\ 1 \leq r \leq r_{\ell}}} \| \alpha_{\ell, r} \|
\ \ \ \text{ and } \ \ \
| \boldsymbol{\alpha} | = \max_{ \substack{  1 \leq \ell \leq d  \\ 1 \leq r \leq r_{\ell}}} | \alpha_{\ell, r} |.
$$

We have the following standard results related to Weyl differencing.
\begin{lem}\cite[Lemma 13.1]{S}
\label{Lemma 13.1 in S}
Suppose $G(\mathbf{x}) = G^{(0)} + G^{(1)}(\mathbf{x} ) + \ldots + G^{(d)}(\mathbf{x} )$, where $G^{(j)}$ is a form of degree $j$
with real coefficients $(1 \leq j \leq d)$ and $G^{(0)} \in \mathbb{R}$. Let $P > 1$, and put
$$
S' = S'( G, P , \mathfrak{B}_0 ) := \sum_{\mathbf{x} \in P \mathfrak{B}_0 \cap \mathbb{Z}^n } e(G(\mathbf{x})).
$$
Let $\mathbf{e}_1, \ldots , \mathbf{e}_{n}$ be the standard basis vectors of $\mathbb{R}^n$.
Let $\varepsilon > 0$ and $2 \leq \ell \leq d $. If $\ell =d$, then let $\theta = 0$ and $q=1$. On the other hand,
if $2 \leq \ell < d$, then suppose $0 \leq \theta < 1/4$ and that there is $q \in \mathbb{N}$ with
$$
q \leq P^{\theta} \text{  and  }  \|  q G^{(j)} \| \leq c P^{\theta - j}  \ \  (\ell < j \leq d).
$$
Then we have
$$
|S'|^{2^{\ell -1}} \ll P^{ (2^{\ell-1} -\ell + 2 \theta )n + \varepsilon} \sum \left(
\prod_{i=1}^n \min ( P^{1 - 2 \theta}, \| q  \Gamma_{\ell, G^{(\ell)}} ( \mathbf{x}_1, \ldots , \mathbf{x}_{\ell-1}, \mathbf{e}_i  ) \|^{-1}   )
 \right),
$$
where the sum $\sum$ is over $(\ell-1)$-tuples of integer points $\mathbf{x}_1, \ldots , \mathbf{x}_{\ell-1}$ in $P \mathfrak{B}_1$,
and the implicit constant in $\ll$ depends only on $n,d$, $c$, and $\varepsilon$.
\end{lem}
We remark that the term $c$ which appears in the statement of Lemma \ref{Lemma 13.1 in S} is not present in the statement of \cite[Lemma 13.1]{S}.
However, it can be seen from the proof of \cite[Lemma 13.1]{S} that this change does not affect the result, or see the explanation given in \cite[pp. 275, line 5]{S}.

\begin{lem}\cite[Lemma 14.2]{S}
\label{Lemma 14.2 in S}
Make all the assumptions of Lemma \ref{Lemma 13.1 in S}. Suppose further that
$$
|S'| \geq P^{n-Q}
$$
where $Q>0.$ Let $\eta > 0$ and $\eta + 4 \theta \leq 1$.
Then the number $N(\eta)$ of integral $(\ell - 1)$-tuples
$$
\mathbf{x}_1, \ldots, \mathbf{x}_{\ell - 1} \in P^{\eta} \mathfrak{B}_1
$$
with
$$
\| q \Gamma_{\ell, G^{(\ell)}} ( \mathbf{x}_1, \ldots , \mathbf{x}_{\ell -1}, \mathbf{e}_i  ) \| < P^{-\ell + 4 \theta +  (\ell - 1) \eta } \ (i=1,\ldots, n)
$$
satisfies
$$
N(\eta) \gg P^{ n(\ell - 1)\eta - 2^{\ell - 1} Q - \varepsilon },
$$
where the implicit constant in $\gg$ depends only on $n,d, c, \eta,$ and $\varepsilon$.
\end{lem}

Let $\mathbf{u} = ( \mathbf{u}_{d}, \ldots, \mathbf{u}_{1} )$ be a system of polynomials in $\mathbb{Q}[x_1, \ldots, x_n]$, where
$\mathbf{u}_{\ell}  = ( u_{\ell,1}, \ldots, u_{\ell, r_{\ell}} )$ is the subsystem of degree $\ell$ polynomials of $\mathbf{u}$ $(1 \leq \ell \leq d)$.
We let $\mathbf{U} = ( \mathbf{U}_{d}, \ldots, \mathbf{U}_{1} )$ be the system of forms, where for each $1 \leq \ell \leq d$,
$\mathbf{U}_{\ell} = ( U_{\ell,1}, \ldots, U_{\ell, r_{\ell}} )$
and $U_{\ell, r}$ is the degree $\ell$ portion of $u_{\ell, r}$ $(1 \leq r \leq r_{\ell} )$.
We define the following exponential sum associated to $\mathbf{u}$,
\begin{equation}
\label{def of S 1}
S( \boldsymbol{\alpha}) = S( \mathbf{u},  \mathfrak{B}_0 ;\boldsymbol{\alpha}) := \sum_{\mathbf{x} \in P \mathfrak{B}_0 \cap \mathbb{Z}^n}
e \left( \sum_{1 \leq \ell \leq d} \sum_{ 1 \leq r \leq r_{\ell} } {\alpha}_{\ell, r} \cdot {u}_{\ell, r}  (\mathbf{x})  \right).
\end{equation}

Let $\mathbf{e}_1, \ldots, \mathbf{e}_n$ be the standard basis vectors of $\mathbb{C}^n$. Let $1 < \ell \leq d$.
We define $\mathbb{M}_{\ell} = \mathbb{M}_{\ell} (\mathbf{U}_{\ell}) $ to be the set of $(\ell-1)$-tuples $(\mathbf{x}_1, \ldots, \mathbf{x}_{\ell-1} ) \in (\mathbb{C}^n)^{\ell-1}$ for which the matrix
\begin{equation}
\label{def of MU}
[m_{r i}] = [ \Gamma_{\ell, U_{\ell, r}} ( \mathbf{x}_1, \ldots , \mathbf{x}_{d-1}, \mathbf{e}_i ) ] \ \ \ \  (1 \leq r \leq r_{\ell}, 1\leq i \leq n)
\end{equation}
has rank strictly less than $r_{\ell}$. For $R_0>0$, we denote $z_{R_0} (\mathbb{M}_{\ell})$ to be the number of integer points $(\mathbf{x}_1, \ldots, \mathbf{x}_{\ell-1} )$ on
$\mathbb{M}_{\ell}$ such that $$\max_{1 \leq i \leq \ell -1} \max_{1 \leq j \leq n}  | x_{ij} | \leq R_0.$$

Given a degree $\ell$ polynomial
$$
u(\mathbf{x}) = \sum_{ \substack{ i_j \in \mathbb{N} \cup \{0\}  (1 \leq j \leq n) \\0 \leq i_1 +  \ldots +  i_n \leq \ell   }} A_{i_1, \ldots, i_n } x_1^{i_1} \ldots  x_n^{i_n}
$$
with real coefficients, we denote
$$
|u| = \max_{\substack{ i_j \in \mathbb{N} \cup \{0\}  (1 \leq j \leq n) \\ 0 \leq i_1 +  \ldots +  i_n \leq \ell
} } |A_{i_1, \ldots, i_n }|
\ \ \
\text{ and }
\ \ \
\| u \| = \max_{\substack{ i_j \in \mathbb{N} \cup \{0\}  (1 \leq j \leq n) \\ 0 \leq i_1 +  \ldots +  i_n \leq \ell   } } \| A_{i_1, \ldots, i_n } \|.
$$
\begin{lem}\cite[Lemma 11.3]{S} \label{lemma 11.3 in S}
Suppose $U(\mathbf{x})$ is a form of degree $\ell$. Then we have
$$
\|  \Gamma_{\ell, U} \| \leq 2^{\ell} \ell^{\ell} \ \|  U  \|. 
$$
\end{lem}
By a similar proof as in \cite[Lemma 11.3]{S}, we can also show that for a degree $\ell$ form $U(\mathbf{x})$ the following holds
\begin{equation}
\label{ineq coeff of Gam}
| \Gamma_{\ell, U } | \leq 2^{\ell} \ell^{\ell} |U|.
\end{equation}

Let $1 < \ell \leq d$ and $r_{\ell} > 0$. We define $g_{\ell}( \mathbf{U}_{\ell} )$
to be the largest real number such that
\begin{equation}
\label{def gd}
z_P(\mathbb{M}_{\ell}) \ll P^{n({\ell}-1) - g_{\ell}( \mathbf{U}_{\ell} ) + \varepsilon}
\end{equation}
holds for each $\varepsilon >0$. It was proved in \cite[pp. 280, Corollary]{S} that
\begin{equation}
\label{h and g}
h_{\ell}( \mathbf{U}_{\ell} ) < \frac{\ell!}{  (\log 2)^{\ell} }  \left( g_{\ell}( \mathbf{U}_{\ell}  ) + ({\ell}-1)r_{\ell} (r_{\ell} - 1)  \right).
\end{equation}
Let
$$
\gamma_{\ell} = \frac{2^{{\ell}-1} ({\ell}-1) r_{\ell}}{ g_{\ell}( \mathbf{U}_{\ell} ) }
$$
when $r_{\ell} >0$ and $g_{\ell}( \mathbf{U}_{\ell} ) > 0$.
We let
$\gamma_{\ell} = 0$ if $r_{\ell} = 0$, and let
$\gamma_{\ell} = + \infty$ if $r_{\ell} > 0$ and $g_{\ell}( \mathbf{U}_{\ell} ) = 0$.
For $\ell$ with $r_{\ell}>0$, we also define
\begin{equation}
\label{def gamma'}
\gamma'_{\ell} = \frac{ 2^{{\ell}-1} }{ g_{\ell}( \mathbf{U}_{\ell} ) } = \frac{ \gamma_{\ell} }{ ({\ell}-1) r_{\ell} }.
\end{equation}

We have to deal with the cases when the coefficients of  $\mathbf{u}$ may depend on $P$. There are essentially two different
scenarios we have to consider, the first of which we refer to as follows.
\newline

Condition $(\star')$: The polynomials of $\mathbf{u}$ have coefficients in $\mathbb{Z}$, and the coefficients of $\mathbf{U}$ do not depend on $P$. However, given $u^{}_{\ell, r}(\mathbf{x})$ $(1 \leq \ell \leq d, 1 \leq r \leq r_{\ell})$ the coefficients of its monomials whose
degrees are strictly less than $\ell$ may depend on $P$.
\newline

The following lemma is essentially \cite[Lemma 15.1]{S}. The point here is that
if we are only considering the case $\ell = d$, then the implicit constants may depend
on $\mathbf{U}_d$ but not on $\mathbf{u}$ (Note for the case $\ell < d$ the implicit constants may depend
on $\mathbf{u}$, see \cite[Lemma 2.2]{Y}).

\begin{lem}\cite[Lemma 15.1]{S}
\label{Lemma 15.1 in S'}
Suppose $\mathbf{u}$ satisfies Condition $(\star')$.
Let $Q > 0$, $\varepsilon >0$, and let $P$ be sufficiently large with respect
to $d$ and $r_d, \ldots, r_1$.
Let $S( \boldsymbol{\alpha})$ be the sum associated to $\mathbf{u}$ as in ~(\ref{def of S 1}).
Given $0 < \eta \leq 1$, one of the following three alternatives must hold:

$(i)$ $|  S( \boldsymbol{\alpha})  | \leq P^{n-Q}$.

$(ii)$ There exists $n_0 \in \mathbb{N}$ such that
$$
n_0 \ll P^{r_{d}(d-1) \eta} \text{  and  } \|  n_0  \boldsymbol{\alpha}_{d} \| \ll P^{ -d + r_{d} (d-1) \eta}.
$$

$(iii)$ $ z_{R_0} (\mathbb{M}_{\ell}) \gg R_0^{ (d-1)n - 2^{d-1}(Q/ \eta) - \varepsilon}$
holds with $R_0 = P^{\eta}$.
\newline
\newline
The implicit constants
depend at most on $n,d, r_d, \eta, \varepsilon$, and $\mathbf{U}_{d}$.
\end{lem}
\begin{proof} We have $\boldsymbol{\alpha} \in \mathbb{R}^R$.
Let us denote
$$
\sum_{\ell = 1}^d \sum_{r=1}^{r_{\ell}} \alpha_{\ell, r} u_{\ell, r} (\mathbf{x}) = G^{(0)}+ G^{(1)}(\mathbf{x}) + \ldots +  G^{(d)}(\mathbf{x}),
$$
where
$G^{(j)}$ is a form of degree $j \ (1 \leq  j \leq d)$ and $G^{(0)} \in \mathbb{R}$.
Then it is clear that $G^{(d)}(\mathbf{x}) = \sum_{r=1}^{r_d} \alpha_{d,r} U_{d, r}(\mathbf{x})$,
and it depends on $\mathbf{U}_{d}$ only, and not on $\mathbf{u}$. With this observation,
by following through the proof of \cite[Lemma 15.1]{S} for the case $\ell = d$ while keeping track
of the constant dependency, we obtain the result.
\end{proof}

From Lemma \ref{Lemma 15.1 in S'},  we obtain the following corollary in a similar manner as in \cite[pp.276, Corollary]{S}.
\begin{cor}\cite[pp.276, Corollary]{S}
\label{cor 15.1 in S'}
Suppose $\mathbf{u}$ satisfies Condition $(\star')$.
Let $S( \boldsymbol{\alpha}) $ be the sum associated to $\mathbf{u}$ as in ~(\ref{def of S 1}).
Suppose $\varepsilon' > 0$ is sufficiently small and $Q>0$ satisfies
$$
Q \gamma'_d < 1.
$$
Then one of the following two alternatives must hold:

$(i)$ $|  S( \boldsymbol{\alpha})  | \leq P^{n-Q}$.

$(ii)$ There exists $n_0 \in \mathbb{N}$ such that
$$
n_0 \ll P^{Q \gamma_d + \varepsilon'} \text{  and  } \|  n_0 \boldsymbol{\alpha}_d \| \ll P^{ -d + Q \gamma_d + \varepsilon'}.
$$
The implicit constants depend at most on $n, d, r_d, \varepsilon', Q$, and $\mathbf{U}_{d}$.
\end{cor}

Now we move on to our next scenario of when the coefficients of $\mathbf{u}$ may depend on $P$.
Let $u^{(j)}_{\ell,r}(\mathbf{x})$ be the homogeneous degree $j$ portion of the polynomial $u_{\ell,r}(\mathbf{x})$.
In the following lemma, for $j < \ell$ the coefficients of $u^{(j)}_{\ell,r}(\mathbf{x})$ may be in $\mathbb{Q}$ and also depend on $P$, but in a controlled manner. On the other hand, the coefficients of $U_{\ell,r}(\mathbf{x})$ do not depend on $P$.
We also note the implicit constants may depend on $\mathbf{U}$ but not on $\mathbf{u}$

\begin{lem}\cite[Lemma 15.1]{S}
\label{Lemma 15.1 in S''}
Suppose $\mathbf{u}$ has coefficients in $\mathbb{Q}$, and further suppose $\mathbf{U}$ has coefficients in $\mathbb{Z}$.
Let $Q > 0$ and $\varepsilon >0$. Let $2 \leq \ell \leq d$ with $r_{\ell} > 0$.
If $\ell =d$, then let $\theta = 0$ and $q=1$.
On the other hand, if $2 \leq \ell < d$, then suppose $0 \leq \theta < 1/4$ and that there is $q \in \mathbb{N}$ with
$$
q \leq P^{\theta}, \ \ \    q \boldsymbol{\alpha}_{\ell'} \in \mathbb{Z}^{r_{\ell'}}  \ \ (\ell < \ell' \leq d),
$$
and
$$
q \alpha_{j, r} u^{(\ell')}_{j, r}(\mathbf{x}) \in \mathbb{Z}[x_1, \ldots, x_n]
$$
for every $\ell < j \leq d, 0 \leq \ell' < j, 1 \leq r \leq r_{j}$.

Let $S( \boldsymbol{\alpha})$ be the sum associated to $\mathbf{u}$ as in ~(\ref{def of S 1}). 
Given $\eta > 0$ with $\eta + 4 \theta \leq 1$, one of the following three alternatives must hold:

$(i)$ $|  S( \boldsymbol{\alpha})  | \leq P^{n-Q}$.

$(ii)$ There exists $n_0 \in \mathbb{N}$ such that
$$
n_0 \ll P^{r_{\ell}(\ell-1) \eta} \text{  and  } \|  q n_0  \boldsymbol{\alpha}_{\ell} \| \ll P^{ -\ell + 4 \theta + r_{\ell} (\ell-1) \eta}.
$$

$(iii)$ $ z_{R_0} (\mathbb{M}_{\ell}) \gg R_0^{ (\ell-1)n - 2^{\ell-1}(Q/ \eta) - \varepsilon}$
holds with $R_0 = P^{\eta}$.
\newline
\newline
The implicit constants depend at most on $n,d, r_d, \ldots, r_1, \eta, \varepsilon$, and $\mathbf{U}$.
\end{lem}

\begin{proof}
We have $\boldsymbol{\alpha} \in \mathbb{R}^{R}$.
Let us denote
$$
\sum_{\ell = 1}^d \sum_{r=1}^{r_{\ell}} \alpha_{\ell, r} u_{\ell, r} (\mathbf{x}) = G^{(0)}+ G^{(1)}(\mathbf{x}) + \ldots +  G^{(d)}(\mathbf{x}),
$$
where
$G^{(\ell')}$ is a form of degree $\ell' \ (1 \leq  \ell' \leq d)$ and $G^{(0)} \in \mathbb{R}$.
Then it is clear that $G^{(d)}(\mathbf{x}) = \sum_{r=1}^{r_d} \alpha_{d,r} U_{d, r}(\mathbf{x})$.
Recall we denote $u^{(\ell')}_{j,r}(\mathbf{x})$ to be  the homogeneous degree $\ell'$ portion of the polynomial $u_{j,r}(\mathbf{x})$.
Then we have
$$
G^{(\ell')}(\mathbf{x}) = \sum_{r=1}^{r_{\ell'}} \alpha_{\ell', r} U_{\ell', r} (\mathbf{x}) + \sum_{j = \ell' + 1}^d \sum_{r=1}^{r_{j}} \alpha_{j, r} u^{(\ell')}_{j, r} (\mathbf{x}) \ \ (1 \leq \ell' < d).
$$
If $\ell< d$, then it is clear from our hypothesis that we have
$$
\| q G^{(\ell')}\| = 0 \leq  P^{\theta - \ell'}
$$ for each $\ell < \ell' \leq d$. 

Suppose the alternative $(i)$ fails. In this case, we may apply Lemma \ref{Lemma 14.2 in S} and obtain that the number
$N(\eta)$ of integral $(\ell-1)$-tuples $\mathbf{x}_1$, \ldots, $\mathbf{x}_{\ell-1}$ in $P^{\eta} \mathfrak{B}_1$ with
\begin{equation}
\label{(15.2) in S'}
\| q  \Gamma_{\ell, G^{(\ell)}} ( \mathbf{x}_1, \ldots , \mathbf{x}_{\ell-1}, \mathbf{e}_i  ) \| < P^{-\ell + 4 \theta + (\ell-1) \eta } \ (i=1,\ldots, n)
\end{equation}
satisfies
$$
N(\eta) \gg R_0^{n(\ell -1) - 2^{\ell-1} (Q/ \eta) - \varepsilon },
$$
where $R_0 = P^{\eta}$, and the implicit constant in $\gg$ depends only on $n,d, \eta,$ and $\varepsilon$.
We have
\begin{eqnarray}
&& \| q  \Gamma_{\ell, G^{(\ell)}} ( \mathbf{x}_1, \ldots , \mathbf{x}_{\ell-1}, \mathbf{e}_i  ) \|
\notag
\\
&=& \| \sum_{r=1}^{r_{\ell}} q \alpha_{\ell, r} \Gamma_{\ell,  U_{\ell, r}}  ( \mathbf{x}_1, \ldots , \mathbf{x}_{\ell-1}, \mathbf{e}_i  )
+ \sum_{j = \ell+1}^d \sum_{r=1}^{r_{j}} q \alpha_{j, r} \Gamma_{ \ell, u^{(\ell)}_{j, r}  }  ( \mathbf{x}_1, \ldots , \mathbf{x}_{\ell-1}, \mathbf{e}_i  ) \|
\notag
\\
&=& \| \sum_{r=1}^{r_{\ell}} q \alpha_{\ell, r} \Gamma_{\ell,  U_{\ell, r}}  ( \mathbf{x}_1, \ldots , \mathbf{x}_{\ell-1}, \mathbf{e}_i  )
\|,
\notag
\end{eqnarray}
because
\begin{eqnarray}
q \alpha_{j, r} \Gamma_{ \ell, u^{(\ell)}_{j, r}  }  ( \mathbf{x}_1, \ldots , \mathbf{x}_{\ell-1}, \mathbf{e}_i  )  \in \mathbb{Z}
\end{eqnarray}
for
each $\ell <  j \leq d, 1 \leq r \leq r_{j}$.
Thus we see that (\ref{(15.2) in S'}) implies
\begin{equation}
\label{(15.2) in S}
\| \sum_{r=1}^{r_{\ell}} q \alpha_{\ell, r} \Gamma_{\ell,  U_{\ell, r}}  ( \mathbf{x}_1, \ldots , \mathbf{x}_{\ell-1}, \mathbf{e}_i  ) \|
<  P^{-\ell + 4 \theta + (\ell-1) \eta } \ \ (i=1,\ldots, n).
\end{equation}

Given $\mathbf{x}_1$, $\ldots$, $\mathbf{x}_{\ell-1}$ as above, we form a matrix
$$
[m_{ri}]_{\mathbf{x}_1, \ldots , \mathbf{x}_{\ell-1}},
$$
where its entries are
$$
m_{r i}= \Gamma_{ \ell, U_{\ell,r} } ( \mathbf{x}_1, \ldots, \mathbf{x}_{\ell-1}, \mathbf{e}_i  )  \  \  \  (1 \leq r \leq r_{\ell}, 1 \leq i \leq n).
$$
Now if this matrix $[m_{r i}]_{\mathbf{x}_1, \ldots , \mathbf{x}_{\ell-1}}$ has rank strictly less than $r_{\ell}$ for each of the $(\ell-1)$-tuples counted by $N(\eta)$, then by the definition of
$z_{R_0}(\mathbb{M}_{\ell})$ we have
$$
z_{R_0}(\mathbb{M}_{\ell}) \geq N(\eta) \gg R_0^{n(\ell-1) - 2^{\ell-1} (Q/ \eta) - \varepsilon },
$$
where the implicit constant in $\gg$ depends only on $n,d, \eta,$ and $\varepsilon$.
Thus we have the alternative $(iii)$ in this case. Hence, we may suppose that at least one of these matrices, which we denote by $[m_{ri}]$, has rank $r_{\ell}$.
Without loss of generality, suppose the submatrix $M_0$ formed by taking the first $r_{\ell}$ columns of $[m_{ri}]$
has rank $r_{\ell}$. 

It follows from the definition of $\Gamma_{\ell, U_{\ell,r} }$ that every monomial occurring in
$\Gamma_{\ell, U_{\ell,r} } (\mathbf{z}_1, \ldots, \mathbf{z}_{\ell})$ has some component of $\mathbf{z}_i = (z_{i,1}, \ldots, z_{i,n})$ as a factor for each $1 \leq i \leq \ell$ \cite[Proof of Lemma 11.2]{S}. Recall we also have
$$
| \Gamma_{\ell, U_{\ell,r} } | \leq 2^{\ell} \ell^{\ell} |U_{\ell, r}|
$$
from (\ref{ineq coeff of Gam}).
Therefore, we have
$$
m_{ri}= \Gamma_{ \ell, U_{\ell,r} } ( \mathbf{x}_1, \ldots, \mathbf{x}_{\ell-1}, \mathbf{e}_i  ) \ll R_0^{\ell -1},
$$
and also
$$
n_0 := \det (M_0) \ll R_0^{r_{\ell}(\ell-1)} = P^{r_{\ell}(\ell-1)\eta},
$$
where the implicit constants in $\ll$ depend only on $n$, $\ell,$ $r_{\ell}$, and $\mathbf{U}_{\ell}$.
Hence, from ~(\ref{(15.2) in S}) we may write
$$
q\sum_{r=1}^{r_{\ell}} \alpha_{\ell,r} m_{ri} = c_i + \beta'_i \ \ (1 \leq i \leq n),
$$
where $c_i$ are integers and $\beta'_i$ are real numbers satisfying
$$
| \beta'_i | < P^{-\ell + 4 \theta + (\ell-1) \eta } \ \ (1 \leq i \leq n).
$$
Let $v_1, \ldots, v_{r_{\ell}}$ be the solution to the system of linear equations
\begin{equation}
\label{(15.3) in S}
\sum_{r=1}^{r_{\ell} } v_r m_{ri} = n_0 c_i \ (1 \leq i \leq r_{\ell}).
\end{equation}
Then we have
\begin{equation}
\label{(15.4) in S}
\sum_{r=1}^{r_{\ell}} ( q n_0 \alpha_{\ell,r} - v_r  )m_{r i} = n_0 \beta'_i \ (1 \leq i \leq r_{\ell}).
\end{equation}
By applying Cram\'{e}r's rule to ~(\ref{(15.3) in S}), it follows that  $v_r \in \mathbb{Z}$ $(1 \leq r \leq r_{\ell})$. Also by applying Cram\'{e}r's rule to ~(\ref{(15.4) in S}),
we obtain
\begin{eqnarray}
\| q n_0 \alpha_{\ell,r} \| \leq | q n_0 \alpha_{\ell,r} - v_r | \ll R_0^{ (\ell-1)(r_{\ell} - 1) }  P^{-\ell + 4 \theta+  (\ell-1) \eta} = P^{ - \ell + 4 \theta +  r_{\ell} (\ell-1)  \eta } ,
\end{eqnarray}
where the implicit constant in $\ll$ depends only on $n$, $\ell$, $r_{\ell}$, and $\mathbf{U}_{\ell}$.
This completes the proof of Lemma \ref{Lemma 15.1 in S''}.
\end{proof}

We then have the following corollary.
\begin{cor}\cite[pp.276, Corollary]{S}
\label{cor 15.1 in S''}
Suppose $\mathbf{u}$ has coefficients in $\mathbb{Q}$, and further suppose $\mathbf{U}$ has coefficients in $\mathbb{Z}$.
Let $Q > 0$ and $\varepsilon >0$. Let $2 \leq \ell \leq d$ with $r_{\ell} > 0$.
If $\ell =d$, then let $\theta = 0$ and $q=1$.
On the other hand, if $2 \leq \ell < d$, then suppose $0 \leq \theta < 1/4$ and that there is $q \in \mathbb{N}$ with
$$
q \leq P^{\theta}, \ \ \    q \boldsymbol{\alpha}_{j} \in \mathbb{Z}^{r_j} \ \   (\ell < j \leq d),
$$
and
$$
q \alpha_{\ell', r} u^{(j)}_{\ell', r}(\mathbf{x}) \in \mathbb{Z}[x_1, \ldots, x_n]
$$
for every $\ell < \ell' \leq d, 0 \leq j < \ell', 1 \leq r \leq r_{\ell'}$.

Let $S( \boldsymbol{\alpha})$ be the sum associated to $\mathbf{u}$ as in ~(\ref{def of S 1}). 
Suppose
$$
4 \theta + Q \gamma'_{\ell} < 1.
$$
Then one of the following two alternatives must hold:

$(i)$ $|  S( \boldsymbol{\alpha})  | \leq P^{n-Q}$.

$(ii)$ There exists $n_0 \in \mathbb{N}$ such that
$$
n_0 \ll P^{ Q \gamma_{\ell} + \varepsilon } \text{  and  } \|  n_0 q \boldsymbol{\alpha}_{\ell} \| \ll P^{ -\ell + 4 \theta + Q \gamma_{\ell} + \varepsilon}.
$$
\newline
The implicit constants
depend at most on $n,d, r_d, \ldots, r_1, Q, \varepsilon$, and $\mathbf{U}$.
\end{cor}

\begin{proof}
The proof is similar to that of \cite[pp.276, Corollary]{S}.
If we have
$$
2^{\ell - 1} Q/\eta < g_{\ell}(\mathbf{U}_{\ell}),
$$
then it is clear that the alternative $(iii)$ of Lemma \ref{Lemma 15.1 in S''} can not occur
for $P$ sufficiently large with respect to $n,d, r_d, \ldots, r_1, \eta, \varepsilon$, and  $\mathbf{U}$.
In particular, this is the case with $\eta = Q \gamma'_{\ell} + \varepsilon'$ where $\varepsilon' > 0$ is sufficiently small.
Note we also have
$$
\eta + 4 \theta < 1,
$$
given $4 \theta + Q \gamma'_{\ell} < 1$. 
\end{proof}

\begin{thebibliography}{9}

\bibitem{B}  B. J. Birch, `{Forms in many variables}', \textit{Proc. Roy. Soc. Ser. A} 265 1961/1962, 245--263.

\bibitem{BGS} J. Bourgain, A. Gamburd and P. Sarnak, `{Affine linear sieve, expanders, and sum-product}', \textit{Invent. Math.} 179 (2010), no. 3, 559--644.

\bibitem{BHB} T.D. Browning, and D.R. Heath-Brown, `{Froms in many variables and differing degrees}',
\textit{J. Eur. Math. Soc.}, to appear.


\bibitem{BDLW} J. Br\"{u}dern, R. Dietmann, J. Liu and T. D. Wooley,
`{A Birch-Goldbach theorem}', \textit{Arch. Math.  (Basel)} 94 (2010), no. 1, 53--58.

\bibitem{CM} B. Cook and {\'A}.  Magyar, `{Diophantine equations in the primes}',  \textit{Invent. Math.} {198} (2014), 701--737.

\bibitem{C} S. Chow, `{Roth-Waring-Goldbach}', arXiv:1602.04012.

\bibitem{D} H. Davenport, \textit{Analytic methods for Diophantine equations and Diopantine inequalities}. Second edition.
Cambridge University Press, Cambridge, 2005.

\bibitem{DRS} W. Duke, Z. Rudnick and P. Sarnak, `{Density of integer points on affine homogeneous varieties}', \textit{Duke Math. J.} 71 (1993), no. 1, 143--179.

\bibitem{GPY}  D. A. Goldston,  J. Pintz and  C. Y. Y{\i}ld{\i}r{\i}m, `{Primes in tuples. I}',  \textit{Ann. of Math.} (2) 170 (2009), no. 2, 819--862.


\bibitem{GT1} B. Green and  T. Tao, `{The primes contain arbitrarily long arithmetic progressions}',
\textit{Ann. of Math.} (2) 167 (2008), no. 2, 481--547.

\bibitem{GT} B. Green and  T. Tao, `{Linear equations in primes}',  \textit{Ann. of Math.} (2) 171 (2010), no. 3, 1753--1850.


\bibitem{H1} H. A. Helfgott, `{Major arcs for Goldbach's problem}', arXiv:1305.2897.

\bibitem{H2} H. A. Helfgott, `{Minor arcs for Goldbach's problem}', arXiv:1205.5252.


\bibitem{H}  L. K. Hua,  \textit{Additive theory of prime numbers}. Translations of Mathematical Monographs, Vol. 13 American Mathematical Society,
Providence, R.I. (1965).

\bibitem{KW} A. V. Kumchev and T. D. Wooley, `{On the Waring-Goldbach problem for eighth and higher powers}',
\textit{J. London Math. Soc.} (2) 93 (2016), no. 3, 811--824.


\bibitem{L} J. Liu, `{Integral points on quadrics with prime coordinates}',
\textit{Monatsh. Math.} 164 (2011), no. 4, 439--465.

\bibitem{LS}  J. Liu and P. Sarnak, `{Integral points on quadrics in three variables whose coordinates have few prime factors}',
\textit{Israel J. of math.} 178 (2010), 393--426.


\bibitem{M}  J. Maynard, `{Small gaps between primes}', \textit{Ann. of Math.} (2) 181 (2015), no. 1, 383--413.


\bibitem{S} W.M. Schmidt, `{The density of integer points on homogeneous varieties}',
\textit{Acta Math.} {154} (1985), no. 3-4, 243--296.

\bibitem{V} I. M. Vinogradov. `{Representation of an odd number as a sum of three primes}', \textit{Dokl. Akad. Nauk. SSSR} 15 (1937), 291--294. 

\bibitem{XY} S. Y. Xiao and S. Yamagishi, `{Zeroes of polynomials in many variables with prime inputs}', arXiv:1512.01258.

\bibitem{Y} S. Yamagishi, `{An exponential sum estimate for systems with linear polynomials}',  arXiv:1607.08283.

\bibitem{Z}  Y. Zhang, `{Bounded gaps between primes}', \textit{Ann. of Math.} (2) 179 (2014), no. 3, 1121--1174.

\end{thebibliography}
\end{document}